\documentclass[12pt,leqno]{amsart}

\usepackage{amssymb}
\usepackage{amsthm}
\usepackage{amsmath}
\usepackage{amscd}
\usepackage[mathscr]{eucal}
\usepackage{verbatim}
\usepackage[all]{xy}

\numberwithin{equation}{section}

\theoremstyle{plain}
\newtheorem{theorem}{Theorem}[section]
\newtheorem{corollary}[theorem]{Corollary}
\newtheorem{lemma}[theorem]{Lemma}
\newtheorem{proposition}[theorem]{Proposition}

\theoremstyle{definition}
\newtheorem{definition}[theorem]{Definition}
\newtheorem{remark}[theorem]{Remark}
\newtheorem{example}[theorem]{Example}

\theoremstyle{remark}

\newcommand{\R}{\mathbb{R}}
\newcommand{\Q}{\mathbb{Q}}
\newcommand{\Z}{\mathbb{Z}}

\newcommand{\C}{\mathbb{C}}
\newcommand{\h}{\mathbb{H}}
\renewcommand{\H}{\mathbb{H}}

\renewcommand{\P}{\mathbb{P}}

\newcommand{\G}{\Gamma}
\newcommand{\g}{\gamma}

\newcommand{\la}{\lambda}

\def\t{\tau}

\newcommand{\back}{\backslash}

\newcommand{\V}{V(\Q)}

\newcommand{\zxz}[4]{\begin{pmatrix} #1 & #2 \\ #3 & #4 \end{pmatrix}}

\newcommand{\kzxz}[4]{\left(\begin{smallmatrix} #1 & #2 \\ #3 & #4\end{smallmatrix}\right) }

\newcommand{\re}{\operatorname{Re}}

\newcommand{\calE}{\mathcal{E}}
\newcommand{\calF}{\mathcal{F}}

\newcommand{\calI}{\mathcal{I}}

\newcommand{\calM}{\mathcal{M}}

\newcommand{\calQ}{\mathcal{Q}}

\newcommand{\frake}{\mathfrak e}

\newcommand{\eps}{\varepsilon}
\newcommand{\bs}{\backslash}

\newcommand{\CT}{\operatorname{CT}}
\newcommand{\vol}{\operatorname{vol}}
\newcommand{\tr}{\operatorname{tr}}

\newcommand{\sgn}{\operatorname{sgn}}
\newcommand{\erf}{\operatorname{erf}}

\newcommand{\Sl}{\operatorname{SL}}

\newcommand{\SL}{\operatorname{SL}}

\newcommand{\Spin}{\operatorname{Spin}}
\newcommand{\PSL}{\operatorname{PSL}}
\newcommand{\Mat}{\operatorname{Mat}}

\newcommand{\Mp}{\operatorname{Mp}}
\newcommand{\Orth}{\operatorname{O}}

\newcommand{\SO}{\operatorname{SO}}

\newcommand{\Ei}{\operatorname{Ei}}
\newcommand{\Ein}{\operatorname{Ein}}
\newcommand{\erfc}{\operatorname{erfc}}
\newcommand{\Iso}{\operatorname{Iso}}

\newcommand{\length}{\operatorname{length}}

\newcommand{\Stab}{\operatorname{Stab}}

\begin{document}

\title[Regularized theta liftings and periods of modular functions]{Regularized theta liftings and periods \\of modular functions}

\date{\today}

\author[Jan H.~Bruinier, Jens Funke, \"Ozlem Imamoglu]{Jan Hendrik Bruinier, Jens Funke, \"Ozlem Imamoglu}
\thanks{J. Bruinier is partially supported by DFG grant BR-2163/2-2.
J. Funke is partially supported by NSF grant   DMS-0710228.
\"O Imamoglu is partially supported by SNF grant  200021-132514.}
%\thanks{\"O Imamoglu is partially supported by SNF grant  200021-132514}
\address{Fachbereich Mathematik,
Technische Universit\"at Darmstadt, Schlossgartenstrasse 7, 64289
Darmstadt, Germany}
\email{bruinier@mathematik.tu-darmstadt.de}

\address{Department of Mathematical Sciences, Durham University, South Road,
Durham, DH1 3LE,
United Kingdom}
\email{jens.funke@durham.ac.uk}

\address{
Departement Mathematik,
ETH Z\"urich,
RŠmistrasse 101,
8092 Z\"urich,
Switzerland}
\email{ozlem@math.ethz.ch}

\begin{abstract}
In this paper, we use regularized theta liftings to construct weak Maass forms weight 1/2 as lifts of weak
Maass forms of weight 0. As a special case we give a new proof of some of recent results of Duke, Toth and Imamoglu
%in \cite{DIT}
on cycle integrals of the modular $j$ invariant and extend these to
any congruence subgroup. Moreover, our methods allow us to settle the
open question of a geometric interpretation for periods of $j$ along
infinite geodesics in the upper half plane. In particular, we give the
`central value' of the (non-existing) `$L$-function' for $j$. The key to
the proofs is the construction of some kind of a simultaneous Green
function for both the CM points and the geodesic cycles, which is
of independent interest.
\end{abstract}

\maketitle

\tableofcontents

\section{Introduction}

%Let  $f(z)$ be a modular function for the full modular group \Gamma(1)=PSL(2,\Z)$. It is well known that  the space of modular funtions whose poles are supported  at the cusp is a polynomial ring in one variable $\C[j]$

%Let $J(z)=j(z)-744 = e^{- 2\pi i z}  + 196884 e^{2\pi i z} + \dots $ be the classical Hauptmodul  for $\Gamma(1)=PSL(2,\Z)$,
%where $z$ is in the upper half plane $\H.$

\noindent{\it Generating series of traces of singular moduli}.
For a non-zero integer $d$, let $\calQ_{d}$ be the set of (positive definite if $d<0$) integral binary quadratic forms $Q=[a,b,c]$ of discriminant $d=b^2-4ac$. The natural action of $\Gamma(1)=\PSL_2(\Z)$ divides $\calQ_{d}$ into finitely many classes. For $d<0$ and $Q \in \calQ_{d}$, the root $z_Q=\frac{-b+\sqrt{d}}{2a}$ defines a CM point in the upper half plane $\H$. The values of the classical $j$-invariant at the CM points have been of classical interest. For $f \in M_0^!=\C[j]$, the space of weakly holomorphic functions of weight $0$ for $\Gamma(1)$, %Then it is well-known that $M_0^! has a unique basis  $\{j_m\}_{m\geq 0} $ of the form $j_m(z)=e^{-2\pi i mz} + O(e^{2\pi i z})$ ($z \in \h$). For example,
%\begin{equation} \label{jmd}
%j_0=1,\quad j_1=j-744, \quad j_2=j^2-1488j+159768, \dots
%\end{equation}
we define for $d<0$ the modular trace of $f$ by %the weighted sum
\begin{equation}\label{modtrace}
 \tr_d(f)  = \sum_{Q\in \G(1) \back \calQ_{d}} \frac{1}{|\Gamma(1)_Q|}
f(z_Q).
\end{equation}
Here $\Gamma(1)_Q$ denotes the finite stabilizer of $Q$. The theory of such traces has enjoyed renewed interest thanks to the work of Borcherds \cite{Bo1} and Zagier \cite{ZagierTr}, where connections between modular traces, automorphic infinite products and weakly holomorphic modular forms of half-integral weight are
established. In particular, a beautiful theorem of Zagier \cite{ZagierTr} shows for $j_1:= j-744$ that the generating series
\begin{equation}\label{Zagier1}
g_1(\tau): = -q^{-1} +2 + \sum_{d=1}^{\infty}  \tr_d(j_1)q^d= -q^{-1} +2 -248 q^3+492 q^4-4119q^7+ \cdots %7256q^8+\dots
\end{equation}
is a weakly holomorphic modular form of weight $3/2$ for the Hecke subgroup $\G_0(4)$. Here $q = e^{2\pi i\tau}$ with $\tau= u+iv \in \h$.

On the other hand, an older result of Zagier \cite{Zagier}
%-- using entirely different methods --
on the Hurwitz-Kronecker class numbers $H(|d|) =\tr_d(1)$, states that
\begin{equation}\label{Zagier2}
g_0(\tau):= -\frac1{12} +\sum_{d=1}^{\infty}  \tr_d(1) q^d + \frac{1}{16 \pi\sqrt{v}} \sum_{n=- \infty}^{\infty} \beta_{3/2}(4\pi n^2 v) q^{-n^2}
\end{equation}
is a  harmonic  weak Maass form of weight $3/2$ for $\G_0(4)$. Here $\beta_{k}(s) = \int_1^{\infty} e^{-st}  t^{-k} dt$. %Also note $-\tfrac1{12}= - \tfrac1{2\pi} \int_{\G(1) \back \h} \tfrac{dxdy}{y^2}$ is the volume of the underlying modular curve.

Using the methods of \cite{FCompo}, in \cite{BFCrelle} Bruinier and Funke unified and generalized \eqref{Zagier1} and \eqref{Zagier2} to traces of arbitrary weakly holomorphic modular functions $f$ of weight zero on modular curves $\G \back \h$ for any congruence subgroup $\G$. The results in \cite{BFCrelle} are obtained by considering a theta lift
\begin{equation}\label{thetaintro}
I_{3/2}(\tau,f) = \int_{\G \back \h} f(z) \cdot \Theta_{L}(\tau,z,\varphi_{KM}) \, d\mu(z)
\end{equation}
of $f$ against a theta series associated to an even lattice $L$ of signature $(1,2)$ and the Kudla-Millson Schwartz function $\varphi_{KM}$ of weight $3/2$ \cite{KMI}. Here $d\mu(z) =
 \tfrac{dx\,dy}{y^2}$. The integral converges, since the decay of the theta kernel turns out to be faster than the exponential growth of $f$. For $f=j_1$ and $f=1$ and the appropriate choice of the lattice $L$ one obtains the generating series \eqref{Zagier1} and \eqref{Zagier2} above, while in addition giving a geometric interpretation to the non-positive Fourier coefficients.
%Here the Schwartz function $\varphi_1$ underlying the theta kernel is closely related to a Green function for the CM points constructed by Kudla \cite{KAnn}.

 \medskip

\noindent{\it Cycle integrals of modular functions}.
In a different direction, turning to the natural question of the case of positive discriminants, Duke, Imamoglu and Toth \cite{DIT} recently studied the cycle integrals of modular functions as analogs of singular moduli. Their work gives an extention and generalization of the results of Borcherds and Zagier. For $d>0$, the two roots of $Q \in \calQ_{d}$ lie in $\mathbb{P}^1(\R)$, and we let $c_Q$ be the properly oriented geodesic in $\h$ connecting these roots. For non-square $d>0$, the stabilizer $\Gamma(1)_Q$ is infinitely cyclic, and we set $C_Q= \Gamma(1)_Q \back c_Q$. Then $C_Q$ defines a closed geodesic on the modular curve. For $f\in M_0^!$ and in analogy with (\ref{modtrace}) let
\begin{equation}\label{tr+}
 \tr_d(f)= \frac1{2\pi}\sum_{Q\in  \Gamma(1)\back \calQ_{d}} \int_{C_Q}  f(z)\frac{dz}{Q(z,1)}.
\end{equation}
One of the main results of \cite{DIT} realizes the generating series of traces of both the CM values and the cycle integrals for any $f \in M_0^!$ as a form of weight $1/2$. More precisely, for $f=j_1$ we have that
 \begin{equation}\label{DIT1}
h_1(\tau):=\sum_{d>0}\tr_d(j_1)\,q^d+2\sqrt{v} \beta^c_{\frac12}(-4 \pi v) q -8\sqrt{v}
 + 2\sqrt{v}  \sum_{d< 0} {\tr_d(j_1)}
\beta_{\frac12}(4\pi|d|v)q^{d}
 \end{equation}
is a  harmonic weak Maass form of weight 1/2 for $\G_0(4)$. Here $\beta^c_{1/2}(s) = \int_0^1 e^{-st}t^{-1/2} dt$ is the `complementary' function to $\beta_{1/2}(s)$. The analog of \eqref{Zagier2} for $f=1$ is that
\begin{equation}\label{DIT2}
  h_0(\tau):=\sum_{d>0}\tr_d(1)\,q^d+\frac{\sqrt{v}}{3}+2\sqrt{v} \sum_{ d < 0  }{\tr_d(1)}\beta_{\frac12}(4\pi|d|v)q^{d} + \sum_{n \neq 0}
  \alpha(4 n^2 v)q^{n^2}-\frac{1}{4\pi}\log v
  \end{equation}
is a weak Maass form of weight 1/2 for $\G_0(4)$. Here $\alpha(s)= \tfrac{\sqrt{s}}{4\pi} \int_0^{\infty} \log(1+t) e^{-\pi st} t^{-1/2} dt$. %Note that for $d>0$ a fundamental discriminant, we have $\tr_d(1)=h_d\log\epsilon_d$, where $h_d$ is the narrow class number and $\epsilon_d>1$ is the fundamental unit of $\Q(\sqrt{d})$.

Duke, Imamoglu and Toth prove their results by first constructing an explicit basis for the space of  harmonic weak Maass forms of weight 1/2 constructed out of Poincar\'e series. The construction of such a basis is quite delicate, due in part to the residual spectrum. Then (\ref{DIT1}) is proved by explicitly computing the cycle integrals of weight zero non-holomorphic Poincar\'e series in terms of exponential sums and then relating these to Kloosterman sums. The construction of $h_0(z)$ is similar using a Kronecker limit type formula for the weight $1/2$ Eisenstein series.

The functions obtained in \cite{Zagier} and \cite{ ZagierTr} and the ones in \cite{DIT} are related via the differential operator $\xi_{1/2} =  2iv^{1/2} \overline{\tfrac{\partial}{\partial \bar{\tau}}}$, which maps forms of weight $1/2$ to the dual weight $3/2$. We have
\begin{equation}\label{intro:xi-rel}
\xi_{1/2}( h_1 )= -2 g_1 \qquad \text{and} \qquad \xi_{1/2}( h_0 )= -2 g_0.
\end{equation}
In this way, the results of Duke, Imamoglu and Toth contain (for $\SL_2(\Z)$) the previous work on modular traces.

\medskip

\noindent{\it The coefficients of square index}.
For {\it square} discriminants $d$, no definition for the modular trace $\tr_d$ in \eqref{tr+} is given in \cite{DIT}, and hence the geometric meaning of the corresponding coefficients $\tr_d(j_1)$ and $\tr_d(1)$ in \eqref{DIT1} and \eqref{DIT2} is left open. Rather, for $d$ a square, these terms represent in \cite{DIT} only the unknown $d$-th Fourier coefficient of the residual Poincar\'e series which define the generating series. Analytically, the Fourier coefficients of square index $d$ are exactly where the weight $1/2$ Poincar\'e series have poles and residual terms have to be subtracted to obtain $h_1$ and $h_0$ (see \cite{DIT} Lemma 3, (2.26) and (2.27)). This makes them rather intractable to compute. Geometrically, the issue is that the stabilizer $\Gamma(1)_Q$ is trivial for square $d$ and hence the cycle $C_Q$ corresponds to an infinite geodesic in the modular curve. Therefore the integral of a modular function over $C_Q$ does not converge. This represents the principle obstacle to a geometric definition of the trace analogous to \eqref{tr+}. In fact, the authors of \cite{DIT} raise the questions whether their results can be approached using theta lifts as in \cite{BFCrelle} and whether one can give more insight to the mysterious nature of the square coefficients.

\medskip

In this paper, we indeed use the theta correspondence to study the traces and periods of modular functions. In particular, we succeed in computing the coefficients of square index in \eqref{DIT1} and \eqref{DIT2} and to give a geometric interpretation for those terms. In fact, we consider the modular traces and periods for any
 (harmonic) weak Maass form on any modular curve $\G \back \h$ defined by a congruence subgroup $\G$.

%Finally, we actually consider the generating series of modular traces for any weak harmonic Maass form of weight $0$.

\medskip

\noindent{\it The central $L$-value of the $j$-invariant}.
We first explain how to define the modular trace for $\G(1)$ for square index $d$. In view of \eqref{tr+} it suffices to regularize the period $\int_{C_Q} f(z)\frac{dz}{Q(z,1)}$ whenever $C_Q$ is an infinite geodesic. We will do this here only when $Q=[0,\sqrt{d},0]$ such that $Q(z,1) = \sqrt{d}z$ and $C_Q$ is (the image of) the imaginary axis. In fact, for $d=1$, we have $C_1=C_Q$. Hence the problem reduces to define the 'central value' of the (non-existing) $L$-function for $f$. Note that if there were  a cusp form $f(z) = \sum_{n>0} a(n) e^{2\pi in z}$ of weight $0$, then the cycle integral of $f$ over $C_Q$ would converge and equal the value of its $L$-function at $s=0$ which is given by
 \begin{equation}\label{cuspformperiod }
 \int_{C_Q}f(z)\tfrac{dz}{z}=2\int_1^\infty f(iy)\tfrac{dy}{y}=2\sum_{n\neq 0}a(n)\int_{2\pi n}^{\infty} e^{-t}\tfrac{dt}{t}.
 \end{equation}
%{\bf the formula is not correct for a higher weight cusp form}. Note that $\int_{2\pi n}^{\infty} e^{-t}\tfrac{dt}{t} = E_1(2\pi n)$ is the exponential integral as defined in \cite{AbSt} , 5.1.1.
In analogy to this, for $f\in M_0^!$, we define
\begin{equation} \label{d=1}
 \int_{C_Q}^{reg} f (z)\tfrac{dz}{z} :=  2\sum_{n\neq 0} a(n)  \calE\calI (2\pi n),
 \end{equation}
where $\calE\calI (w)$ is related to the exponential integrals defined in \cite{AbSt}, 5.1.1/2 by
\[
\calE\calI (w) := \int_{w}^{\infty} e^{-t} \frac{dt}{t} =
\begin{cases}
E_1(w) & \text{if $w >0$ } \\
-\Ei(-w) & \text{if $w <0 $}.
\end{cases}
 \]
Here in the second case the integral is defined as the Cauchy principal value.

A more geometric characterization for the regularized period is

\begin{theorem}\label{Intro:spec}
Let $f \in M^!_0$ be a modular function with vanishing constant coefficient. Then for $C_Q$ the imaginary axis as above we have for any $T>0$ that
\begin{align*}
\int^{reg}_{C_Q}  f(z) \frac{dz}{z} = 2\int_{i}^T f(z) \frac{dz}{z} - \int^{iT+1}_{iT}
f(z) \left(\psi(z) + \psi(1-z) \right)dz.
\end{align*}
Here $\psi(z)=\frac{\Gamma'(z)}{\Gamma(z)}$ is the Euler Digamma function.
\end{theorem}
This formula is strikingly similar to the one in \cite{FMspec}, Lemma 4.3 and Theorem 4.4, where the critical $L$-values of a modular form (not necessarily cuspidal) are interpreted as cohomological periods of holomorphic $1$-forms with values in a local system over certain closed `spectacle' cycles. We will discuss an analogous cohomological interpretation of Theorem~\ref{Intro:spec} elsewhere.

For $T=1$ and using that $j_1$ has real Fourier coefficients we obtain the following beautiful formula for the regularized integral of $j_1$ along the imaginary axis. We have
 \begin{equation}\label{jLvalue}
 \int_{C_Q}^{reg} j_1(z)\frac{dz}{z}= -2\Re\left(\int_i^{i+1} j_1(z)\psi(z) dz \right).
 \end{equation}
In fact, this formula was first suggested to us by D. Zagier, who obtained \eqref{jLvalue} based on heuristic arguments and verified numerically that defining $\tr_1(j_1)$ using \eqref{jLvalue} gives the correct value for $\tr_1(j_1)$ in \eqref{DIT1}.

\medskip

\noindent{\it The main result}.
Before we describe the theta lift we employ, we first state our main result in a special case. Let $p$ be a prime (or $p=1$). We consider the set $\calQ_{d,p}$ of quadratic forms $[a,b,c]\in \calQ_{d}$ such that $a\equiv 0\pmod{p}$. The group $\G_0^{\ast}(p)$, the extension of the Hecke group $\Gamma_0(p)\subset\Gamma(1)$ with the Fricke involution $W_p=\kzxz{0}{-1}{p}{0}$, acts on $\calQ_{d,p}$ with finitely many orbits. Let $f \in M^!_0(\Gamma^{\ast}_0(p))$ be a weakly holomorphic modular function of weight $0$ for $\Gamma^{\ast}_0(p)$.
%and denote its Fourier expansion by
%\begin{align}\label{four1}
%$f(z)=\sum_{n\gg-\infty}a(n)e(nz)$.
%\end{align}
We define the modular trace of $f$ of index $d\ne 0$ by
\begin{equation}\label{modtracecm}
\tr_d(f) =
\begin{cases}
\sum_{Q\in \Gamma^{\ast}_0(p) \back \calQ_{d,p}} \frac{1}{|\Gamma^{\ast}_0(p)_Q|}
f(\alpha_Q) & \text{if $d<0$}, \\
\frac{1}{2\pi} \sum_{Q\in \Gamma^{\ast}_0(p)\back \calQ_{d,p}}  \int^{reg}_{\Gamma^{\ast}_0(p)_Q \back c_Q} f(z)\tfrac{dz}{Q(z,1)} & \text{if $d>0$}.
\end{cases}
\end{equation}
%Here $\Gamma^{\ast}_0(p)_Q$ is the stabilizer of $Q$ in $\Gamma^{\ast}_0(p)$.

 \begin{theorem}\label{th:intro}
Let $f(z)=\sum_{n\gg-\infty}a(n)e(nz)  \in M^!_0(\Gamma^{\ast}_0(p))$ with $a(0)=0$. Then
 \begin{multline*}\label{H-intro}
 H(\tau,f) :=
\sum_{d>0}\tr_d(f)\,q^d+2\sqrt{v}\sum_{m>0}\sum_{n<0}  a(mn) \beta^c_{\frac12}(-4 \pi m^2v) q^{m^2} -2\sqrt{v} \tr_{0}(f) \\ + 2\sqrt{v}  \sum_{d< 0} {\tr_d(f)}
\beta_{\frac12}(4\pi|d|v)q^{d}
%\, + \, a(0)\left[\sum_{n \neq 0}\alpha(4 n^2 v)q^{n^2}-\frac{1}{4\pi}\log v\right]
\end{multline*}
is a  harmonic weak Maass form of weight $1/2$ for $\Gamma_0(4p)$. Here
\[
\tr_{0}(f) = -\frac{1}{2\pi}  \int_{\Gamma^{\ast}_0(p) \back \h}^{reg} f(z) d\mu(z) = 4   \sum_{n > 0} a(-n)\sigma_1(n).
\]
is the suitably regularized average value of $f$ on $\Gamma^{\ast}_0(p) \back \h$.
\end{theorem}

%\mathcal{F}(2 \sqrt{\pi v N}m)

%It is also possible to give the result for functions which have non-zero constant coefficients (see Theorem  \ref{th:Main-Maass}). For the constant function $f=1$ the result takes the shape very similar to \eqref{DIT2}.
%\begin{theorem}\label{th:intro-con}
% The function given by the generating series
%\begin{align} \label{H0}
%H_0(\tau,f)=
%&\sqrt{v} \frac{\pi}{3} + \text{isotropic vectors} \\ \nonumber
%&+ \pi \sum_{d<0}\tr_{d }(1) \beta(4|d|v)\frac{e(d{\tau})}{2\sqrt{d}}
% + \sum_{d>0}   \tr_{d} (1) \frac{e(d{\tau})}{2\sqrt{d}}\\ \nonumber
%&+
%\sum_{m\neq 0} \alpha(4m^2v)
 %\ e(m^2\tau).
%\end{align}
%has weight weight $1/2$ for $\Gamma_0(p)$.
%\end{theorem}
For $p=1$ and $f=j_1$ we recover $h_1(\tau)$. However, now with explicit geometric formulas for the square coefficients in the generating series. For $f=1$ we have a similar theorem which generalizes $h_0$.
The statements for any congruence subgroup are formulated in terms of quadratic spaces of signature (2,1).

\medskip

\noindent{\it The regularized theta lift}.
To prove our results we study the theta integral
\begin{equation}\label{theta0}
I_{1/2}(\tau,f) = \int_{\G \back \h} f(z) \cdot \Theta_{L}(\tau,z,\varphi_0) \, d\mu(z).
\end{equation}
Here the kernel function $ \Theta_{L}(\tau,z,\varphi_0)$ is the Siegel theta series associated to the standard Gaussian $\varphi_0$ of weight $1/2$ for a rational quadratic space of signature $(2,1)$. It is related to the kernel of \eqref{thetaintro} via
\begin{equation}\label{intro:xi-0KM}
\xi_{1/2}( \Theta_{L}(\tau,z,\varphi_0))=  -\frac1{\pi} \Theta_{L}(\tau,z,\varphi_{KM}),
\end{equation}
and we obtain formally the same relation for the theta lifts $I_{1/2}$ and $I_{3/2}$, which matches (up to a constant) the relations given in \eqref{intro:xi-rel}. When $f$ is a Maass cusp form the integral (\ref{theta0}) converges, and this lift has been previously studied by Maass \cite{Maass}, Duke \cite{Duke}, and Katok and Sarnak \cite{KS} among others. However, the theta kernel $\Theta_{L}(\tau,z,\varphi_0)$ in contrast to the one used for $I_{3/2}$ in \cite{BFCrelle} is now moderately {\it increasing}. Hence when the input function $f$ is not a cusp form, the integral does not converge (even for $f=1$) and has to be regularized.

We analyze in detail two different approaches
%to the problem
to regularize \eqref{theta0} for any weak Maass form $f(z)$ of weight $0$ for $\G$ with eigenvalue $\la$ under the Laplace operator $\Delta_z$. The case $\la=0$ is the most interesting, which we now describe. First, following an idea of Borcherds \cite{Bo1} and Harvey-Moore \cite{HM}, we regularize the integral by integrating over a truncated fundamental domain $\calF_T$ for $\G \back \h$ and taking a limit. More precisely, for a complex variable $s$ we consider
\begin{equation}\label{intro:reg1}
 \lim_{T \to \infty} \int_{\calF_T} f(z) \cdot \Theta_{L}(\tau,z,\varphi_0) \,y^{-s} d\mu(z),
\end{equation}
where $\calF_T$ is a suitable truncated fundamental domain for $\Gamma$.
For the real part of $s$ sufficiently large, the limit converges and admits a meromorphic continuation to the whole $s$-plane. Then we regularize \eqref{theta0} by taking the constant term in the Laurent expansion of \eqref{intro:reg1} at $s=0$.

The second approach uses differential operators in the spirit of the regularized Siegel-Weil formula of Kudla and Rallis \cite{KR}. Using Eisenstein and Poincar\'e series of weight $0$ one can construct a `spectral deformation' of $f$, that is, a family of functions $f_s(z)$ such that $ \Delta_z f_s = s(1-s)f$ and $f_1=f$, and then we consider
\begin{equation}\label{intro:reg2}
\frac{1}{s(1-s)}  \int_{\G \back \h} f_s(z) \cdot \Delta_z \Theta_{L}(\tau,z,\varphi_0) \,d\mu(z).
\end{equation}
The point is that $ \Delta_z \Theta_{L}(\tau,z,\varphi_0)$ is of very rapid decay (like the kernel for \eqref{thetaintro}) and hence the integral converges. Furthermore, by the adjointness of the Laplace operator we see that \eqref{intro:reg2} formally equals \eqref{theta0}. Then we can regularize \eqref{theta0} to be the constant term in the Laurent expansion of \eqref{intro:reg2} at $s=1$. We show

\begin{theorem}
Let $f$ be a  harmonic weak Maass form of weight $0$ for a congruence subgroup $\G$. If the constant terms of the Fourier expansion of $f$ at all cusps of $\G$ vanish, then the two regularizations of \eqref{theta0} coincide. Otherwise, they differ by an explicit linear combination of holomorphic unary Jacobi theta series of weight $1/2$.
\end{theorem}

%In view of this result it would be interesting to work out how one can also regularize the Borcherds lift (which goes from $\SL_2$ to the orthogonal group) via spectral deformations and differential operators.

By lifting Poincar\'e series of weight $0$, we are also able to realize the Poincar\'e series of weight $1/2$ which occur in \cite{DIT} as theta lifts, explicitly relating our approach to the one of Duke, Imamoglu, and Toth.

\medskip

\noindent{\it The Green function $\eta$ and the Fourier expansion of the theta lift}.
The key to compute the Fourier coefficients of the lift $I_{1/2}(\tau,f)$, is the construction of a Green function $\eta$ for the Schwarz function $\varphi_0$. We view the set of all rational quadratic forms $Q=[a,b,c]$ together with the discriminant form $d=b^2-4ac$ as a quadratic space $\mathcal{Q}$ of signature $(2,1)$ whose associated symmetric space is equivalent to $\h$. In this way, $\varphi_0$ can be regarded as a function on $\mathcal{Q} \times \h$. We explicitly construct (some kind of) a Green function $\eta(Q,z)$ for $\varphi_0(Q,z)$ for {\it all} $Q$ of non-zero discriminant. More precisely, we have $\Delta_z \eta(Q,z) = - \tfrac1{4\pi}  \varphi_0(Q,z)$ outside the singularities of $\eta$. If $d<0$, then $\eta(Q)$ has a logarithmic singularity at the point $z_Q$, while for $d>0$, the function $\eta(Q)$ is differentiable, but not $C^1$, and the discontinuity of $\partial\eta$ exactly occurs at the geodesic cycle $c_Q$. We show

 \begin{theorem}
Let $Q$ be a integral binary quadratic form with discriminant $d \ne 0$, not a square, with stabilizer $\G_Q$ in $\G$. Then for $f$ a weak Maass form of weight $0$ with eigenvalue $\la$, the integral $ \int_{\G_{Q} \back \h} f(z) \varphi_0(Q,z) d\mu(z)$ converges, and we have
\begin{multline*}
 \int_{\G_{Q} \back \h} f(z) \varphi_0(Q,z) d\mu(z) = - \frac1{4\pi}  \la  \int_{\G_{Q} \back \h} f(z) \eta(Q,z) d\mu(z)  \\ +
 \begin{cases} \left( \frac1{|\G_Q|} f(z_Q) \right) 2\beta_{\frac12}(4\pi|d|)e^{-2\pi d} & \text{if $d <0$} \\
 \left( \int_{C_Q}  f(z)\frac{dz}{Q(z,1)} \right) e^{-2\pi d} & \text{if $d >0$}.
\end{cases}
 \end{multline*}
\end{theorem}

This essentially computes the Fourier coefficients of non-square index for $I_{1/2}(\tau,f)$ (at least when $\la=0$).
For the other Fourier coefficients we also utilize $\eta$. The square coefficients however are also in our approach quite complicated and require some rather intricate considerations since in that case $\int_{\G_Q \back \h} f(z) \varphi_0(Q,z) d\mu(z)$ does not converge as $\G_Q$ is trivial.

The existence of such a Green function is rather surprising, and we believe is of independent interest. Moreover, $\eta$ refines Kudla's Green function $\xi$ for $\varphi_{KM}$ \cite{KAnn} which played a crucial role in studying \eqref{thetaintro} in \cite{BFCrelle}. However, $\xi$ only has singularities along the CM points and hence cannot detect the periods over the geodesic cycles. Note that $\xi$ plays an important role in the Kudla program (see eg. \cite{KMSRI}) which is concerned with the realization of generating series in arithmetic geometry as automorphic forms, in particular as the derivative of Eisenstein series. It is therefore an interesting question how the results of this paper and $\eta$ in particular fit into this framework.

\medskip

\noindent{\it Other Input}.
One can also study other input functions $f$ for the lift $I_{1/2}(\tau,f)$. One natural extension is to consider  a meromorphic function $f$ of weight $0$ with at most simple poles in $\h$. For example, for $d>0$ a non-square, taking the $d$-th coefficient (in $\tau$) of the lift of $f_w(z) = j'(w)/(j(z)-j(w))$ (with $w \in \h$), one obtains a non-holomorphic form of weight $2$ (in $w$). Using the techniques of this paper one can prove that this form is a `completion' of the holomorphic generating series
\[
F_d(w) = - \sum_{m\geq 0} \tr_d(j_m) e^{2\pi i mw},
\]
which in \cite{DIT}, Theorem 5, is shown to be a holomorphic modular integral of weight $2$ with a rational period function. Here $\{j_m\}$ denotes the unique basis of $M_0^!$ whose members are of the form $j_m(z) = e^{-2\pi i mz} + O(e^{2\pi iz})$.

Another interesting case is when $f = \log \| F \|$ is the logarithm of the Petersson metric of a meromorphic modular form $F$. For example, for $F(z)=\Delta(z)$, the discriminant function, one can show (similarly as in Theorem~1.2 in \cite{BFCrelle}) that $I_{1/2}(\tau,f)$ is equal to the constant term in the Laurent expansion at $s=1$ of the derivative of an Eisenstein series of weight $1/2$. In general, one can view the lift of such input as the adjoint of the
(additive) Borcherds lift, which uses the same kernel function $\Theta_L(\tau,z,\varphi_0)$.
We will consider these lifts in a different paper.

\medskip

The theta lift has been studied recently also by Matthes \cite{Matthes} using the second regularization via differential operators. More precisely, he considers the analogous lift for general hyperbolic $n$-space. In our case, he considers (mainly) input functions with non-zero eigenvalue under the Laplace operator and employs a different method to compute the coefficients of non-square index leaving the square coefficients open.

\medskip

We thank D. Zagier for his interest and encouragement and for sharing his formula \eqref{jLvalue} for the `central $L$-value' of $j_1$ with us. The first two authors thank the Forschungsinstitut f\"ur Mathematik at ETH Z\"urich for the generous support for this research throughout multiple visits in the last years.

%\tableofcontents

\section{The orthogonal group and vector valued modular forms for $\SL_2$}

\subsection{Modular curves associated to the orthogonal group $\SO(2,1)$}

Let $N$ be a positive integer.
Let $(V,Q)$ be the three-dimensional quadratic space over $\Q$
given by the trace zero $2\times 2 $ matrices
\begin{equation} \label{iso}
V  :=\left\{ X = \begin{pmatrix} x_1 & x_2 \\ x_3 & -x_1
 \end{pmatrix} \in \Mat_2(\Q) \right\},
\end{equation}
with the quadratic form $Q(X) = -N\det(X)$.  The corresponding bilinear
form is $(X,Y) = N\tr(XY)$, and its signature is $(2,1)$.  We let $G = \Spin(V)$, viewed
as an algebraic group over $\Q$, and write $\bar{G}$ for
its image in $\SO(V)$. We realize the associated Hermitean
symmetric space as
the Grassmannian of negative lines in $V(\R)$:
\[
D = \{ z \subset V(\R) ; \;\text{$\dim z =1$ and $Q|_z < 0$} \}.
\]
The group $\Sl_2(\Q)$ acts on $V$ by conjugation
\[
g.X := gXg^{-1}
\]
for $X\in V$ and $g\in \Sl_2(\Q)$, which gives rise to isomorphisms
$G\simeq\SL_2$ and $\bar G\simeq\operatorname{PSL_2}$.

We identify $D$ with the complex upper half plane $\H$ as follows, see \cite{KAnn}, section 11. Let $z_0\in {D}$ be the line spanned by
$\left( \begin{smallmatrix} 0 & 1 \\ -1 & 0 \end{smallmatrix}\right)$. Its stabilizer in ${G}(\R)$ is equal to  ${K} = \SO(2)$.
For $z= x+iy
\in \H$, we choose $g_z \in {G}(\R)$ such that $g_zi = z$ and put
\begin{equation}
X(z) :=  \frac1{\sqrt{N}}g_z.\zxz{0}{1}{-1}{0} = \frac1{\sqrt{N}y}
\begin{pmatrix} -x & z\bar{z} \\ -1 & x
\end{pmatrix}\in V(\R).
\end{equation}
We obtain the isomorphism $\H\to D$, $z \mapsto g_z z_0
%= \Span\left(g_z . \left( \begin{matrix} 0 & 1 \\ -1 & 0\end{matrix}\right)\right)
= \R X(z)$. We also define the quantity
\begin{equation}
R(X,z) = (X,X) + \frac12(X,X(z))^2,
\end{equation}
which is nonnegative, and vanishes exactly when $X \in \R X(z)$.

Let $L \subset \V$ be an even lattice of full rank and write $L'$
for the dual lattice of $L$. Let $\G$ be a congruence subgroup of
$\Spin(L)$ which takes $L$ to itself and acts trivially on the
discriminant group $L'/L$. We set $M = \G \back D$.

\begin{example}\label{nicelattice}
A particularly attractive lattice in $V$ is
\[
L=\left\{\zxz{b}{c/N}{a}{-b}; \, a,b,c\in \Z\right\}.
\]
The dual lattice is equal to
\[
L'=\left\{\zxz{b}{c/N}{a}{-b}; \,\text{$a,c\in \Z$, $b\in \frac{1}{2N}\Z$} \right\}.
\]
We have $L'/L\cong \Z/2N\Z$, the level of $L$ is $4N$, and we can take $\Gamma=\Gamma_0(N)$.
% takes $L$ to itself and acts trivially on $L'/L$.
\end{example}

\begin{comment}
\begin{example}
\label{DIT-case}
Another important lattice is
\[
L_1=\left\{\zxz{b}{2c}{-2a}{-b}; \, a,b,c\in \Z \right\}
\]
for $N=1$. Note that $\Stab(L_1) = \SL_2(\Z)$. For $X= \kzxz{b}{2c}{-2a}{-b}$ we have $Q(X) = b^2- 4ac$, the discriminant form for the binary quadratic form $[a,b,c]$. In fact, the map $X= \kzxz{b}{2c}{-2a}{-b} \mapsto [a,b,c]$ is $\SL_2(\Z)$-equivariant. Moreover, $L_1$ is closely related to the lattice $L$ in Example~\ref{nicelattice} for $N=1$, since $L_1=2L'$.
%
%Namely, we have
%\[
%L_1 = 2L \amalg   2(L+h),
%]
%where $h = \kzxz{1/2}{0}{0}{-1/2} \in L'$ is a representative of the unique non-trivial coset for $L'/L$.
\end{example}
\end{comment}

Since $V$ is isotropic, the modular curve $M$ is a non-compact Riemann surface.
The group $\Gamma$ acts on the set $\Iso(V)$ of isotropic lines in $V$.
The cusps of $M$ correspond to the $\G$-equivalence classes of $\Iso(V)$, with $\infty$ corresponding to the isotropic line $\ell_0$ spanned by $u_0= \left( \begin{smallmatrix}  0&1 \\0&0  \end{smallmatrix}\right)$.
\label{r:sigma}
For $\ell\in  \Iso(V)$, we pick $\sigma_{\ell} \in \SL_2(\Z)$ such that $\sigma_{\ell} \ell_0 = \ell$ and set $u_{\ell} = \sigma_{\ell}^{-1} u_0$.
%We orient all lines $\ell \in \Iso(V)$  by requiring that  $\sigma_{\ell} u_0$ is a positively oriented basis vector of $\ell$.
We let $\G_{\ell}$ be the stabilizer of $\ell$  in $\Gamma$. Then
\[
\sigma_{\ell}^{-1}\bar{\G}_{\ell} \sigma_{\ell} = \left\{
\begin{pmatrix} 1&k\alpha_{\ell} \\0&1 \end{pmatrix} ; \; k \in \Z
 \right\}
\]
for some $\alpha_{\ell} \in\Z_{>0}$, the width of the cusp $\ell$.
% Since $\sigma_{\ell} \in \SL_2(\Z)$, we see that $\alpha_{\ell}$ does not depend on the choice of $\sigma_{\ell} \in \SL_2(\Z)$.
There is also a $\beta_{\ell}\in\Q_{>0}$ such
that $\beta_{\ell} u_{\ell}$ is a primitive element of $\ell \cap L$. Finally, we write $\eps_{\ell} = \alpha_{\ell}/\beta_{\ell}$. Note that $\eps_{\ell}$ does not depend on the choice of $\sigma_{\ell}$ (even if we picked $\sigma_{\ell}$ in $\SL_2(\Q)$, see \cite[Definition~3.2]{FCompo}).

We compactify $M$ to a compact Riemann surface $\bar{M}$ in the usual
way by adding a point for each cusp $\ell\in \Gamma\bs \Iso(V)$. For every cusp
$\ell$ we choose sufficiently small neighborhoods $U_{\ell}$. We write
$q_{\ell} = e\left (\sigma^{-1}_{\ell} z/\alpha_{\ell}\right)$ with $z
\in U_{\ell}$ for the local variable (and for the chart) around $\ell
\in \bar{M}$. For $T>0$, we let $ U_{1/T} = \{w \in \C; \; |w| <
\frac{1}{2\pi T}\}$, and note that for $T$ sufficiently big, the
inverse images $q_{\ell}^{-1} (U_{1/T})$ are disjoint in $M$.  We
truncate $M$ by setting
\begin{equation*}\label{truncated}
M_T = \bar{M} \setminus \coprod_{[\ell] \in \Iso(V)} q_{\ell}^{-1} (U_{1/T}).
\end{equation*}

\subsection{Vector valued modular forms}

Here we recall some facts on vector valued modular forms and weak Maass forms for the Weil representations.
See e.g. \cite{Bo1}, \cite{Br2} for more details.

We let $\Mp_2(\R)$ be the two-fold cover of $\SL_2(\R)$
realized by the two choices of holomorphic square roots of $\tau
\mapsto j(g,\tau) = c\tau + d$, where $g = \left(
\begin{smallmatrix} a&b \\ c&d \end{smallmatrix} \right) \in
\SL_2(\R)$.
Let $\G'\subset \Mp_2(\R)$ be the inverse image of $\Sl_2(\Z)$ under the covering map.
We denote the standard basis of the group algebra $\C[L'/L]$  by $\{\frake_{h};\; h\in L'/L\}$.
Recall that there is a Weil representation $\rho_L$ of $\G'$  on the
group algebra $\C[L'/L]$, see \cite[Section 4]{Bo1} or \cite[Chapter 1.1]{Br2} for explicit formulas.
%
%(see \cite{Bo1}, \cite{Br2}).
%We denote the standard basis elements of $\C[L'/L]$ by
%$\frake_{h}$, where $h \in {L'/L}$.
%For the generators $S =
%\left( \left(
%\begin{smallmatrix} 0&-1\\1&0
%\end{smallmatrix} \right), \sqrt{\tau} \right)$,  and $T=\left( \left( \begin{smallmatrix} 1 &1 \\ 0&1 \end{smallmatrix} \right),1\right)$ of $\G'$,
%the action of $\rho_L$ is given by
%\begin{align*}
%\rho_L(T) \frake_{h} &= e(Q(h)) \frake_{h},\\
%\rho_L(S) \frake_{h} &= \frac{\sqrt{i}^{-1}}{\sqrt{|L'/L|}}
%\sum_{h'\in L'/L} e(- (h,h')) \frake_{h'}.
%\end{align*}
%We let $\tilde\Gamma_\infty:=\langle T\rangle\subset\tilde\Gamma$.

Let $\Gamma''\subset\Gamma'$ be a subgroup of finite index.
For $k\in \frac{1}{2}\Z$,
we let $A_{k,L}(\Gamma'')$ be the space of $C^{\infty}$ automorphic
forms of weight $k$ with respect to $\rho_L$ for $\Gamma''$. That is, $A_{k,L}(\Gamma'')$ consists of
those $C^{\infty}$-functions $f:\H \to \C[L'/L]$  that satisfy
\[
f(\g'\tau) = \phi^{2k}(\tau) \rho_L(\g',\phi) f(\tau)
\]
for $(\g',\phi) \in \G''$. Note that the components $f_h$ of $f$
define scalar valued $C^{\infty}$ modular forms of weight $k$ for the
subgroup $\G''\cap \Gamma'(\ell)$, where $\ell$ denotes the level of
the lattice $L$ and $\Gamma'(\ell)$ is the principal congruence
subgroup of level $\ell$.

Following \cite[Section 3]{BFDuke}, we call a function $f\in A_{k,L}(\Gamma'')$ a {\em weak Maass form} of weight
$k$ for $\Gamma''$
with representation $\rho_L$, if it is an eigenfunction of the hyperbolic Laplacian
\begin{align}
\label{defdelta}
\Delta_k = -v^2\left( \frac{\partial^2}{\partial u^2}+ \frac{\partial^2}{\partial v^2}\right) + ikv\left( \frac{\partial}{\partial u}+i \frac{\partial}{\partial v}\right),
\end{align}
and if it has at most linear exponential growth at the cusps of
$\Gamma''$. The latter condition means that there is a $C>0$ such that
for any cusp $s\in \P^1(\Q)$ of $\Gamma''$ and $(\delta,\phi)\in
\Gamma'$ with $\delta\infty=s$ the function $f_s(\tau) =
\phi(\tau)^{-2k} \rho_L^{-1}(\delta,\phi) f(\delta\tau)$ satisfies
$f_s(\tau)=O(e^{C v})$ as $v\to \infty$ (uniformly in $u$, where
$\tau=u+iv$).

The function $f$ is called a {\em harmonic weak Maass form} if it is a weak Maass form with eigenvalue $0$ under $\Delta_k$. We write $H_{k,L}(\Gamma'')$ for
the space of harmonic weak Maass forms of weight $k$ for $\Gamma''$
with representation $\rho_L$.

Recall that there is a differential operator $\xi_k=2iv^k\overline{\frac{\partial}{\partial\bar\tau }}$
%given by $\xi_K(f)=2iv^k\overline{\frac{\partial}{\partial\bar\tau }f}$
taking $H_{k,L}(\Gamma'')$ to the space of weakly holomorphic modular
forms of `dual' weight $2-k$ for $\Gamma''$ with the dual
representation of $\rho_L$.  We let $H_{k,L}^+(\Gamma'')$ be the
subspace of those $f\in H_{k,L}(\Gamma'')$ for which $\xi_k(f)$ is a
cusp form. Moreover, we let $M_{k,L}^!(\Gamma'')$ be the kernel of
$\xi_k$, that is, the space of weakly holomorphic modular forms for
$\Gamma''$. Summarizing we have the chain of inclusions
\[
M_{k,L}^!(\Gamma'')\subset H_{k,L}^+(\Gamma'')\subset H_{k,L}(\Gamma'') \subset A_{k,L}(\Gamma'').
\]
In the case where $\Gamma''=\Gamma'$ we will drop the $\Gamma''$ from the notation
and, for instance, simply write $M_{k,L}^!$. If the representation $\rho_L$ is trivial (that is, $L$ is unimodular) we drop the $L$ from the notation.

\begin{example}
\label{ex:iso}
We consider the lattice $L=\left\{\kzxz{b}{c/N}{a}{-b};\,a,b,c\in
    \Z\right\}$ of level $4N$ from Example~\ref{nicelattice}. Then given $g =
  \sum_{h \in L'/L} g_h \frake_h$ in $A_{k,L}$, the sum
\begin{equation}\label{+space-form}
\tilde g(\tau) = \sum_{h \in L'/L} g_h(4N\tau)
\end{equation}
gives a scalar-valued form of weight $k$ for $\G_0(4N)$ satisfying the plus condition, i.e., the $n$-th Fourier coefficient vanishes unless $n$ is a square modulo $4N$. In fact, if $N=p$ is a prime, and $k\in 2\Z+\frac{1}{2}$, this gives an isomorphism between $M^!_{k,L}$ and the space $M^{+,!}_{k}(p)$ of scalar valued weakly holomorphic forms for $\G_0(4p)$ in the Kohnen plus space (see e.g. \cite[Example 2.3]{Bo1} and \cite[\S5]{EZ}).
\end{example}

For an isotropic line $\ell$ in $V$, we define the space $W = W_{\ell}= \ell^{\perp}/\ell$ which is naturally a unary positive definite quadratic space with the quadratic form  $Q(\bar{X}) = Q(X)$. Then
\[
K_{\ell}= \left( L \cap \ell^{\perp}\right) / \left( L \cap\ell  \right)
\]
defines an even lattice in $W$.
Using \cite[Proposition 2.2]{Br2}, it is easy to see that the dual lattice is given by
\[
K_{\ell}'= \left( L' \cap \ell^{\perp}\right) / \left( L' \cap\ell  \right).
\]
%\texttt{Take out soon.}
%In fact, it is clear that the dual of $K_\ell$ contains $\left( L' \cap \ell^{\perp}\right) / \left( L' \cap\ell  \right)$. To prove the other inclusion, we let $x\in V\cap \ell^\perp$ such that the image $\bar x\in W$ is contained in $K_\ell'$.
%Then $(x,u)\in \Z$ for all $u\in L\cap \ell^\perp$. Using the notation of \cite[Proposition 2.2]{Br2}, we have that for any $y\in L$, the vector $\tilde y= y-(y,\frac{z}{N})\zeta$ lies in $L\cap \ell^\perp$. (Here $z$ is a primitive vector of $L\cap \ell$ and $N\in \N$ is defined by $(z,L)=N\Z$.) Hence $(x,y) -(x,\zeta)(y,\frac{z}{N})\in \Z$ for all $y\in L$.
%Consequently, the vector $\tilde x = x-(x,\zeta)\frac{z}{N}$ lies in $L'$.
%Since it is orthogonal to $\ell$, we find that $\tilde x\in L'\cap \ell^\perp$.
%Since $(\tilde x,u)=(x,u)$ for all $u\in L\cap \ell^\perp$, the claim follows.
%
We have the exact sequence
\begin{align}
\label{eq:cl2}
\xymatrix{ 0\ar[r]& L'\cap \ell/L\cap\ell^\perp \ar[r]& L'\cap \ell^\perp /L\cap\ell^\perp\ar[r]&  K_\ell'/K_\ell\ar[r] & 0 }.
\end{align}
%\texttt{Take out soon.}
%Let $z$ be a primitive vector of $L\cap \ell$ and $N\in \N$ be defined by $(z,L)=N\Z$. Then we can identify  $L'\cap \ell/L\cap\ell^\perp$ with the subgroup of order $N$  in $L'/L$ generated by $\frac{z}{N}$. Moreover, $L'\cap \ell^\perp /L\cap\ell^\perp$ can be identified with those $\gamma \in L'/L$ which are orthogonal to
%%$\frac{z}{N}$
%$L'\cap\ell$. This follows from the fact that
%$\{\gamma\in L';\; (\gamma,z)\equiv 0\pmod{N}\} = (L'\cap \ell^\perp)+L$.

The vector valued theta function
\begin{align}
\Theta_{K_\ell}(\tau)= \sum_{\lambda\in K_\ell'}e(Q(\lambda)\tau)\frake_{\lambda+K_\ell}
\end{align}
associated to $K_\ell$ defines a holomorphic modular form in $M_{1/2,K_\ell}$, whose components we denote by $\theta_{K_\ell,\bar h}(\tau)$ for $\bar h \in K_\ell'/K_\ell$.
Recall from \cite[Lemma 5.6]{Br2} (or more generally \cite[Theorem 4.1]{Sch}) that there is a map from vector valued modular forms for $\rho_{K_\ell}$ to vector valued modular forms for $\rho_L$.
Using it, we see that
\begin{align}
\label{eq:tildethetaK}
\tilde \Theta_{K_\ell}(\tau)=
%\sum_{\substack{h\in L'\cap \ell^\perp/L} }
\sum_{\substack{h\in L'/L\\ h\perp \ell} }
\theta_{K_\ell,\bar h}(\tau)\frake_{h}
\end{align}
defines a vector valued holomorphic modular form in $M_{1/2,L}$.
Here $\bar h $ denotes the image of $h$ under the map in  \eqref{eq:cl2}. We let $b_{\ell}(m,h)$ be the $(m,h)$-th Fourier coefficient of $\tilde\Theta_{K_\ell}(\tau)$. Note that for $m >0$ we have $b_{\ell}(m,h)=0$ unless $m/N$ is a square and there exists a vector $X \in L+h$ perpendicular to $\ell$ of length $Q(X)=m$. In that case we have $b_{\ell}(m,h)= 1$ if $h \not \equiv -h \mod L$ and $2$ otherwise.

%\[
%K_{\ell,h}= \left(\ell^{\perp} \cap (L+h) \right) / \left(\ell \cap (L+h) \right)
%\]
%(if not empty) is naturally a union of cosets of lattices in $W$. We define
%\[
%\theta_{K_{\ell},h}(\tau) :=  \sum_{x \in K_{\ell,h}} e\left( {\tau} q(x) \right),
% \]
%and $\Theta_{K_{\ell}}(\tau)$ as before. Of course $\theta_{K_{\ell},h}(\tau)$ has weight $1/2$, and, in fact, as a consequence of the next proposition we have

%\begin{lemma}
%\[
%\Theta_{K_{\ell}}(\tau) \in A_{\tfrac12,L}.
%\]
%\end{lemma}

\section{Cycles and traces}

In this section, we give define in our setting the cycles and traces. In particular, we explain in detail how to regularize the periods of weakly holomorphic functions over infinite geodescis.

\subsection{Heegner points}

Heegner points in $M$ are given as follows, see e.g. \cite{FCompo},
\cite{KS}, \cite{KAnn}. For $X \in V$ of negative length $Q(m)<0$, we put
\begin{equation}
D_X = \R X = \{z \in D; \, R(X,z)=0 \} \in D.
\end{equation}
Via \cite{KAnn}, (11.9) we see
\begin{equation}\label{R-elliptic}
R(X,z) = 2m \sinh^2( d(z,D_X)) = \frac{m}{2 \Im(D_x)^2 y^2} |z-D_X|^2|z-\overline{D_X}|^2.
\end{equation}
Here $d(\cdot,\cdot)$ denotes the hyperbolic distance with respect to the standard hyperbolic distance. We note that in the upper half plane we have
\[
D_X =\frac{-b}{2a}+\frac{i\sqrt{|d|}}{2|a|}
\]
for $X = \kzxz{b}{2c}{-2a}{-b}$ with $Q(X)=
%-N\det(X)=
Nd<0$. We set $D_X = \emptyset$ if $Q(X) \geq 0$. The stabilizer $G_X$ of $X$ in $G(\R)$ is isomorphic to $\SO(2)$ and for $X \in L'$, the group $\G_X = G_X \cap \G$ is finite. We denote
the image of $D_X$ in $M$, counted with multiplicity
$\tfrac1{|\overline{\G}_X|}$, by $Z(X)$.

For $m \in \Q^{\times}$ and $h \in L'/L$,  the group $\G$ acts on  ${L}_{m,h} =
\{X \in L +h ;\; Q(X) =m\}$
 with finitely many orbits. For $m<0$, we define the \emph{Heegner divisor} of index $(m,h)$ on $M$ by
\begin{equation}
Z(m,h) = \sum_{X \in \G \back L_{m,h} } Z(X).
\end{equation}
For any function $f$ on $M$, we then define the trace following
\cite{ZagierTr} and \cite{BFCrelle} by
 \begin{equation}
 \tr_{m,h}(f) = \sum_{X \in \G \back L_{m,h} } \frac{1}{\# \bar{\G}_X} f(D_X).
\end{equation}
For the lattice in Example~\ref{nicelattice} with $N=1$, this gives exactly twice the trace of modular functions defined in the introduction, since our trace counts positive and negative definite binary quadratic forms of discriminant $m$.

\subsection{Geodescis}

A vector $X \in V(\Q)$ of positive length $m$ defines a geodesic $c_X$ in $D$ via
\[
c_X = \{ z \in D; \; z \perp X \} = \{ z \in D; \; (X(z),X)=0 \},
\]
see e.g. \cite{Shintani}, \cite{KS}, \cite{KM90}. In this situation, we have
\[
|(X,X(z))| = 2\sqrt{m} \sinh(d(z,c_X)),
\]
where $d(z,c_X)$ denotes the hyperbolic distance of $z$ to the geodesic $c_X$. Hence $R(X,z) = 2m \cosh^2(d(z,c_X))$. Explicitly, for $X = \kzxz{b}{2c}{-2a}{-b}$, we have
\[
c_X = \{z \in D; \; a|z|^2+b\Re(z)+c=0\}.
\]
We orient the geodesics as follows. For $X = \pm \kzxz{1}{0}{0}{-1}$, the geodesic $c_{ X} = \pm (0,i\infty)$ is the imaginary axis with the indicated orientation. The orientation preserving action of $\SL_2(\R)$ then induces an orientation for all $c_X$.

We define the line measure $d z_{X}$ for $c_X$ by $d z_{X} = \pm
\tfrac{dz}{ \sqrt{m} z}$ for $X = \pm  \sqrt{m/N} \kzxz{1}{0}{0}{-1}$ and then by $dz_{g^{-1} X}
= d(gz)_X$ for $g \in \SL_2(\R)$. So for $X = \tfrac{1}{\sqrt{N}} \kzxz{b}{2c}{-2a}{-b}$,
we have
\[
dz_X = \tfrac{dz}{a z^2+bz+c}.
\]
In terms of $X(z)$ and $R(X,z)$ we have
\[
dz_X = -2i \frac{(X,\partial X(z))}{R(X,z)}.
\]
Indeed, this holds for $X = \sqrt{m/N}\kzxz{1}{0}{0}{-1}$, since we have $(X,\partial X(z)) = i \sqrt{m} \bar{z}/y^2 dz$ and $R(X,z) = 2m|z|^2/y^2$. Then the $G$-equivariance properties of $X(z)$ and $R(X,z)$ imply the claim for general $X$.

The stabilizer $\bar{\G}_X$ is either trivial (if the orthogonal complement $X^{\perp} \subset V$ is isotropic over $\Q$) or infinite cyclic (if $X^{\perp}$ is non-split over $\Q$). We set $c(X)=  \G_X \back c_X$, and by slight abuse of notation we use the same symbol for the image of $c(X)$ in $M$. If $\bar\G_X$ is infinite, then $c(X)$ is a closed geodesic in $M$, while $c(X)$ is an infinite geodesic if $\bar{\G}_X$ is trivial. The last case happens exactly when $Q(X) \in N (\Q^{\times})^2$. We define the trace for positive index $m$ and $h\in L'/L$ of a
continuous function $f$ on $M$ by
\begin{equation}
\tr_{m,h}(f) = \frac{1}{2\pi} \sum_{X \in \G \back L_{m,h} } \int_{c(X)} f(z) \,dz_X.
\end{equation}
Since
\begin{equation}\label{period-equiv}
\int_{c(X)} f(z) dz_X = \int_{c(g^{-1} X )} f(g z) dz_{g^{-1} X},
\end{equation}
for $g \in G$, this is independent of the choice of $X \in \G \back L_{m,h}$. Note that a priori the integral only converges if the geodesics are closed, i.e., $ m \notin N (\Q^{\times})^2$. Otherwise the geodesics $c(X)$ are infinite and $\int_{c(X)} f(z) dz_X$ may have to be regularized.
We will describe this in  the next subsection.

%We will describe this for  harmonic weak Maass form in the next subsection. Then the trace is defined in exactly the same way as before.

\subsection{Infinite Geodesics}
Assume that $X$ with $Q(X)=m>0$ gives rise to an infinite geodesic in $M$, that is $\bar{\G}_X =1$. So $c(X)= c_X$. These geodesics correspond to the split hyperbolic coefficients. In this section we will show  how to regularize the periods of harmonic weak Maass forms over the infinite geodesics. We also define the complementary trace which gives the contribution of the negative Fourier coefficients of the holomorphic part of $f$.

\subsubsection{Regularized periods and the central $L$-value of the $j$-invariant}

Assume that $X$ with $Q(X)=m>0$ gives rise to an infinite geodesic in $M$, that is $\bar{\G}_X =1$. So $c(X)= c_X$. We will describe now how for $f \in H^+_0(\G)$, we can regularize the period $\int_{c_X} f(z) dz_X$. Note that $f_\ell(z):=f(\sigma_\ell z)$ can be written as $f_\ell = f_\ell^+ + f_\ell^-$, where the Fourier expansions of $f_\ell^+$ and $f_\ell^-$  are of the form
\begin{align*}
f_\ell^+(z) = \sum_{n \in \frac1{\alpha_{\ell}} \Z} a^+_{\ell} (n) e(nz)  \qquad \text{and} \qquad
f_\ell^-(z) =  \sum_{\substack{n \in \frac1{\alpha_{\ell}} \Z_{< 0}}} a^-_{\ell}(n) e(n\bar{z}),
\end{align*}
where $ a^+_{\ell} (n) = 0$ for $n \ll 0$.

Now $X^{\perp}$ is split over $\Q$, a rational hyperbolic plane spanned by two rational isotropic lines $\ell_X$ and $\tilde\ell_X$. In fact, the geodesic $c_X$ connects the corresponding two cusps (which are not necessarily $\G$-inequivalent). We can distinguish these lines by requiring that $\ell_X$ represents the endpoint of the geodesic. Note $\tilde\ell_X = \ell_{-X}$. We have
\[
\sigma_{\ell_X}^{-1} X = \sqrt{m/N}\begin{pmatrix} 1& -2r \\ 0 & - 1 \end{pmatrix}.
\]
 for some $r \in \Q$. Hence the geodesic $c_X$ is explicitly given in $D \simeq \h$ by
\begin{equation}\label{realpart}
c_{X} = \sigma_{\ell_X} \{ z \in D; \; \Re(z) = r\}.
\end{equation}
We call $r=r_+= \re(c_X)$ the {\it real part} of the geodesic $c_X$. It depends on the choice of $\sigma_{\ell_X}$.  %By \eqref{period-equiv} we therefore see that it suffices by considering $f(\sigma_{\ell_X}z)$ instead of $f(z)$ to regularize the period for a vertical geodesic with real part $r$ as above.
Pick a number $c=c_+>0$. We then have (still formally)
\[
\sqrt{m}\int_{c_X} f(z) dz_X = \int_{\Re(z)=r_+} f_{X}(z) \frac{dz}{z-r} = \int_{c_+}^{\infty} f_X(iy+r_+) \frac{dy}{y} +   \int_{c_-}^{\infty} f_{-X}(iy + r_-) \frac{dy}{y}.
 \]
Here $f_{\pm X}(z) = f(\sigma_{\ell_{\pm X}} z)$, $r_-$ is the real part of $c_{-X}$ and $c_- = \Im(\sigma_{\ell_{-X}}^{-1}(r+ic_+)$. If we write $r=a/b$ with coprime $a,b \in \Z$ and $b>0$, then $c_- = 1/c_+b^2$.

The extension of the definition \eqref{d=1} to the general situation is

\begin{definition}\label{reg-def1}
Let $f \in H^+_0(\G)$ and let $c_X$ be an infinite geodesic connecting two rational cusps in $M$. Then with the notation as above we set
\begin{align*}
\sqrt{m}\int^{reg}_{c_X} f(z) dz_X &:=  a^+_{\ell_X}(0) \log c_+ +  \sum_{n \ne 0} a^+_{\ell_X}(n)e^{2\pi in r_+} \mathcal{EI}(2\pi nc_+)   \\
&  \quad  + a^+_{\ell_{-X}}(0) \log c_- +  \sum_{n \ne 0} a^+_{\ell_{-X}}(n)e^{2\pi in r_-}  \mathcal{EI}(2\pi nc_-) \\
& \quad  +\int_{c_+}^{\infty} f^-_X(iy+r_+) \frac{dy}{y} +   \int_{c_-}^{\infty} f^-_{-X}(iy + r_-) \frac{dy}{y}.
\end{align*}
This is well defined since the Fourier coefficients of $f$ also depend on the choice of $\sigma_{\ell}$.
\end{definition}

We now give a different characterization of the regularized integral. We will need $\psi(w) = \tfrac{ \G'(w)}{\G(w)}$, the digamma function, see \cite{AbSt}, for which we have
\begin{equation}\label{digamma}
\psi(w) = -\g + \sum_{n=0}^{\infty} \frac{1}{n+1} - \frac{1}{n+w}.
\end{equation}
We set
\[
c_X^{c,T} = \{ z \in c_X;  c \leq  \Im( \sigma_{\ell_X} z) \leq T \}.
\]

\begin{theorem}\label{reg-def2}
Let $f \in H^+_0(\G)$. Assume $c_X$ is a vertical geodesic. Then for any $T_+,T_->0$ and the notation as above, we have
\begin{align*}
& \int^{reg}_{c_X}  f(z) dz_X \\ & \quad = \int_{c_X^{c_+,T_+}}   f(z)  dz_X  - \frac1{2} \int^{i\tfrac{T_+}{\alpha_X}+1}_{i\tfrac{T_+}{\alpha_X}}
f(\alpha_X z+ r_+ ) \left(\psi(z) + \psi(1-z) + 2 \log \alpha_X \right) dz \\
& \quad \;+ \int_{c_{-X}^{c_-,T_-}}   f(z)  dz_{-X}  - \frac1{2} \int^{i\tfrac{T_-}{\alpha_{-X}}+1}_{i\tfrac{T_-}{\alpha_{-X}}}
f(\alpha_{-X} z+ r_- ) \left(\psi(z) + \psi(1-z) + 2 \log \alpha_{-X} \right) dz \\
& \quad  \; +\int_{c_+}^{\infty} f^-_X(iy+r_+) \frac{dy}{y} +   \int_{c_-}^{\infty} f^-_{-X}(iy + r_-) \frac{dy}{y}.
\end{align*}
In particular for $T_+=c_+$ and ${T_-}={c_-}$ and $f \in M^!_0(\G)$ we obtain
\begin{align*}
\int^{reg}_{c_X} f(z) dz_X &=
- \frac1{2} \int^{i\tfrac{c_+}{\alpha_X}+1}_{i\tfrac{c_+}{\alpha_X}}
f(\alpha_X z+ r_+ ) \left(\psi(z) + \psi(1-z) + 2 \log \alpha_{X}\right) dz \\
& \quad - \frac1{2} \int^{i\tfrac{c_-}{\alpha_{-X}}+1}_{i\tfrac{c_-}{\alpha_{-X}}}
f(\alpha_{-X} z+ r_- ) \left(\psi(z) + \psi(1-z) + 2 \log \alpha_{-X} \right) dz.
\end{align*}
\end{theorem}

\begin{proof}

For simplicity, we assume $c_X$ is the imaginary axis and also $f \in M^!_0(\G)$. We first show the independence of $T=T_+$. For that we also assume for the moment $\alpha_X=1$. Pick another $T_1>0$. By Cauchy's theorem we have
\begin{equation}\label{eq:T22}
 - \int_{iT}^{iT+1} f(z) \psi(z) dz =
- \int_{iT_1+1}^{iT+1}  f(z) \psi(z) dz - \int_{iT_1}^{iT_1+1}  f(z) \psi(z) dz
+ \int_{iT_1}^{iT}  f(z) \psi(z) dz.
\end{equation}
For the first integral on the right hand side we see using $\psi(z+1) = \psi(z) + 1/z$ that
\begin{equation}\label{eq:T33}
- \int_{iT_1+1}^{iT+1}  f(z) \psi(z) dz = -\int_{iT_1}^{iT}  f(z) \psi(z) dz -  \int_{iT_1}^{iT}f(z) \frac{dz}{z}.
\end{equation}
We also have
\begin{equation}\label{eq:T11}
 \int_{c_X^{c_+,T}}   f(z)  dz_X = \int_{ic_+}^{iT} f(z) \frac{dz}{z}.
\end{equation}
Then by \eqref{eq:T11}, \eqref{eq:T22}, and  \eqref{eq:T33} we conclude
\begin{equation*}
\frac12 \left(\int_{c_X^{c_+,T}}   f(z)  dz_X -\int_{iT}^{iT+1} f(z) \psi(z) dz \right) = \frac12 \left(
\int_{c_X^{c_+,T_1}}   f(z)  dz_X - \int_{iT_1}^{iT_1+1} f(z) \psi(z) dz \right).
\end{equation*}
The same holds for $\psi(z)$ replaced by $\psi(1-z)$. This shows the independence of $T$. Now assume $T=c_+=c$. With $\alpha = \alpha_X$ we first have
\begin{align*}
- a_{\ell_X}(0) \int_{ic/\alpha}^{ic/\alpha+1} \left( \psi(z) +  \psi(1-z) \right) dz &= - a_{\ell_X}(0) \left[ \log\left( \tfrac{\G(1+ic/\alpha)}{\G(ic/\alpha)}\right) - \log\left(\tfrac{\G(1-(1+ic/\alpha))}{ \G(1-ic/\alpha)}\right)\right]  \\ &=-2 a_{\ell_X}(0)  \log(c/\alpha).
 \end{align*}
 Now consider $f_0(z)= f(z)-a_{\ell_X}(0)$. Then from \eqref{digamma} we see
\begin{equation}
\label{psi-principle}
 \int_{ic/\alpha}^{ic/\alpha+1} f_0(\alpha z) (\psi(z) + \psi(1-z)) dz = -\int_{ic/\alpha}^{ic/\alpha+ \infty} f_0(\alpha z) \frac{dz}{z} +  \int_{ic/\alpha}^{ic/\alpha - \infty} f_0(\alpha z) \frac{dz}{z}.
\end{equation}
Plugging in the Fourier expansion of $f_0$ we obtain
\begin{multline*}
-\int_{ic/\alpha}^{ic/\alpha + \infty} f_0(\alpha z) \frac{dz}{z} +  \int_{ic/\alpha}^{ic/\alpha - \infty} f_0(\alpha z) \frac{dz}{z} \\ =-  \sum_{n \ne 0} a_{\ell_X}(n) \left( \int_{ ic}^{ic+ \infty}
 e^{2 \pi i n z} \frac{dz}{z} + \overline{\int_{ ic}^{ic+ \infty}e^{2 \pi i n z} \frac{dz}{z}} \right).
\end{multline*}
According to \cite[Equations 5.1.30/31]{AbSt}, we have
\[
\int_{ ic}^{ic+ \infty} e^{2 \pi i n z} \frac{dz}{z}=
\begin{cases}
E_1(2\pi nc ) & \text{if $n>0$} \\
-\Ei(2\pi |n|c) - i \pi & \text{ if $n<0$}.
\end{cases}
\]
So, finally,
\[
-  \frac12 \int_{ic}^{ic+1} f(z) (\psi(z) + \psi(1-z)+2 \log \alpha) dz = - a_{\ell_X}(0) \log(c) + \sum_{n \ne 0} a_{\ell_X}(n) \mathcal{EI}(2\pi nc).
\]
Carrying out the same analysis for the other cusp $\ell_{-X}$ completes the proof of the theorem.
\end{proof}

\begin{remark}
We consider $j_1(z) \in M_0^!(\SL_2(\Z))$. Then applying Theorem~\ref{reg-def2} gives
\[
\int_0^{\infty,reg} j_1(iy) \frac{dy}{y} = - \int_{i}^{i+1} j_1(z) \left(\psi(z) + \psi(1-z)\right) dz.
\]
This is exactly Zagier's regularization for the `central $L$-value of $j_1$'. He arrived to this formula by  the following heuristic considerations. We need to give a meaning to the expression
\[
2 \int_0^ij_1(z)\frac{dz}{z}.
\]
We deform the path of integration to the semicircle to the left (respectively right) of the imaginary axis starting at $0$ and ending at $i$. Under the transformation $z \to -1/z$ this path becomes the horizontal half line in the upper half plane beginning at $\infty$ (respectively $-\infty$) ending at $i$. Hence we obtain
\[
 \int_{i}^{i+\infty} j_1(z)\frac{dz}{z} +  \int_{i}^{i-\infty} j_1(z)\frac{dz}{z}
\]
But now these integrals converge, and by \eqref{psi-principle} we obtain
\[
-\int_{i}^{i+1} j_1(z) \left(\psi(z) + \psi(1-z)\right) dz = -2 \Re\left(  \int_{i}^{i+1} j_1(z) \psi(z) dz \right).
\]
Here we used that the Fourier coefficients of $j_1$ are real.
In fact, one can show that this is equal to
$
-2 \Re  \int_{\rho^2}^{{\rho}} j_1(z) \psi(z) \,dz
$,
where $\rho = e^{2\pi i/6}$.
\end{remark}

\begin{remark}
We work out the trace $\tr_{m^2,0}(1)$ for the constant function $1$. We have
\[
\tr_{m^2,0}(1) =\frac{1}{\pi m} \sum_{k=1}^{2m \eps_\ell} \log \frac{(k\beta_{\ell},2m)}{2m}.
\]
%In particular, in case of the lattice $L_1$ in Example~\ref{DIT-case} we obtain
%\[
%\tr_{m^2,0}(1) =\frac{1}{\pi m} \sum_{k=1}^{m-1} \log \frac{(k,m)}{m} = - \frac{p-1}{p} \log p,
%\]
%if $m=p$ is a prime.
\end{remark}

\subsubsection{Complementary trace}

Assume that $X$ with $Q(X)$ gives rise to an infinite geodesic, that is $\bar{\G}_X =1$.
Let $f \in H^+_0(\G)$ be a weak Maass form with holomorphic Fourier coefficients $a_{\ell}^+(n)$. Then we define its {\em complementary trace} for $m \in N (\Q^{\times})^2$ and $h\in L'/L$ by
\[
\tr^c_{m,h}(f) = \sum_{ X \in \G \back L_{m,h}} \sum_{n   <0} a^+_{\ell_X}(n) e^{2\pi i \re(c(X))n } +  \sum_{n <0} a^+_{\ell_{-X}}(n) e^{2\pi i \re(c(-X))n }.
\]
Note that in \cite{BFCrelle} this quantity is denoted by $\tr_{m,h}(f)$. We have (see \cite{BFCrelle}, Proposition~4.7)
\begin{multline*}
\tr^c_{m,h}(f) =
2 \sqrt{m/N}  \sum_{\ell \in \G \back \Iso(V)} \eps_{\ell}  \\ \times \left[\delta_{\ell}(m,h) \sum_{n \in \tfrac{2}{\beta_{\ell}}\sqrt{m/N} \Z_{<0} } a^+_{\ell}(n) e^{2\pi i  r_+ n }   +  \delta_{\ell}(m,-h)
\sum_{n  \in \tfrac{2}{\beta_{\ell}}\sqrt{m/N}\Z_{<0} }
a^+_{\ell}(n) e^{2\pi i  r_- n } \right].
\end{multline*}
Here $\delta_{\ell}(m,h) = 1$ if the $(m,h)$-th Fourier coefficient $b_{\ell}(m,h)$ of $\tilde\Theta_{K_\ell}(\tau)$ is nonzero, that is, if there exists a vector $X \in L_{m,h}$ such that $c_X$ ends at the cusp $\ell$. In that case $r_{\pm}$ is the real part of any such $X$. In particular, $ \tr^c_{m,h}(f) = 0$ for  $m \gg 0$.

\subsection{Average values of harmonic weak Maass forms}

We define the regularized `average value' of a (suitable) function $f$ on $M$ by
\[
 \int_M^{reg}f(z) d\mu(z) =  \lim_{T \to \infty} \int_{M_T} f(z) d\mu(z)
\]
as in \cite{BFCrelle}, (4.6). By Remark~4.9 in \cite{BFCrelle} we have for weakly holomorphic $f$ that
\begin{equation}\label{reg-F-value}
\int^{reg}_{M} f(z) d\mu(z)
=-8\pi \sum_{\ell \in \G \back \Iso(V)}   \alpha_\ell
\sum_{\substack{N\in \Z_{\geq 0}}} a_\ell(-N)\sigma_1(N).
\end{equation}
Here $\sigma_0(0) = -1/24$. The formula also holds for $f \in H_0^+(\G)$. Indeed, we let $\calE_2(z)=-\frac{3}{\pi y}-24\sum_{n=0}^\infty \sigma_1(n)e^{2\pi i nz}$ be the (non-holomorphic) Eisenstein series $\calE_2(z)$ of weight $2$ for $\Sl_2(\Z)$. Then $\bar\partial (\calE_{2}(z)dz)= - \tfrac{3}{\pi} d\mu(z)$. Hence by Stokes's theorem we obtain
\[
\int^{reg}_{M} f(z) d\mu(z) = - \frac{\pi}{3} \lim_{T \to \infty}\int_{\partial M_T} f(z) \calE_2(z)dz + \int_{M} (\bar \partial f(z)) \calE_{2}(z)dz,
\]
The first term gives \eqref{reg-F-value}, while the second vanishes as the Petersson scalar product of the cusp form $\xi_0(f)$ against an Eisenstein series.

We define as in \cite{BFCrelle} the trace of index $(0,h)$ by
\[
\tr_{0,h}(f) = - \delta_{h,0} \frac{1}{2\pi}  \int_M^{reg}f(z) d\mu(z).
\]
Here $\delta_{h,0}$ is Kronecker delta. For $f=1$ we also write $\vol(M) =  - \frac{1}{2\pi}  \int_M d\mu(z)$.

\section{The main result}\label{sec:Main-results}

We are now ready to state the main result of this paper. It will be proved using the regularized theta lift in Sections~\ref{sec:Green} and \ref{sec:Fourier}.

\begin{theorem}\label{th:Main-Maass}
Let $h \in L'/L$. Let $f \in H^+_0(\G)$ be a weak Maass form and assume that the constant coefficients $a_{\ell}^+(0)$ vanish at all cusps $\ell$.
Then the generating series
\begin{align*}
H_h(\tau,f):= &-2 \sqrt{v} \tr_{0,h}(f) \\
&+  \sum_{m<0}\tr_{m,h}(f) \frac{\erfc(2\sqrt{\pi |m|v}) }{2\sqrt{|m|}}e(m{\tau})\\
& + \sum_{m>0}  \tr_{m,h} (f) e(m\tau) \\
& +2 \sum_{m>0} \tr^c_{Nm^2,h}(f) \left( \int_0^{\sqrt{v}} e^{4\pi Nm^2 w^2} dw \right) e(Nm^2\tau)
\end{align*}
defines the $h$-component of a weak Maass form of weight $1/2$ for the representation $\rho_L$.
If $f$ has non-zero constant coefficients then one has to add
\begin{multline*}
-  \frac1{\sqrt{N}\pi }  \sum_{\ell \in \G \back \Iso(V)} a^+_{\ell}(0)\eps_{\ell} \biggl[ \frac12\left( \log(4 \beta_{\ell}^2 \pi v) +\g + \psi(k_{\ell}/\beta_{\ell})+\psi(1-k_{\ell}/\beta_{\ell})\right) \\
+ \sum_{m>0} b_{\ell}(Nm^2,h) \mathcal{F}(2 \sqrt{\pi v N}m) e(Nm^2\tau) \biggr]
\end{multline*}
to the generating series. Here $b_{\ell}(m,h)$ denotes the $(m,h)$-th coefficient of $\tilde\Theta_{K_\ell}(\tau)$,
%defined in \eqref{eq:tildethetaK},
and
$k_{\ell}$ is defined by $ \ell \cap (L+h) = \Z \beta_{\ell} u_{\ell} + k_{\ell}u_{\ell}$ and $0 \leq k_{\ell} < \beta_{\ell}$. Furthermore, we (formally) set $\psi(0)=-\g$, which is justified since $-\g$ is the constant term of the Laurent expansion of $\psi$ at $0$. Finally, we have set
\[
\mathcal{F}(t):= \log t - \sqrt{\pi} \int_0^t e^{w^2} \erfc(w) dw  + \tfrac12 \log(2) + \tfrac14 \g,
\]
where $\erfc(w) = \tfrac{2}{\sqrt{\pi}} \int_w^{\infty} e^{-t^2} dt$ is the complementary error function. In particular, we have
\[
\Delta_\tau H_h(\tau,f) =
-\sum_{\ell\in \Gamma\bs \Iso(V)}\frac{a^+_\ell(0) \eps_\ell }{4\sqrt{N}\pi }\theta_{K_{\ell,\bar{h}}}(\tau).
\]

\end{theorem}

As a special case we consider the constant function $f=1$.

\begin{theorem}\label{th:Main-con}
Let $h \in L'/L$. Then
\begin{align*}
H_h(\tau,1)=&-2 \sqrt{v} \vol(M) \\
&+\sum_{m<0} \deg Z(m,h)  \frac{\erfc(2\sqrt{\pi |m|v})}{2\sqrt{|m|}}{e(m{\tau})}\\
& + \sum_{\substack{m>0\\ m \notin N (\Q^{\times})^2} } \left(\sum_{X \in
\G \back L_{m,h}} \length(c(X)) \right) {e(m{\tau})}  \\
& + \sum_{m>0} \tr_{Nm^2,h}(1) q^{Nm^2} \\
&-  \frac1{\sqrt{N}\pi}  \sum_{\ell \in \G \back \Iso(V)} \eps_{\ell} \sum_{m>0} b_{\ell}(Nm^2,h)
 ( \mathcal{F}(2 \sqrt{\pi v N}m)) e(Nm^2\tau) \\
&-  \frac1{2\sqrt{N} \pi}  \sum_{\ell \in \G \back \Iso(V)} \eps_{\ell}  \left[  \log(4 \beta_{\ell}^2 \pi v) +\g + \psi(k_{\ell}/\beta_{\ell})+\psi(1-k_{\ell}/\beta_{\ell})\right].
\end{align*}
defines the $h$-component of a weak Maass form for $\rho_L$ of weight $1/2$.
\end{theorem}

\begin{remark}
The special functions $\mathcal{F}$ above and $\alpha$ in \eqref{DIT2} differ by a constant, since both map to $\beta_{3/2}$ under $\xi_{1/2}$. This constant is absorbed by the undefined term $\tr_{Nm^2}(1)$ in \eqref{DIT2}.
\end{remark}

\begin{remark}
Applying $\xi_{1/2}$ to $H_h(\tau,f)$ we recover the generating series for the traces of modular functions over CM points obtained in \cite{BFCrelle} (and \cite{FCompo} for $f=1$) which in turn generalized the results of Zagier \cite{Zagier,ZagierTr}.
\end{remark}

\begin{example}
We recover the theorems in the introduction by considering the lattice of Example \ref{nicelattice} for $N=p$ using  Example \ref{ex:iso}.
Alternatively, one can consider for $N=1$ the lattice
\[
L =\left\{\zxz{b}{2c}{-2ap}{-b}; \, a,b,c\in \Z \right\}
\]
and employ the same arguments as in \cite[Section~6]{BFCrelle}.
\end{example}

\begin{example}
We consider for $N=1$ the lattice $L$ of Example~\ref{nicelattice}. Let $h\in L'/L\cong\Z/2\Z$ be the non-trivial element.
Then for $m\in \Z_{>0}$ we have
\[
\tr_{m,L}(1) =2 H(4m) \qquad \text{and} \qquad\tr_{-m/4,L}(1)+ \tr_{-m/4,L+h}(1)= 2 H(m),
\]
where $H(m)$ is the Kronecker-Hurwitz class number, see e.g.~\cite[Section 3]{FCompo}. We let $r_3(m)$ be the representation number of $m$ as the sum of three squares. Then the famous class number relation states
$
r_3(m) = 12 \left(H(4m) - 2H(m) \right)$.
Hence if we define
\[
H(\tau):=6 \big( 2H_{L}(4\tau,1)+ 2H_{L+h}(4\tau,1) -H_L(\tau,1)\big),
\]
we obtain
\[
\xi_{1/2} H(\tau) = \theta^3(\tau).
\]

\end{example}

\section{Theta series and the regularized theta lift}\label{sec: reg-theta}

In this section we define regularized theta lifts of automorphic functions with singularities on $M$ against the theta function $\Theta_L(\tau,z,\varphi_0)$.

\subsection{Some Schwartz functions}

We consider the standard Gaussian $\varphi_0$ on $V(\R)$,
\begin{equation}
\varphi_0(X,z) = e^{-\pi (X,X)_z},
\end{equation}
where $(X,X)_z$ is the majorant associated to $z \in D$ which is given by
\begin{equation}\label{Rformel}
(X,X)_z= (X,X)+ (X,X(z))^2.
\end{equation}
Hence $(X,X)_z = -(X,X)+2R(X,z)$. Recall that the Schwartz function $\varphi_0$ has weight $1/2$ under the Weil representation acting on the space of Schwartz functions on $V(\R)$, see e.g. \cite{Shintani}. Accordingly, for $\tau = u+ iv \in \h$, we define
\begin{align}
\varphi_0(X,\tau,z) = \sqrt{v} \varphi_0(\sqrt{v}X,z) e^{\pi i (X,X) u} =  \sqrt{v}
e^{\pi i (X,X)_{\tau,z}},
\end{align}
where
\begin{equation}
(X,X)_{\tau,z} = u(X,X) + iv(X,X)_z
%= (X,X)\tau +iv(X,X(z))^2
= (X,X)\bar{\tau} + 2ivR(X,z).
\end{equation}
To distinguish between the Laplacians acting on functions in the two
variables $z\in D$ and $\tau\in \H$, we often write $\Delta_z=\Delta_{0,z} = -y^2
\left( \tfrac{\partial^2}{\partial x^2} + \tfrac{\partial^2}{\partial
    y^2}\right)$ for the hyperbolic Laplacian of weight $0$ on $D$,
and write $\Delta_{\tau} =\Delta_{1/2,\tau} $
%= - v^2 \left( \tfrac{\partial^2}{\partial u^2} +\tfrac{\partial^2}{\partial v^2}\right) +iv\tfrac{\partial}{\partial\overline{\tau}}$
for the hyperbolic Laplacian on $\h$ of weight
 $1/2$ as in \eqref{defdelta}. We have
\[
- \Delta_{1/2,\tau} = L_{5/2} R_{1/2} +\frac12 =
R_{-3/2}L_{1/2},
\]
where $R_{k} = 2i \tfrac{\partial}{\partial \tau} +kv^{-1}$ and $L_k
= -2i v^2 \frac{\partial}{\partial \bar{\tau}}$ are the weight $k$
raising and lowering operators acting on functions on $\h$.
The following lemma is well known, see e.g. \cite{Shintani}, \cite[p. 205, eqn. 2.10]{KS}.

\begin{lemma}\label{Laplace}
We have
\begin{equation*}
\Delta_{\tau} \varphi_0(X,\tau,z) = \frac14\Delta_z \varphi_0(X,\tau,z).
\end{equation*}
\end{lemma}

We define another Schwartz function by
\begin{equation}\label{KM-def}
\varphi_{1}(X,\tau,z) := -\frac1{\pi} L_{\frac12} \varphi_0(X,\tau,z).
\end{equation}
Hence $\varphi_1$ has weight $-3/2$.
\begin{lemma}\label{BFDuke}
We have
\[
\Delta_z \varphi_0(X,\tau,z) = 4\pi R_{-3/2}
\varphi_{1}(X,\tau,z).
\]
\end{lemma}

\begin{remark}
\label{KM-remark}
The Schwartz function $\varphi_1^V=\varphi_1$ associated to the space $V$ of signature $(2,1)$ is very closely related to the Kudla-Millson Schwartz form $\varphi_{KM}^{V^-}$ for the space $V^-$, which is given by considering on $V$ the quadratic form $-Q(X)$ of signature $(1,2)$ (see also \cite{BFCrelle}, Section~7). The form $\varphi_{KM}^{V^-}$ has weight $3/2$, and we have
\[
\varphi_{KM}^{V^-}(X,\tau,z) = v^{-3/2}\overline{\varphi_{1}(X,\tau,z)} d\mu(z).
\]
Here $d\mu(z) = \tfrac{dx\,dy}{y^2}$ is the invariant volume form on $D$. Moreover, we see that
\[
\varphi_{KM}^{V^-}(X,\tau,z) =  -\frac1{\pi} \xi_{1/2} \varphi_0(X,\tau,z) d\mu(z),
\]
where $\xi_k f = v^{k-2} \overline{L_k f } = R_{-k} v^k \overline{f}$.
%Recall that $\xi_k$ is a map from weight $k$ to the ``dual'' weight $2-k$.
\end{remark}

\subsection{Theta series and theta lifts}

We let $\varphi$ be the Schwartz function $\varphi_0$ or $\varphi_{1}$ on $V(\R)$ of weight $k$ (equal to $1/2$ or $-3/2$). Then for $h \in L'/L$, we define a theta series by
\begin{align}
\theta_{h}(\tau,z,\varphi) = \sum_{X \in h + L} \varphi(X,\tau,z).
\end{align}
In the variable $z$ it is $\Gamma$-invariant and therefore descends to
a function on $M$. We also define a vector valued theta series by
\begin{align}
\Theta_L(\tau,z,\varphi) =
 \sum_{h\in L'/L} \theta_{h}(\tau,z,\varphi) \frake_{h}.
\end{align}
%We often write $\Theta(\tau,z,\varphi)$.
As a function of $\tau \in \h$, we have that $\Theta_L(\tau,z,\varphi) \in A_{k,L}$, see e.g.~\cite{Bo1}.

We let $f(z)$ be a $\G$-invariant function on $D$. We define the
theta lift of $f$ by
\begin{align}\label{theta-integral}
I(\tau,f) = \int_{M} f(z) \Theta_L(\tau,z,\varphi_0) d\mu(z)
=  \sum_{h \in
L'/L} \left( \int_M f(z) \theta_{h}(\tau,z,\varphi_0)  d\mu(z) \right)
\frake_h,
\end{align}
where $ d\mu(z) = \tfrac{dx \, dy}{y^2}$ is the invariant volume form on $D$. We also write
\begin{equation}\label{comp-int}
I_h(\tau,f) = \int_M f(z) \theta_{h}(\tau,z,\varphi_0) d\mu(z)
\end{equation}
for the individual components.  If $f$ is of sufficiently rapid decay
at the cusps, the theta integral
converges and defines a (in general non-holomorphic) modular form on
the upper half plane of weight $1/2$ of type $\rho_L$. In fact, for $f$ a Maass cusp form, the lift was considered by Kartok-Sarnak \cite{KS}. In the present paper, we are particularly interested in the case when $f$ is {\em not\/} of rapid decay at the cusps. Then the theta
integral typically does not converge and needs to be regularized. We will carry this out in the remainder of this section. In Sections~\ref{sec:Green} and \ref{sec:Fourier} will show that $I(\tau,f)$ for $f \in H_0^+(\G)$ will give the generating series for the traces given in Section~\ref{sec:Main-results}.

\subsection{The growth of the theta kernel}

The growth of the theta functions $\Theta_L(\tau,z,\varphi)$ near the cusps of $M$ is given as follows.

\begin{proposition}
\label{theta-decay}
As $ y \to \infty$ we have
\begin{itemize}
\item[(i)]
\[
\Theta_L(\tau,\sigma_{\ell} z,\varphi_0) = y \frac{1}{\sqrt{N}\beta_{\ell}} \tilde\Theta_{K_{\ell}}(\tau) + O(e^{-Cy^2}).
\]
In particular, if $\ell^{\perp} \cap (L+h) = \emptyset$, then $\theta_h(\tau,\sigma_{\ell} z,\varphi_0) =
O(e^{-Cy^2})$.
\item[(ii)]
\[
\Theta_L(\tau, \sigma_{\ell} z,\varphi_{1})= O(e^{-Cy^2}).
\]
\item[(iii)]
\[
\Theta_L(\tau,\sigma_{\ell} z,\Delta_z \varphi_{0})= O(e^{-Cy^2}).
\]
\end{itemize}

In particular, $\theta_h(\tau,z,\varphi_{1})$ and  $\theta_h(\tau,z,\Delta_z \varphi_{0})$ are `square exponentially' decreasing at all cusps of $M$.

\end{proposition}

\begin{proof}

This follows from a very special case of \cite[Theorem~5.2]{Bo1}. For $\varphi_0$ also see \cite[Theorem~2.4]{Br2}, and for $\varphi_{1}$ in view of Remark~\ref{KM-remark} see the proof of Proposition~4.1 in \cite{FCompo} and \cite[Proposition~4.1]{BFCrelle}. For convenience we sketch the proof.

i)
Since $\Theta_L(\tau,\sigma_{\ell} z,\varphi)= \Theta_{\sigma_\ell^{-1}L}(\tau,z,\varphi)$, it
is enough to consider the cusp $\infty=\ell_0$. A primitive norm $0$ vector of $\sigma_\ell^{-1}L$ in $\ell_0$ is given by $B_\ell:=\kzxz{0}{\beta_\ell}{0}{0}$.
We may now apply \cite[Theorem~5.2]{Bo1} for the lattice $\sigma_\ell^{-1}L$, the Schwartz function $\varphi_0$, and the primitive norm $0$ vector $B_\ell$ to obtain the assertion.
\begin{comment}
with We can also assume $L+h = (r_2\Z + h_2)\kzxz{}{1}{}{} \oplus (r_3\Z +h_3)\kzxz{}{}{1}{} \oplus (r_1\Z +h_1)\kzxz{1}{}{}{-1}$. For $X = \kzxz{x_1}{x_2}{x_3}{-x_1}$, we have
\[
(X,X(z)) = \frac{\sqrt{N}}{y}( x_3z\bar{z}-2x_1x-x_2),
\]
so that
\[
\varphi_0(x,\tau,z) = \sqrt{v} e\left( v \tfrac{N}{2y^2}(x_3|z|^2-2x_1x-x_2)^2 + N(x_1^2+x_2x_3)\tau \right).
\]
This gives the desired decay for all terms in the sum over $L+h$ with $x_3 \ne 0$. For the vectors in $ \ell^{\perp} \cap L+h$ (i.e., $x_3=0$) we apply partial Poisson summation with respect to the sum over $x_2 \in r_2\Z +h_2$ and obtain
\begin{align*}
\sum_{\substack{x_1 \in r_1\Z + h_1 \\ x_2 \in r_2
\Z + h_2 }} \varphi_0\left(\left(
\begin{smallmatrix} x_1 & x_2 \\ & -x_1
\end{smallmatrix} \right),\tau,z \right) \notag
&=\frac{y}{\sqrt{N}r_2} \sum_{\substack{x_1 \in r_1\Z + h_1 \\ w \in \Z}} e\left(- w\tfrac{h_2}{r_2}
\right)e\left(Nx_1^2{\tau}\right) e\left(
\tfrac{-2wx_1x}{r_2}\right)   e^{ -\pi \frac{y^2w^2}{Nr_2^2v}}.
\end{align*}
The terms with $w \ne 0$ are $O(e^{-Cy^2})$, while for $w=0$ we obtain $y \frac{1}{\sqrt{N}\beta_{\ell}} \theta_{K_{\ell},h}(\tau)$. This gives (i).
\end{comment}

ii) Since $\varphi_{1} = -\tfrac{1}{\pi} L_{1/2}\varphi_0$ and since $\tilde \Theta_{K_{\ell}}(\tau)$ is holomorphic, we obtain (ii) by applying the lowering operator $L_{1/2}$ to (i).

iii) This follows from (ii) by applying the raising operator $R_{-3/2}$ to $\Theta_L(\tau,z,\varphi_{1})$.
\end{proof}

\begin{comment}
\begin{remark}\label{mixedmodelformula}
\texttt{This will come later!}
Let $L$ as in Example~\ref{nicelattice}. Then the proof of Proposition~\ref{theta-decay} for the cusp $\infty$ takes a particularly nice form, as we now explain. For the cusp $\infty$, we can realize $W$ as $\Q \kzxz{1}{}{}{-1}$. Hence $K = \Z \kzxz{1}{}{}{-1}$. Moreover $L'/L \simeq K'/K$.  We define the $\C[K'/K]$-valued theta series
\[
\Theta_{K}(\tau, \alpha,\beta) = \sum_{
\lambda\in K'} e\big( Q(\lambda+\beta)\tau
-(\lambda +\beta/2, \alpha)  \big) \mathfrak{e}_{\lambda+K},
\]
and we write $\theta_{K,h}(\tau,\alpha,\beta)$ for the individual components. Let $h \in L'
/L$. Then by \cite{Bo1}, Theorem~5.2 we have
\[
\theta_h(\tau,z,\varphi_0)  = \sqrt{N}y
\sum_{c, d \in \Z}  \exp\left( -\pi \frac{Ny^2}{2v} |c\tau+d|^2
\right) \theta_{K,h}\left(\tau, d x,  -c x \right)
\]
and
\[
\theta_h(\tau,z,\varphi_{1})   = -N^{3/2}y^3
\sum_{c, d \in \Z}(c\tau+d)^2
 \exp\left( -\pi \frac{Ny^2}{2v} |c\tau+d|^2
\right) \theta_{K,h}\left(\tau, d x,  -c x\right).
\]
\end{remark}
\end{comment}

\subsection{Regularization using truncated fundamental domains}

Following an idea of Borcherds \cite{Bo1} and Harvey--Moore \cite{HM}, we regularize the theta integral by integrating over a truncated fundamental domain and then taking a limit.

If $h(s)$ is a meromorphic function in a neighborhood of $s_0\in \C$, we denote by
$\CT_{s=s_0} [h(s)]$ the constant term in the Laurent expansion of $h$ at $s=s_0$.
Let $\calF=\{z\in \H;\; \text{$|x|\leq 1/2$ and $|z|\geq 1$}\}$ be the standard fundamental domain for the action of $\Gamma(1)=\Sl_2(\Z)$ on the upper half plane. For a positive integer $a$ we put
\[
\calF^a:=\bigcup_{j=0}^{a-1}
 \zxz{1}{j}{0}{1}\calF.
\]
As before, let $\Gamma\subset \Gamma(1)$ be a congruence subgroup.
Recall that for $\ell\in \Iso(V)$ we chose $\sigma_\ell\in \Gamma(1)$ such that $\sigma_\ell \ell_0 = \ell$, and  $\alpha_\ell$ denotes the width of the cusp $\ell$, see Section~\ref{r:sigma}.

%We obtain a fundamental domain for the action of $\Gamma$ on $D$ as follows.

\begin{lemma}
We have the disjoint left coset decomposition
\[
\bar\Gamma(1)=\bigcup_{\ell\in \Gamma\bs \Iso(V)}\bigcup_{j\; (\alpha_\ell)}
\bar \Gamma  \sigma_\ell \zxz{1}{j}{0}{1}.
\]
\end{lemma}

Consequently, a fundamental domain or the action of $\Gamma$ on $D$ is given by
\begin{align}
\label{eq:ti1.5}
\calF(\Gamma)= \bigcup_{\ell\in \Gamma\bs \Iso(V)}\sigma_\ell\calF^{\alpha_\ell}.
\end{align}
Moreover, if $f$ is of rapid decay, the theta integral \eqref{theta-integral} is given by
\begin{align}
\label{eq:ti2}
I(\tau,f) = \sum_{\ell\in \Gamma\bs \Iso(V)}
\int_{\calF^{\alpha_\ell}} f(\sigma_\ell z) \Theta_L(\tau,\sigma_\ell z,\varphi_0) d\mu(z).
\end{align}

Now assume that $f$ is a function on $M$ which is not necessarily decaying at the cusps.
For $T>0$, let we truncate $\calF^a$ and put $\calF^a_T:=\{z\in \calF^a;\; y\leq T\}$.

\begin{definition}
\label{def:iint}
We define the regularized theta lift of $f$ by
\begin{align}
I^{reg}(\tau,f)= \CT_{s=0}\left[I^{reg}(\tau,s,f)\right],
\end{align}
where
% as the constant term in the Laurent expansion at $s=0$ of the meromorphic continuation in $s$ of
\begin{align}
\label{eq:ti3}
I^{reg}(\tau,s,f) = \sum_{\ell\in \Gamma\bs \Iso(V)}\lim_{T\to \infty}
\int_{\calF_T^{\alpha_\ell}} f(\sigma_\ell z) \Theta_L(\tau,\sigma_\ell z,\varphi_0) y^{-s}d\mu(z).
\end{align}
\end{definition}

Here $s$ is an additional complex variable. If $f$ is rapidly decaying at the cusps, then it is easily seen that the regularized theta lift agrees with the classical lift \eqref{eq:ti2}.
However, as we will now see, the regularized lift makes sense for a much wider class of functions $f$.

\begin{proposition}
\label{prop:reg1}
Let $f$ be a weak Maass form of weight zero for $\Gamma$ with eigenvalue $\lambda=s'(1-s')$.
Denote the constant term of the Fourier expansion of $f$ at the cusp $\ell$ by
$A_\ell y^{s'}+B_\ell y^{1-s'}$ with constants $A_\ell,B_\ell \in \C$.
Then $I^{reg}(\tau,s,f)$ converges locally uniformly in $s$ for $\Re(s)>\max(\Re(s'),1-\Re(s'))$ and defines an element of $A_{1/2,L}$. It has a meromorphic continuation to the whole $s$-plane, which is holomorphic in $s$ up to first order poles at $s=s'$ and $s=1-s'$. The function
\begin{align*}
&I^{reg}(\tau,s,f) -\frac{1}{\sqrt{N}}\sum_{\ell\in \Gamma\bs \Iso(V)} \eps_\ell \tilde\Theta_{K_{\ell}}(\tau)\left(\frac{A_\ell}{s-s'}+\frac{B_\ell}{s+s'-1}\right)
%\\
%&\frac{1}{\sqrt{N}}\sum_{\ell\in \Gamma\bs \Iso(V)} B_\ell\eps_\ell \tilde\Theta_{K_{\ell}}(\tau),
\end{align*}
has a holomorphic continuation to all $s\in \C$.
Moreover, the regularized theta lift $I^{reg}(\tau,f)$ defines an element of  $A_{1/2,L}$.
\end{proposition}

\begin{proof}
It suffices to show that for any $\ell\in \Gamma\bs \Iso(V)$, the integral
\begin{align}
\label{eq:ti4}
I_\ell^{reg}(\tau,s,f) = \lim_{T\to \infty}
\int_{\calF_T^{\alpha_\ell}} f(\sigma_\ell z) \Theta_L(\tau,\sigma_\ell z,\varphi_0) y^{-s}d\mu(z)
\end{align}
converges for $\Re(s)>\max(\Re(s'),1-\Re(s'))$ and has a meromorphic
continuation in $s$ with the appropriate poles and residues.  In view
of Proposition \ref{theta-decay}, we split up the integral as follows:
\begin{align*}
I_\ell^{reg}(\tau,s,f) &= \lim_{T\to \infty}
\int_{\calF_T^{\alpha_\ell}} f(\sigma_\ell z)
\left(\Theta_L(\tau,\sigma_\ell z,\varphi_0) -y \frac{1}{\sqrt{N}\beta_{\ell}} \tilde\Theta_{K_{\ell}}(\tau)\right) y^{-s}d\mu(z)\\
&\phantom{=}{}+
\lim_{T\to \infty}
\int_{\calF_T^{\alpha_\ell}} f(\sigma_\ell z) \frac{1}{\sqrt{N}\beta_{\ell}} \tilde\Theta_{K_{\ell}}(\tau) y^{1-s}d\mu(z).
\end{align*}
Because of Proposition \ref{theta-decay}, the function in the integral of the first summand is of square exponential decay as $y\to \infty$. Therefore the first summand converges for all $s\in \C$ and defines a holomorphic function of $s$.

For the second summand we split up the integral as
\begin{align*}
\lim_{T\to \infty}
\int_{\calF_T^{\alpha_\ell}} f(\sigma_\ell z) \frac{ \tilde\Theta_{K_{\ell}}(\tau)}{\sqrt{N}\beta_{\ell}} y^{1-s}d\mu(z)
&=\frac{1}{\sqrt{N}\beta_{\ell}} \tilde\Theta_{K_{\ell}}(\tau)\int_{\calF_1^{\alpha_\ell}} f(\sigma_\ell z)  y^{1-s}d\mu(z)\\
&\phantom{=}{}+\frac{1}{\sqrt{N}\beta_{\ell}} \tilde\Theta_{K_{\ell}}(\tau)
\int_{y=1}^\infty\int_{x=0}^{\alpha_\ell} f(\sigma_\ell z) \,dx\,\frac{dy}{y^{1+s}}.
\end{align*}
The first summand on the right hand side is an integral over a compact domain. It converges for all $s$ and defines a holomorphic function in $s$. For the second summand on the right hand side we use the Fourier expansion of $f$ at the cusp $\ell$,
\[
f(z)= \sum_{n\in \Z} a_\ell(n,y)e(nx/\alpha_\ell).
\]
We find that it is given by
\begin{align}
\label{eq:critical}
\frac{\alpha_\ell}{\sqrt{N}\beta_{\ell}} \tilde\Theta_{K_{\ell}}(\tau)
\int_{y=1}^\infty a_\ell(0,y)\,\frac{dy}{y^{1+s}}.
\end{align}
If we write $a_\ell(0,y) = A_\ell y^{s'}+B_\ell y^{1-s'}$, we see that the integral exists for $\Re(s)>\max(\Re(s'),1-\Re(s'))$, and it is equal to
\[
\frac{\alpha_\ell}{\sqrt{N}\beta_{\ell}} \tilde\Theta_{K_{\ell}}(\tau)\left(\frac{A_\ell}{s-s'}+\frac{B_\ell}{s+s'-1}\right).
\]
Using the fact that $\eps_\ell = \alpha_\ell/\beta_\ell$ and putting together the contributions of the different cusps $\ell$, we obtain the assertion.
\end{proof}

In the next proposition we give a formula for $I^{reg}(\tau,f)$ as a limit, not involving the additional parameter $s$.

\begin{proposition}
\label{prop:reg2}
Let $f$ be a weak Maass form of weight zero for $\Gamma$ with eigenvalue $\lambda=s'(1-s')$.
Denote the constant term of the Fourier expansion of $f$ at the cusp $\ell$ by
$A_\ell y^{s'}+B_\ell y^{1-s'}$ with constants $A_\ell,B_\ell \in \C$.
If $s'\neq 0,1$, then $I^{reg}(\tau,f)$ is equal to
\begin{align*}
\sum_{\ell\in \Gamma\bs \Iso(V)}\lim_{T\to \infty}\left[
\int_{\calF_T^{\alpha_\ell}} f(\sigma_\ell z) \Theta_L(\tau,\sigma_\ell z,\varphi_0) d\mu(z)-\left(A_\ell\frac{T^{s'}}{s'}+B_\ell\frac{T^{1-s'}}{1-s'}\right)\frac{\eps_\ell\tilde\Theta_{K_{\ell}}(\tau)}{\sqrt{N}} \right].
\end{align*}
If $s'=0$, then the same formula holds when $\frac{T^{s'}}{s'}$ is replaced by  $\log(T)$.
If $s'=1$, then the same formula holds when $\frac{T^{1-s'}}{1-s'}$ is replaced by  $\log(T)$.
\end{proposition}

\begin{proof}
%Throughout the proof, we denote by
%$\CT_{s=0} [h(s)]$ the constant term in the Laurent expansion at $s=0$ of a meromorphic function $h$ in $s$.
For a cusp $\ell$, we let $I^{reg}_\ell(\tau,f)=\CT_{s=0}[I_\ell^{reg}(\tau,s,f)]$.
 %be the constant term in the Laurent expansion at $s=0$ of $$.
Then we have
\[
I^{reg}(\tau,f)=\sum_{\ell\in \Gamma\bs \Iso(V)} I^{reg}_\ell(\tau,f),
\]
and it suffices to show that $I^{reg}_\ell(\tau,f)$ is given by
\[
\lim_{T\to \infty}\left[
\int_{\calF_T^{\alpha_\ell}} f(\sigma_\ell z) \Theta_L(\tau,\sigma_\ell z,\varphi_0) d\mu(z)-\left(A_\ell\frac{T^{s'}}{s'}+B_\ell\frac{T^{1-s'}}{1-s'}\right)\frac{\eps_\ell\tilde\Theta_{K_{\ell}}(\tau)}{\sqrt{N}} \right].
\]
The proof of Proposition \ref{prop:reg1} shows that
$I^{reg}_\ell(\tau,f)$ is equal to
\begin{align*}
&\lim_{T\to \infty}
\int_{\calF_T^{\alpha_\ell}} f(\sigma_\ell z)
\left(\Theta_L(\tau,\sigma_\ell z,\varphi_0) -y \frac{1}{\sqrt{N}\beta_{\ell}} \tilde\Theta_{K_{\ell}}(\tau)\right) d\mu(z)\\
&{}+\frac{1}{\sqrt{N}\beta_{\ell}} \tilde\Theta_{K_{\ell}}(\tau)\int_{\calF_1^{\alpha_\ell}} f(\sigma_\ell z)  y \,d\mu(z)\\
&{}+\frac{\eps_\ell}{\sqrt{N}} \tilde\Theta_{K_{\ell}}(\tau)
\CT_{s=0}\left[\int_{y=1}^\infty (A_\ell y^{s'}+B_\ell y^{1-s'})\,\frac{dy}{y^{1+s}}\right].
\end{align*}
Now  the simple observation
$$\int_{\calF_T^{\alpha_\ell}} f(\sigma_\ell z)y\,d\mu(z)=\int_{\calF_1^{\alpha_\ell}}
f(\sigma_\ell z)y\,d\mu(z)+\  \int_{y=1}^T \int_{x=0}^{\alpha_{\ell} }f(\sigma_\ell z) y\,d\mu(z)$$
together with the identity
\begin{align*}
&\CT_{s=0}\left[\int_{y=1}^\infty (A_\ell y^{s'}+ B_\ell y^{1-s'}) \,\frac{dy}{y^{1+s}}\right] \\
&=\lim_{T\to\infty}\left[ A_\ell\left(\int_1^T y^{s'}\, \frac{dy}{y} -\frac{T^{s'}}{s'}\right)+B_\ell\left(\int_1^T y^{1-s'}\, \frac{dy}{y} -\frac{T^{1-s'}}{1-s'}\right)\right]\\
&=\lim_{T\to\infty}\left[  \int_{y=1}^T \int_{x=0}^{\alpha_{\ell} }f(\sigma_\ell z) y\,d\mu(z)-\left(A_{\ell}\frac{T^{s'}}{s'} +B_\ell \frac{T^{1-s'}}{1-s'}\right)\right]
\end{align*}
gives the result when $s'\neq 0$.
The  $s'=0$ case follows similary using the identity
\[
\CT_{s=0}\left[\int_{y=1}^\infty   \,\frac{dy}{y^{1+s}}\right] =\lim_{T\to\infty} \left[ \int_1^T \frac{dy}{y} -\log(T)\right].
\]
%\[
%\CT_{s=0}\left[\int_{y=1}^\infty y^{s'} \,\frac{dy}{y^{1+s}}\right] =
%\begin{cases} \lim_{T\to\infty} \left[ \int_1^T y^{s'}\, \frac{dy}{y} -\frac{T^{s'}}{s'}\right],&\text{if $s'\neq 0$,}\\[2ex]
% \lim_{T\to\infty} \left[ \int_1^T \frac{dy}{y} -\log(T)\right],&\text{if $s'=0$,}
% \end{cases}
%\]
\end{proof}

\subsubsection{The Laplacian}
We now consider the action of the hyperbolic Laplacian on the regularized theta lift. The main result of this section is
\begin{theorem}
\label{thm:lapl}
Let $f$ be a weak Maass form for $\Gamma$ with eigenvalue $\lambda=s'(1-s')$.
Denote the constant term of the Fourier expansion of $f$ at the cusp $\ell$ by
$A_\ell y^{s'}+B_\ell y^{1-s'}$ with constants $A_\ell,B_\ell \in \C$.
We have
%$I^{reg}(\tau,f)$ is equal to
\begin{align*}
4\Delta_\tau I^{reg}(\tau,f)=\begin{cases}\lambda I^{reg}(\tau,f),&\text{if $s'\neq 0,1$,}\\[1ex]
-\sum_{\ell\in \Gamma\bs \Iso(V)}\frac{A_\ell\eps_\ell }{\sqrt{N}}\tilde \Theta_{K_\ell}(\tau),&\text{if $s'= 0$,}\\[1ex]
-\sum_{\ell\in \Gamma\bs \Iso(V)}\frac{B_\ell\eps_\ell }{\sqrt{N}}\tilde \Theta_{K_\ell}(\tau),&\text{if $s'= 1$.}
\end{cases}
\end{align*}
\end{theorem}

To prove this theorem we need two propositions which will be used also in the sequel.
We start by noting that  in the view of Propsition \ref{theta-decay}, for $f\in A_0(\Gamma)$ which has  at most linear exponential growth at the cusps of $\Gamma$,     the integral
$$\int_M f(z) \Theta_L(\tau,z,\Delta_z\varphi_0)\,d\mu(z)$$ converges, and we have:

\begin{proposition}
\label{prop:reg2.5}
Let $f\in A_0(\Gamma)$ and assume that $f$ has at most linear exponential growth at the cusps of $\Gamma$.
Denote the constant term of the Fourier expansion of $f$ at the cusp $\ell$ by
$a_\ell(0,y)$.
Then we have
\begin{align*}
&\int_M f(z) \Theta_L(\tau,z,\Delta_z\varphi_0)\,d\mu(z)\\
&=
\sum_{\ell\in \Gamma\bs \Iso(V)}
\lim_{T\to \infty}\Bigg[
\int_{\calF_T^{\alpha_\ell}} (\Delta_z f)(\sigma_\ell z) \Theta_L(\tau,\sigma_\ell z,\varphi_0) d\mu(z)\\
&\phantom{=
\sum_{\ell\in \Gamma\bs \Iso(V)}
\lim_{T\to \infty}\Big[}
{}
-\frac{\eps_\ell\tilde\Theta_{K_{\ell}}(\tau)}{\sqrt{N}} \left[(
1-y\frac{\partial}{\partial y}) a_\ell(0,y)\right]_{y=T}\Bigg].
\end{align*}
In particular, the limit on the right hand side exists.
\end{proposition}

\begin{proof}
%First, we note that in view of Propsition \ref{theta-decay} the integrals over $M$ in the statement of %Proposition \ref{prop:reg3} converge absolutely.
According to \eqref{eq:ti1.5}, we have
\begin{align}
\label{eq:u11}
\int_M f(z) \Theta_L(\tau,z,\Delta_z\varphi_0)\,d\mu(z)=
\sum_{\ell\in \Gamma\bs \Iso(V)}
\lim_{T\to \infty}
\int_{\calF_T^{\alpha_\ell}} f(\sigma_\ell z) \Theta_L(\tau,\sigma_\ell z,\Delta_z\varphi_0) d\mu(z).
\end{align}
%Note that in view of Propsition \ref{theta-decay}, the limit as $T\to \infty$ on the right hand side exists.
We use Stokes' theorem to rewrite the integral. For smooth functions $f$ and $g$ on $D$ we have
\begin{equation}\label{diff-ops}
\frac{2}{i}(\partial \bar\partial f) g = (\Delta f) g\,d\mu(z)
=  f(\Delta g)\,d\mu(z) +\frac{2}{i} d\big( (\bar \partial f) g+ f (\partial g)\big).
\end{equation}
Consequently,
\begin{align*}
& \int_{\calF_T^{\alpha_\ell}} f(\sigma_\ell z) \Theta_L(\tau,\sigma_\ell z,\Delta_z\varphi_0) d\mu(z)\\
&=\int_{\calF_T^{\alpha_\ell}} (\Delta_z f)(\sigma_\ell z) \Theta_L(\tau,\sigma_\ell z,\varphi_0) d\mu(z)\\
&\phantom{=} {}-\frac{2}{i}
\int_{\partial \calF_T^{\alpha_\ell}} \left[\big(\bar \partial f(\sigma_\ell z)\big) \Theta_L(\tau,\sigma_\ell z,\varphi_0) + f(\sigma_\ell z)\big(\partial \Theta_L(\tau,\sigma_\ell z,\varphi_0)\big) \right].
\end{align*}
Using Propsition \ref{theta-decay} and inserting the Fourier expansions of $f$ at the cusps, we find for $T\to \infty$  that
\begin{align}
\label{eq:u22}
&\sum_{\ell\in \Gamma\bs \Iso(V)}
\int_{\calF_T^{\alpha_\ell}} f(\sigma_\ell z) \Theta_L(\tau,\sigma_\ell z,\Delta_z\varphi_0) d\mu(z)\\
\nonumber
&= \sum_{\ell\in \Gamma\bs \Iso(V)}
\int_{\calF_T^{\alpha_\ell}} (\Delta_z f)(\sigma_\ell z) \Theta_L(\tau,\sigma_\ell z,\varphi_0) d\mu(z)\\
\nonumber
&\phantom{=} {}+\frac{2}{i}
\sum_{\ell\in \Gamma\bs \Iso(V)}\frac{\tilde \Theta_{K_\ell}(\tau)}{\beta_\ell\sqrt{N}}
\int\limits_{z=0+iT}^{\alpha_\ell+iT}
 \left(\frac{\partial}{\partial \bar z} a_\ell(0,y)\right) y+ a_\ell(0,y)\left(\frac{\partial}{\partial z} y\right)\, dx + O(\frac{1}{T}).
\end{align}
The second term on the right hand side is equal to
\[
-\sum_{\ell}\frac{\eps_\ell \tilde \Theta_{K_\ell}(\tau)}{\sqrt{N}}
\left[(1-y\frac{\partial}{\partial y} )a_\ell(0,y)\right]_{y=T}
\]
Inserting this into \eqref{eq:u22} and then into \eqref{eq:u11}, we obtain the assertion.
\end{proof}

\begin{proposition}
\label{prop:reg3}
Let $f$ be a weak Maass form of weight zero for $\Gamma$ with eigenvalue $\lambda=s'(1-s')$.
Denote the constant term of the Fourier expansion of $f$ at the cusp $\ell$ by
$a_\ell(0,y)=A_\ell y^{s'}+B_\ell y^{1-s'}$ with constants $A_\ell,B_\ell \in \C$.
If $s'\neq 0,1$, then
%$I^{reg}(\tau,f)$ is equal to
\begin{align*}
I^{reg}(\tau,\Delta_z f)=\int_M f(z) \Theta_L(\tau,z,\Delta_z\varphi_0)\,d\mu(z).
\end{align*}
If $s'=0$, then we have
\begin{align*}
0=I^{reg}(\tau,\Delta_z f)=\int_M f(z) \Theta_L(\tau,z,\Delta_z\varphi_0)\, d\mu(z)+\sum_{\ell\in \Gamma\bs\Iso(V)} \frac{A_\ell\eps_\ell}{\sqrt{N}}\tilde \Theta_{K_\ell}(\tau).
\end{align*}
If $s'=1$, then we have
\begin{align*}
0=I^{reg}(\tau,\Delta_z f)=\int_M f(z) \Theta_L(\tau,z,\Delta_z\varphi_0)\,d\mu(z)+\sum_{\ell\in \Gamma\bs\Iso(V)} \frac{B_\ell\eps_\ell}{\sqrt{N}}\tilde \Theta_{K_\ell}(\tau).
\end{align*}
\end{proposition}

\begin{proof}
%First, we note that in view of Propsition \ref{theta-decay} the integrals over $M$ in the statement of Proposition \ref{prop:reg3} converge absolutely.
We use Proposition \ref{prop:reg2.5} for
\[
\int_M f(z) \Theta_L(\tau,z,\Delta_z\varphi_0)\,d\mu(z)
\]
together with the fact that
\[
%\sum_{\ell\in \Gamma\bs \Iso(V)}\frac{\eps_\ell \tilde \Theta_{K_\ell}(\tau)}{\sqrt{N}}\left(
%(1-s')A_\ell T^{s'}+s' B_\ell T^{1-s'}\right)=
%
\left[(1-y\frac{\partial}{\partial y} )a_\ell(0,y)\right]_{y=T}
=(1-s')A_\ell T^{s'}+ s' B_\ell T^{1-s'}.
\]
Comparing this with the formula for $I^{reg}(\tau,\Delta_z f)$ of Proposition \ref{prop:reg2}, we obtain the assertion.
\end{proof}

%\begin{theorem}
%\label{thm:lapl}
%Let $f$ be a weak Maass form for $\Gamma$ with eigenvalue $\lambda=s'(1-s')$.
%Denote the constant term of the Fourier expansion of $f$ at the cusp $\ell$ by
%$A_\ell y^{s'}+B_\ell y^{1-s'}$ with constants $A_\ell,B_\ell \in \C$.
%We have
%%$I^{reg}(\tau,f)$ is equal to
%\begin{align*}
%4\Delta_\tau I^{reg}(\tau,f)=\begin{cases}\lambda I^{reg}(\tau,f),&\text{if $s'\neq 0,1$,}\\[1ex]
%-\sum_{\ell\in \Gamma\bs \Iso(V)}\frac{A_\ell\eps_\ell }{\sqrt{N}}\tilde \Theta_{K_\ell}(\tau),&\text{if $s'= 0$,}\\[1ex]
%-\sum_{\ell\in \Gamma\bs \Iso(V)}\frac{B_\ell\eps_\ell }{\sqrt{N}}\tilde \Theta_{K_\ell}(\tau),&\text{if $s'= 1$.}
%\end{cases}
%\end{align*}
%\end{theorem}

\begin{proof}[Proof of Theorem \ref{thm:lapl}]
%Finally we return to the proof of  Theorem \ref{thm:lapl}.
According to Proposition \ref{prop:reg2} and Lemma \ref{Laplace} we have
\begin{align*}
4\Delta_\tau I^{reg}(\tau,f)&= 4\sum_{\ell\in \Gamma\bs \Iso(V)}\lim_{T\to \infty}
\int_{\calF_T^{\alpha_\ell}} f(\sigma_\ell z) \Delta_\tau\Theta_L(\tau,\sigma_\ell z,\varphi_0) d\mu(z)\\
&= \int_M f(z) \Theta_L(\tau,z,\Delta_z\varphi_0).
%-\left(A_\ell\frac{T^{s'}}{s'}+B_\ell\frac{T^{1-s'}}{1-s'}\right)\frac{\eps_\ell\tilde\Theta_{K_{\ell}}(\tau)}{\sqrt{N}} \right].
\end{align*}
Now the assertion of the theorem follows from Proposition \ref{prop:reg3}.
\end{proof}

\begin{comment}
\subsubsection{The harmonic case}

We now approach a different approach which does not use differential operators or Poincar\'e series, but rather follows the approach of Borcherds \cite{Bo1}.

We let $f(z)$ be a weak Maass form with constant terms $a_{\ell}(0)$. Then we define
\begin{multline}
\int^{reg}_M f(z) \theta_{h}(\tau,z,\varphi_0) d\mu(z) \\
= \lim_{T \to \infty} \left[
\int_{M_T} f(z) \theta_{h}(\tau,z,\varphi_0) d\mu(z) - \sum_{[\ell]} \eps_{\ell} a_{\ell}(0) \theta_{K_{\ell},h}(\tau) \log T \right].
\end{multline}
\end{comment}

\subsection{Regularization using differential operators}

Proposition \ref{prop:reg3} also leads to a different way of defining the regularized integral for weak Maass forms with non-zero eigenvalue. This regularization uses differential operators, in fact, the Laplace operator $\Delta_z$ on $M$.
It is in the spirit of the regularized Siegel-Weil formula of Kudla-Rallis via regularized theta lifts \cite{KR}. For a related treatment, see Matthes \cite{Matthes}.

\begin{proposition}
\label{prop:reg4}
Let $f$ be a weak Maass form for $\Gamma$ with eigenvalue $\lambda\neq 0$.
Then
\begin{align*}
I^{reg}(\tau,f)=\frac{1}{\lambda}
\int_M f(z) \Theta_L(\tau,z,\Delta_z\varphi_0).
\end{align*}
The integral on the right hand side converges absolutely.
\end{proposition}

\begin{proof}
According to Proposition \ref{prop:reg3}, we have
\begin{align*}
I^{reg}(\tau,f)&= \frac{1}{\lambda}I^{reg}(\tau,\Delta_z f)\\
&= \frac{1}{\lambda}
\int_M f(z) \Theta_L(\tau,z,\Delta_z\varphi_0).
\end{align*}
By Proposition~\ref{theta-decay}, the integral converges.
This proves the proposition.
\end{proof}

By Lemma~\ref{BFDuke} we directly see:

\begin{lemma}
If $f$ is a weak Maass form for $\Gamma$ with eigenvalue $\lambda\neq 0$, then
\[
I^{reg}(\tau,f) =   \frac{4\pi}{\la} R_{-3/2} \left( \int_M f(z) \Theta_{L}(\tau,z, \varphi_{1})\, d\mu(z) \right).
\]

\end{lemma}

\subsubsection{Eigenvalue zero}

We now explain how the regularization by means of differential operators
described in Proposition \ref {prop:reg4} can be extended to the case when the eigenvalue is $0$.
The idea is to use a `spectral deformation' of $f$ into a family of eigenfunctions.

\begin{proposition}
\label{prop:defor}
Let $f\in H_0(\Gamma)$ be a harmonic weak Maass form. There exists an open neighborhood $U\subset \C$ of $1$ and a holomorphic family of functions $(f_s)_{s\in U}$ on $D$ such that
$f_s(z)$ is a weak Maass form of weight $0$ for $\Gamma$ with eigenvalue $s(1-s)$, and  $f_1=f$.
\end{proposition}

\begin{proof}
This result can be proved using Poincar\'e and Eisenstein series for $\Gamma$.
See for example \cite[Section 3]{Fay}, \cite[pp.~660]{Hej}, or \cite[Proposition 1.12]{Br2}.
\end{proof}

Let $f\in H_0(\Gamma)$ and $(f_s)_{s\in U}$ be as in Proposition \ref{prop:defor}.
Denote the constant term of the Fourier expansion of $f_s$ at the cusp $\ell$ by
$a_\ell(0,y,s)=A_\ell(s) y^{s}+B_\ell(s) y^{1-s}$ with holomorphic functions $A_\ell(s),B_\ell(s)$.
In view of Proposition \ref{prop:reg4}, we have for $s\in U\setminus\{1\}$ that
\[
I^{reg}(\tau,f_s)= \frac{1}{s(1-s)} \int_M f_s(z) \Theta_L(\tau,z,\Delta_z\varphi_0)\,d\mu(z).
\]
The right hand side defines a meromorphic function for all $s\in U$.
In view of Proposition \ref{prop:reg3}, it has a first order pole at $s=1$ with residue
\begin{align}
\label{eq:jres}
\sum_{\ell\in \Gamma\bs\Iso(V)} \frac{B_\ell(1)\eps_\ell}{\sqrt{N}}\tilde \Theta_{K_\ell}(\tau).
\end{align}
We can define a regularized theta integral by putting
\begin{align}
\label{eq:jint}
J^{reg}(\tau,f):= \CT_{s=1}\left[\frac{1}{s(1-s)} \int_M f_s(z) \Theta_L(\tau,z,\Delta_z\varphi_0)\, d\mu(z)\right].
\end{align}
We now compare this with the regularized theta integral of Definition \ref{def:iint}.

\begin{proposition}
\label{prop:specdef}
Let $f\in H_0(\Gamma)$ and $(f_s)_{s\in U}$ be as in Proposition \ref{prop:defor}.
Denote the constant term of the Fourier expansion of $f_s$ at the cusp $\ell$ by
$a_\ell(0,y,s)=A_\ell(s) y^{s}+B_\ell(s) y^{1-s}$ with holomorphic functions $A_\ell(s),B_\ell(s)$. Then we have
\[
J^{reg}(\tau,f)= I^{reg}(\tau,f)
+\sum_{\ell\in \Gamma\bs\Iso(V)} \frac{B'_\ell(1)\eps_\ell}{\sqrt{N}}\tilde \Theta_{K_\ell}(\tau).
\]
\end{proposition}

\begin{proof}
We put the Laurent expansion $f_s(z) = f(z)+f'_1(z)(s-1)+O((s-1)^2)$ of $f_s$ at $s=1$ into the definition of $J^{reg}(\tau,f)$. Here $f'_1(z)$ means the value of $f'_s(z)= \frac{\partial}{\partial s} f_s(z)$ at $s=1$.
Noticing that $\CT_{s=1}[\frac{1}{s(1-s)}]=1$,  we obtain that
\begin{align}
\label{eq:jint1}
J^{reg}(\tau,f)= \int_M f(z) \Theta_L(\tau,z,\Delta_z\varphi_0)\,d\mu(z)
-\int_M f_1'(z) \Theta_L(\tau,z,\Delta_z\varphi_0)\,d\mu(z).
\end{align}
According to Proposition \ref{prop:reg3}, the first term on the right hand side is equal to
\[
-\sum_{\ell\in \Gamma\bs\Iso(V)} \frac{B_\ell(1)\eps_\ell}{\sqrt{N}}\tilde \Theta_{K_\ell}(\tau).
\]

We compute the second term on the right hand side of \eqref{eq:jint1} by means of Proposition \ref{prop:reg2.5}.
We have
\[
\left[(
1-y\frac{\partial}{\partial y}) a_\ell(0,y,s)\right]_{y=T}= (1-s)A_\ell(s)T^s+sB_\ell(s)T^{1-s}.
\]
Consequently, for the derivative with respect to $s$ at $s=1$ we find
\[
\left[(
1-y\frac{\partial}{\partial y}) a_\ell'(0,y,1)\right]_{y=T}= -A_\ell(1)T+B_\ell(1)+B_\ell'(1)-B_\ell(1)\log(T).
%-A_\ell(0)+A_\ell'(0)+A_\ell(0)\log(T)+B_\ell(0)T.
\]
If we call this quantity $C_\ell(T)$, we obtain
\begin{align*}
&\int_M f'_1(z) \Theta_L(\tau,z,\Delta_z\varphi_0)\,d\mu(z)\\
&=
\sum_{\ell\in \Gamma\bs \Iso(V)}
\lim_{T\to \infty}\Bigg[
\int_{\calF_T^{\alpha_\ell}} (\Delta_z f'_1)(\sigma_\ell z) \Theta_L(\tau,\sigma_\ell z,\varphi_0) d\mu(z)-\frac{\eps_\ell\tilde\Theta_{K_{\ell}}(\tau)}{\sqrt{N}} C_\ell(T)
\Bigg].
\end{align*}
Since $\Delta f_s= s(1-s) f_s$, we have $\Delta f_1'=-f$. By means of Proposition
\ref{prop:reg2}, we get
\begin{align*}
&\int_M f'_1(z) \Theta_L(\tau,z,\Delta_z\varphi_0)\,d\mu(z)=-I^{reg}(\tau,f)-\sum_{\ell\in \Gamma\bs \Iso(V)}\left(B_\ell(1)+B_\ell'(1)\right)\frac{\eps_\ell\tilde\Theta_{K_{\ell}}(\tau)}{\sqrt{N}}.
\end{align*}
Inserting this into \eqref{eq:jint1}, we obtain the assertion.
\end{proof}

In particular, we see that the regularized theta integral $J^{reg}(\tau,f)$ depends on the choice of the spectral deformation $f_s$. However, the dependency is mild, since only the derivatives of the constant terms in the Fourier expansions at $s=1$ enter.

\section{The lift of Poincar\'e series and the regularized lift of $j_m$}
\label{sect:6}

\subsection{Scalar valued  Poincar\'e series of weight $0$}
\label{sect:3.3}

Here we construct scalar valued Poincar\'e series of weight $0$ for the
group $\Gamma\subset G(\Q)$.  We will show in Section \ref{sect:6}
that the theta lifts of these series are given by linear combinations
of the Poincar\'e series of the previous subsection.  For simplicity we
assume for the construction that
$\Gamma=\Gamma_0(N)$, since this is the only case which we will need later for the comparison of our results with \cite{DIT}. We let $\Gamma_{\infty}= \langle \kzxz{1}{1}{0}{1}\rangle$ be the subgroup of translations.

Let $I_{\nu}(z)$  be the usual modified  Bessel function as in \cite[Chapter 9.6]{AbSt}.
For $s\in \C$, $y\in
\R_{>0}$ and $n\in \Q$, we let
\begin{align}\label{calI}
\calI_{n}(y,s)=
\begin{cases}
2\pi |n|^{\frac{1}{2}} y^{\frac{1}{2}} I_{s-\frac{1}{2}}(2\pi |n|y),& \text{if $n\neq 0$,}\\
y^{s}, & \text{if $n= 0$.}
\end{cases}
\end{align}
For $m\in \Z$ we define
\begin{align}\label{DefF}
G_{m}(z,s)=\frac{1}{2}\sum_{\gamma\in \Gamma_\infty\bs \Gamma}
\left[ \calI_{m}(y,s) e(m x) \right]\mid_{0}\gamma.
\end{align}
The series converges for $\Re(s)>1$ and defines a weak Maass form of
weight $0$ for $ \Gamma_0(N)$. It has the
eigenvalue $s(1-s)$ under $\Delta_{0}$.
The function $G_0$ is the usual Eisenstein series while $G_m$ for $m
 \neq 0$ was studied by Neunh\"offer \cite{Ne} and Niebur \cite{Ni},
 among others.
%For $m=0$, the function $f_{0}(z,s)$ is the usual
%Eisenstein series of weight $0$ for $\Gamma_0(N)$ in the cusp $\infty$.
If $m\neq 0$, it follows from its Fourier expansion, Weil's bound and the properties
of the $I$-Bessel function that $G_m(z,s)$ has a holomorphic continuation to $\Re(s)>3/4$.
If $m<0$, then $G_m(z,1)\in H^+_0(\Gamma)$.

\subsection{Vector valued Poincar\'e  series of half-integral weight}
\label{sect:3.2}
We recall the definition of vector valued Poincar\'e series for the
Weil representation in a setup which is convenient for the present
paper. These series are vector valued analogues of the Poincar\'e
series of weight $1/2$ considered in \cite[Section 2]{DIT}.

Let $M_{\nu,\,\mu}(z)$ and $W_{\nu,\,\mu}(z)$ \label{bi4} be the usual
Whittaker functions (see p.~190 of \cite{AbSt}).  For $s\in \C$, $v\in
\R_{>0}$ and $n\in \Q$, we let
\begin{align}\label{calM}
\calM_{n}(v,s)=
\begin{cases} \Gamma(2s)^{-1}(4\pi |n| v)^{-1/4} M_{\frac{1}{4}\sgn n,\,s-\frac{1}{2}}(4\pi |n|v),& \text{if $n\neq 0$,}\\
v^{s-\frac{1}{4}}, & \text{if $n= 0$.}
\end{cases}
\end{align}
For $h\in L'/L$,  and $m\in \Z+Q(h)$ we define
\begin{align}\label{DefcalF}
P_{m,h}(\tau,s)=\frac{1}{2}\sum_{\gamma\in \Gamma'_\infty\bs \Gamma'}
\left[ \calM_{m}(v,s) e(m u)\frake_h\right]\mid_{1/2,L}\gamma.
\end{align}
Here
 $\Gamma'_\infty:=\langle T\rangle\subset\tilde\Gamma$ with $T = \left( \kzxz{1}{1}{0}{1},1\right)$. The series converges for $\Re(s)>1$ and defines a weak Maass form of
weight $1/2$ for $ \Gamma'$ with representation $\rho_L$. It has the
eigenvalue $(s-\frac{1}{4})(\frac{3}{4}-s)$ under $\Delta_{1/2}$.
When $Q(h)\in \Z$ and $m=0$, the function $P_{0,h}(\tau,s)$ is a
vector valued Eisenstein series of weight $1/2$.
%The special value at $s=1-k/2$ is harmonic.
%If $k\in\Z-\frac{1}{2}$, it has the principal part
%$ q^{-m}\frake_h+ q^{-m}\frake_{-h}+C$ for some constant $C\in\C[L'/L]$.

\begin{comment}
\begin{equation}
P_{L}(\t,s)= \sum_{\g' \in \G'_\infty\backslash\G' }\left(\psi_{N}^{1/2}(\t, s)\mathfrak{e}_{h}\right)\mid_{1/2}(\g'),
\end{equation}
where
 \begin{equation}
\psi_N^{1/2}(\t,s)= \G(s+1/2)^{-1} (\pi Nv)^{-1/4} M_{\frac{1}{4},\frac{s}{2}-\frac{1}{4}} (\pi Nv)\exp( \frac{\pi iNu}{2}).
\end{equation}
{\bf need general lattice}
Here $M$ denotes the $M$-Whittaker function as in...
\end{comment}

For $L$ as in Example~\ref{nicelattice} we may apply the map \eqref{+space-form} to $P_{m,h}(\tau,s)$ for $m=\frac{d}{4N}$ and $d\in \Z$ to obtain the Poincar\'e series $P_d^+(\tau,s)$ which were considered in \cite{DIT}.

%\begin{remark}
%Applying the $\xi$-operator to $P_{L}(\t,s)$ gives Poincar\'e series... of weight $3/2$ which was con%sidered by....
%\end{remark}

\subsection{The lift of Poincar\'e series}

We now assume that $L$ is the lattice defined in
Remark \ref{nicelattice}.  We identify the finite quadratic module
$L'/L$ with $\Z/2N\Z$ equipped with the quadratic form $r\mapsto
r^2/4N$.  Moreover, we assume that $\Gamma=\Gamma_0(N)$.  In this
section we will explicitly calculate the regularized lift of the
Poincar\'e series $G_{-m}(z,s)$ defined in Section \ref{sect:3.3} for
$m\geq 0$.

For the cusp $\ell_0=\infty$, we can realize $W=V\cap \ell^\perp/\ell$
as $\Q \kzxz{1}{}{}{-1}$. Hence $K:=K_{\ell_0} = \Z
\kzxz{1}{}{}{-1}$. Moreover $L'/L \simeq K'/K$.  For $\alpha,\beta\in
W(\R)$ we define the $\C[K'/K]$-valued theta series
\[
\Theta_{K}(\tau, \alpha,\beta) = \sum_{
\lambda\in K'} e\big( Q(\lambda+\beta)\tau
-(\lambda +\beta/2, \alpha)  \big) \mathfrak{e}_{\lambda+K},
\]
and we write $\theta_{K,h}(\tau,\alpha,\beta)$ for the individual
components.  Notice that the theta function $\Theta_{K}(\tau)$ defined
earlier is equal to $\Theta_{K}(\tau, 0,0)$.  The following
proposition is a special case of \cite[Theorem~5.2]{Bo1}.

\begin{proposition}
\label{prop:poincare}
We have the identity
\begin{align*}
\Theta_L(\tau,z,\varphi_0) &= \sqrt{N}y\Theta_{K}(\tau,0,0)\\
&\phantom{=}
{}+ \frac{\sqrt{N}y}{2}\sum_{n=1}^\infty \sum_{\gamma\in \Gamma'_\infty\bs \Gamma'}
\left[ \exp\left(-\frac{\pi N n^2 y^2}{v}\right) \Theta_K(\tau,nx ,0)\right]\mid_{1/2,K}\gamma.
\end{align*}
\end{proposition}

%\texttt{There should be no $2$ in the denominator of the argument of exp!}

\begin{comment}
\[
\theta_h(\tau,z,\varphi_0)  = \sqrt{N}y
\sum_{c, d \in \Z}  \exp\left( -\pi \frac{Ny^2}{2v} |c\tau+d|^2
\right) \theta_{K,h}\left(\tau, d x,  -c x \right)
\]
and
\[
\theta_h(\tau,z,\varphi_{1})   = -N^{3/2}y^3
\sum_{c, d \in \Z}(c\tau+d)^2
 \exp\left( -\pi \frac{Ny^2}{2v} |c\tau+d|^2
\right) \theta_{K,h}\left(\tau, d x,  -c x\right).
\]
\end{comment}

\begin{theorem}
\label{thm:liftp}
Assume that $\Re(s)>1$.
If $m$ is a positive integer we have that
\begin{align*}
I^{reg}(\tau, G_{-m}(z,s)) =
\sqrt{\pi N}\Gamma\left(s/2\right)
%(s-1) \frac{\G (s/2+1)}{2\sqrt{\pi}}
\sum_{n|m}
P_{\frac{m^2}{4Nn^2},\frac{m}{n}}(\tau, \frac{s}{2}+\frac{1}{4}).
\end{align*}
For  $m=0$ we have
\begin{align*}
I^{reg}(\tau,G_{0}(z,s))
&=\frac{N^{\frac{1}{2}-\frac{s}{2}}}{2}\zeta^*(s)P_{0,0}(\tau,\frac{s}{2}+\frac{1}{4}).
\end{align*}
Here $\zeta^*(s)=\pi^{-s/2}\Gamma(s/2)\zeta(s)$, and $P_{m,h}(\tau,s)$ denotes the $\C[L'/L]$-valued weight $1/2$ Poincar\'e series defined in Section \ref{sect:3.2}.
\begin{comment}
\begin{align}
I^+(\t, G_{-m}(z,s)) & := \sum_h I_h(4\t,G_{-m}(z,s))\\
%& =\frac{1}{s(s-1)}\xi_{3/2}\int_ {\G\backslash D} G_{-m}(z,s)\Theta_{KM}(\t,z,)\\
& =\frac{1}{s} \frac{\G (s/2+1)}{2\sqrt{\pi}}\sum_{n | m} P_{n^2}^+(\t,s/2+1/4)
\end{align}
where $P_d^+(\t,s)$ is the scalar valued half integral poincare series in DIT .
\end{comment}
\end{theorem}

\begin{proof}
According to Proposition \ref{prop:reg4} we have
\begin{align*}
I^{reg}(\tau,G_{-m}(z,s))=\frac{1}{s(1-s)}
\int_M G_{-m}(z,s) \Theta_L(\tau,z,\Delta_z\varphi_0)\,d\mu(z).
\end{align*}
The theta function on the right hand side is square exponentially decreasing at all cusps.
Hence, by the usual unfolding argument, we find that
\begin{align*}
I^{reg}(\tau,G_{-m}(z,s))=\frac{1}{s(1-s)}
\int_{\Gamma_{\infty}\bs \H} \calI_{-m}(y,s) e(-m x)\Theta_L(\tau,z,\Delta_z\varphi_0)\, d\mu(z).
\end{align*}
By Proposition \ref{prop:poincare}, we may replace
$\Theta_L(\tau,z,\Delta_{0,z}\varphi_{0})$ by
$\Delta_{0,z}
\tilde \Theta_L(\tau,z,\varphi_{0})$, where
\[
\tilde \Theta_L(\tau,z,\varphi_{0})=\frac{\sqrt{N}y}{2}\sum_{n=1}^\infty \sum_{\gamma\in \Gamma'_\infty\bs \Gamma'}
\left[ \exp\left(-\frac{\pi N n^2 y^2}{v}\right) \Theta_K(\tau,nx ,0)\right]\mid_{1/2,K}\gamma.
\]
The function $\tilde \Theta_L(\tau,z,\varphi_{0})$ and its partial derivatives
have square exponential decay as $y \to\infty$.
Therefore, for $\Re(s)$ large, we may move the Laplace operator to $\calI_{-m}(y,s) e(-m x)$
to obtain
\begin{align}
\label{eq:lifti}
I^{reg}(\tau,G_{-m}(z,s))&=\frac{1}{s(1-s)}
\int_{\Gamma_{\infty}\bs \H}\big(\Delta_{z} \calI_{-m}(y,s) e(-m x)\big)\tilde\Theta_L(\tau,z,\varphi_0)\, d\mu(z)\\
\nonumber
&= \int_{\Gamma_{\infty}\bs \H} \calI_{-m}(y,s) e(-m x)
\tilde\Theta_L(\tau,z,\varphi_{0})\,d\mu(z)\\
\nonumber
&=\frac{\sqrt{N}}{2}\sum_{n=1}^\infty\sum_{\gamma\in \Gamma'_\infty\bs \Gamma'}  I(\tau,s,m,n)\mid_{1/2,K}\gamma,
\end{align}
where
\begin{align*}
I(\tau,s,m,n)=  \int_{y=0}^\infty\int_{x=0}^1 \calI_{-m}(y,s) e(-m x)
\exp\left(-\frac{\pi N n^2 y^2}{v}\right)
\Theta_K(\tau,nx ,0)y\,d\mu(z).
\end{align*}
Using the fact that $K'=\Z\kzxz{1/2N}{}{}{-1/2N}$ and the identification $K'/K\cong \Z/2N\Z$, we have
\[
\Theta_{K}(\tau, nx,0) = \sum_{b\in \Z} e\left(\frac{b^2}{4N}\tau-nbx\right)\frake_b.
\]
Inserting this in the formula for $I(\tau,s,m,n)$, and by integrating over $x$, we see that
$I(\tau,s,m,n)$ vanishes when  $n\nmid m$. If $n\mid m$, then  only the summand for $b=-m/n$ occurs, and so
\begin{align}
\label{eq:iformula}
I(\tau,s,m,n)&=  \int_{0}^\infty \calI_{-m}(y,s) \exp\left(-\frac{\pi N n^2 y^2}{v}\right)\frac{dy}{y}
 e\left( \frac{m^2}{4Nn^2}\tau\right)\frake_{-m/n}.
\end{align}

We first compute the latter integral for $m>0$. In this case we have
\[
\calI_{-m}(y,s)= 2\pi m^{\frac{1}{2}} y^{\frac{1}{2}} I_{s-\frac{1}{2}}(2\pi my).
\]
Inserting this and substituting $t=y^2$ in the integral, we obtain
\begin{align*}
&\int_{0}^\infty \calI_{-m}(y,s) \exp\left(-\frac{\pi N n^2 y^2}{v}\right)\frac{dy}{y}\\
&= 2\pi \int_{0}^\infty \sqrt{m y}I_{s-1/2}(2\pi m y)\exp\left(-\frac{\pi N n^2 y^2}{v}\right)\frac{dy}{y}\\
&= \pi \sqrt{m}\int_{0}^\infty I_{s-1/2}(2\pi m \sqrt{t}) \exp\left(-\frac{\pi N n^2 t}{v}\right)t^{-3/4}\,dt.
\end{align*}
The latter integral is a Laplace transform which is computed in \cite{B2} (see (20) on p.197).
%
%Now the following integral formula (see  \cite{GR}, p 741, (6.643.2))
%\begin{equation}\label{GRintegral}
%\int_0^\infty t^{\mu-1/2}e^{-\alpha t} %I_{2\nu}(2\beta\sqrt{t})dt=\dfrac{\G(\mu+\nu+1/2)}{\G(2\nu+1)}\beta^{-1}e^{\beta^2/2\alpha}\alpha^{-\mu}M_{-\mu,\nu}(\beta^2/\alpha)
%\end{equation}
%with $\alpha=\pi N n^2/v$, $\beta=\pi m$, $\mu=-1/4$ and $2\nu=s-1/2$,   gives
%
Inserting the evaluation, we obtain
\begin{align*}
&\int_{0}^\infty \calI_{-m}(y,s) \exp\left(-\frac{\pi N n^2 y^2}{v}\right)\frac{dy}{y}\\
&=\frac{\sqrt{\pi}\Gamma\left(s/2\right)}{\Gamma(s+1/2)}\left(\frac{ Nn^2}{ \pi m^2 v}\right)^{1/4} M_{1/4,s/2-1/4}\left(\frac{\pi m^2 v}{Nn^2}\right)
\exp\left( \frac{\pi m^2 v}{2Nn^2}\right)\\
&=\sqrt{\pi}\Gamma\left(s/2\right)
\calM_{\frac{m^2}{4Nn^2}}(v,s/2+1/4)
\exp\left( \frac{\pi m^2 v}{2Nn^2}\right).
\end{align*}
Consequently, we have in the case $n\mid m$ that
\begin{align*}
I(\tau,s,m,n)&=  \sqrt{\pi}\Gamma\left(s/2\right)
\calM_{\frac{m^2}{4Nn^2}}(v,s/2+1/4)
e\left(\frac{m^2}{4Nn^2} u \right)\frake_{-m/n}.
\end{align*}
Substituting this in  \eqref{eq:lifti}, we see that
\begin{align*}
I^{reg}(\tau,G_{-m}(z,s))
&=\sqrt{\pi N}\Gamma\left(s/2\right)
\sum_{n\mid m}
P_{\frac{m^2}{4Nn^2},-\frac{m}{n}}(\tau,\frac{s}{2}+\frac{1}{4}).
\end{align*}
Since $P_{m,h}(\tau,s)=P_{m,-h}(\tau,s)$, this concludes the proof of the theorem for $m>0$.

We now compute integral in \eqref{eq:iformula} for $m=0$. In this case we have
$\calI_{0}(y,s)= y^s$.
Inserting this into \eqref{eq:iformula},
%and substituting $t=y^2$ in the integral,
we find
\begin{align*}
\int_{0}^\infty \calI_{0}(y,s) \exp\left(-\frac{\pi N n^2 y^2}{v}\right)\frac{dy}{y}&=  \int_{0}^\infty \exp\left(-\frac{\pi N n^2 y^2}{v}\right)y^{s-1}\,dy\\
&= \frac{\Gamma(s/2)}{2}\left(\frac{v}{\pi N n^2}\right)^{s/2}.
\end{align*}
Hence, we obtain
\begin{align*}
I(\tau,s,m,n)&= \frac{\Gamma(s/2)}{2}\left(\frac{v}{\pi N n^2}\right)^{s/2}\frake_0. \end{align*}
Substituting into  \eqref{eq:lifti}, we see that
\begin{align*}
I^{reg}(\tau,G_{0}(z,s))&=\frac{N^{\frac{1}{2}-\frac{s}{2}}}{4}
\pi^{-s/2}\Gamma(s/2)\zeta(s)
\sum_{\gamma\in \Gamma'_\infty\bs \Gamma'}  v^{s/2}\frake_0\mid_{1/2,K}\gamma\\
&=\frac{N^{\frac{1}{2}-\frac{s}{2}}}{2}\zeta^*(s)P_{0,0}(\tau,\frac{s}{2}+\frac{1}{4}).
\end{align*}
This concludes the proof of the theorem for $m=0$.
\end{proof}

\subsection{The case of level $1$}

As an application, we consider the case $N=1$ where $\Gamma=\Sl_2(\Z)$. We compute the lift of the space $M^!_0(\Gamma)=\C[j]$. A basis for this space is given by the functions $j_m$ for $m\in \Z_{\geq 0}$ whose Fourier expansion starts as
\[
j_m (z)= q^{-m}+O(q).
\]
For instance, we have $j_0=1$ and $j_1=j-744$.

%----- Other spectral deformation
\begin{comment}
We begin by computing the lift of the constant function in terms of Eisenstein series.
As a spectral deformation of the constant function $j_0=1$ in the sense of Proposition \ref{prop:defor} we chose
\[
f_s(z)=2(s-1)G_0^*(z,s),
\]
where $G_0^*(z,s)=\zeta^*(2s)G_0(z,s)$ is the `completed' non-holomorphic Eisenstein series for $\Sl_2(\Z)$. Then we have $f_1(z)=1$, and the constant term of $f_s(z)$ at the cusp $\ell_0=\infty$ is given by
$A_\infty(s)y^s+B_\infty(s)y^{1-s}$ with
\begin{align*}
A_\infty(s)&= 2(s-1)\zeta^*(2s)= \frac{\pi}{3}(s-1)+O((s-1)^2),\\
B_\infty(s)&= 2(s-1)\zeta^*(2-2s)=1+(\gamma-\log(4\pi))(s-1)+O((s-1)^2).
\end{align*}
According the Theorem \ref{thm:liftp}, for $\Re(s)>1$, the lift of $f_s$ is equal to
\begin{align}
\label{eq:lifte}
I^{reg}(\tau,f_s)= (s-1)\zeta^*(s)\zeta^*(2s)P_{0,0}(\tau,\frac{s}{2}+\frac{1}{4}).
\end{align}
By \eqref{eq:jres}, the right hand side has a meromorphic continuation to $\C$ with a first order pole at $s=1$ with residue $B_\infty(1)\Theta_K(\tau)=\Theta_K(\tau)$. In particular, we see that
$P_{0,0}(\tau,\frac{s}{2}+\frac{1}{4})$ has a first order pole at $s=1$ with residue $\frac{6}{\pi}\Theta_K(\tau)$. We obtain the following corollary to Theorem \ref{thm:liftp}.
% and Proposition \ref{prop:specdef}.

\begin{corollary}
We have
\begin{align*}
J^{reg}(\tau,1) &= \CT_{s=1}\left[ (s-1)\zeta^*(s)\zeta^*(2s)P_{0,0}(\tau,\frac{s}{2}+\frac{1}{4})\right],\\
I^{reg}(\tau,1) &= J^{reg}(\tau,1) -B'_\infty(1) \Theta_K(\tau).
\end{align*}
\end{corollary}
\end{comment}
%-----------------

We begin by computing the lift of the constant function in terms of Eisenstein series.
As a spectral deformation of the constant function $j_0=1$ in the sense of Proposition \ref{prop:defor} we chose
\[
f_s(z)=\frac{\zeta^*(2)}{\zeta^*(2s-1)}G_0(z,s).
\]
It is well known that $G_0(z,s)$ has a first order pole at $s=1$, which cancels out against the pole of $\zeta^*(2s-1)$. We have $f_1(z)=1$, and the constant term of $f_s(z)$ at the cusp $\ell_0=\infty$ is given by
$A_\infty(s)y^s+B_\infty(s)y^{1-s}$ with
\begin{align*}
A_\infty(s)&= \frac{\zeta^*(2)}{\zeta^*(2s-1)}= \frac{\pi}{3}(s-1)+O((s-1)^2),\\
B_\infty(s)&= \frac{\zeta^*(2)}{\zeta^*(2s)}=1+\left(\gamma+\log(\pi)-\frac{12\zeta'(2)}{\pi^2}\right)(s-1)+O((s-1)^2).
\end{align*}
According to Theorem \ref{thm:liftp}, for $\Re(s)>1$, the lift of $f_s$ is equal to
\begin{align}
\label{eq:lifte}
I^{reg}(\tau,f_s)= \frac{\pi\zeta^*(s)}{12\zeta^*(2s-1)}P_{0,0}(\tau,\frac{s}{2}+\frac{1}{4}).
\end{align}
By \eqref{eq:jres}, the right hand side has a meromorphic continuation to $\C$ with a first order pole at $s=1$ with residue $B_\infty(1)\Theta_K(\tau)=\Theta_K(\tau)$. In particular, we see that
$P_{0,0}(\tau,\frac{s}{2}+\frac{1}{4})$ has a first order pole at $s=1$ with residue $\frac{6}{\pi}\Theta_K(\tau)$. We obtain the following corollary to Theorem \ref{thm:liftp}.
% and Proposition \ref{prop:specdef}.

\begin{corollary}
We have
\begin{align*}
J^{reg}(\tau,1) &= \CT_{s=1}\left[ \frac{\pi\zeta^*(s)}{12\zeta^*(2s-1)} P_{0,0}(\tau,\frac{s}{2}+\frac{1}{4})\right],\\
I^{reg}(\tau,1) &= J^{reg}(\tau,1) -B'_\infty(1) \Theta_K(\tau).
\end{align*}
\end{corollary}

We now compute the lift of $j_m$ for $m>0$ in terms of Poincar\'e series.
It follows from the Fourier expansion, Weil's bound and the properties
of the $I$-Bessel function that $G_{-m}(z,s)$ has a holomorphic continuation to $\Re(s)>3/4$.
The constant term of the Fourier expansion of $G_{-m}(z,s)$ is equal to
\[
\frac{4\pi m^{1-s} \sigma_{2s-1}(m)}{(2s-1)\zeta^*(2s)} y^{1-s},
\]
where $\sigma_{2s-1}(m)=\sum_{d\mid m} d^{2s-1}$, see e.g.~\cite{Ni}, \cite{Fay}.
We define
\begin{equation}\label{jm}
j_m(z,s)
:=G_{-m}(z,s)-\frac{4\pi m^{1-s}\sigma_{2s-1}(m)}{(2s-1)\zeta^*(2s-1)}G_0(z,s).
\end{equation}
%It follows from its Fourier expansion, Weil's bound and the properties
%of the $I$-Bessel function that $j_m(\t,s)$
This function has an analytic
continuation to $\Re(s) >3/4$. The constant term in its Fourier expansion is given by
$A_\infty(s)y^s+B_\infty(s)y^{1-s}$ with
\begin{align*}
A_\infty(s)&= -\frac{4\pi m^{1-s}\sigma_{2s-1}(m)}{(2s-1)\zeta^*(2s-1)},\\
B_\infty(s)&= 0. \end{align*}
Moreover, we have
\begin{equation}
\label{niej}
    j_m(z,1)=j_m(z).
\end{equation}
Hence, we may use the functions $j_m(z,s)$ as spectral deformations of the $j_m(z)$.

According to Theorem \ref{thm:liftp}, for $\Re(s)>1$, the lift of $j_m(z,s)$ is equal to
\begin{align*}
%\label{eq:lifte}
I^{reg}(\tau,j_m(\cdot,s))= \sqrt{\pi }\Gamma\left(\frac{s}{2}\right)
%(s-1) \frac{\G (s/2+1)}{2\sqrt{\pi}}
\sum_{n|m}
P_{\frac{m^2}{4n^2},\frac{m}{n}}(\tau, \frac{s}{2}+\frac{1}{4})   +A_\infty(s)
%\frac{2\pi m^{1-s}\sigma_{2s-1}(m)\zeta^*(s)}{(2s-1)\zeta^*(2s-1)}
\frac{\zeta^*(s)}{2}
P_{0,0}(\tau,\frac{s}{2}+\frac{1}{4}).
\end{align*}
By \eqref{eq:jres}, the right hand side has a holomorphic continuation to a neighborhood of $s=1$.
Its value at $s=1$ is equal to the regularized integral $J^{reg}(\tau,j_m)$.
Since $B_\infty(s)=0$, we obtain the following corollary to
Proposition \ref{prop:specdef}.

\begin{corollary}
For $m>0$
we have $I^{reg}(\tau,j_m)=J^{reg}(\tau,j_m)$.
\end{corollary}

\section{A Green function for $\varphi_0$}\label{sec:Green}

In this section, we introduce a Green function $\eta$ for the Schwartz function $\varphi_0$. Its properties will be the key for the proof of the results in Section~\ref{sec:Main-results}.

\subsection{The singular function $\eta$}

We first recall the definition of Kudla's Green function $\xi$ for $\varphi_1$ (see \cite[Section~11]{KAnn} and for our setting Remark~\ref{KM-remark}). It is defined for nonzero vectors $X \in V(\R)$ and given by
\begin{align}
\xi(X,\tau,z) &= v^{3/2}E_1(2\pi vR(X,z)) e(Q(X)\bar{\tau}) \\&= v^{3/2}\left(
\int_1^{\infty} e^{-2\pi vR(X,z) t} \frac{dt}{t} \right)
e(Q(X)\bar{\tau}). \notag
\end{align}
Here $E_1(w) =  \int_w^{\infty} e^{-t} \tfrac{dt}{t}$ with $w \in \C \back \R_{\leq 0}$ is the exponential
integral as in \cite{AbSt}. Since
\begin{equation}\label{E1formula}
E_1(w) = -\g - \log(w) + \int_{0}^w (1-e^{-t}) \tfrac{dt}{t},
\end{equation}
(the last function on the right hand is entire and is denoted by $\Ein(w)$) we directly see that $\xi$ has a logarithmic singularity for $z = D_X$ when $R(X,z)=0$ and is smooth for $Q(X) \geq 0$. Outside the singularity one has, see \cite{KAnn},
\begin{equation}\label{xi1}
dd^c\xi(X,\tau,z) =\varphi_{1}(X,\tau,z)d\mu(z),
\end{equation}
which can be also obtained via Lemma~\ref{BFDuke}. Here $d^c =
\tfrac{1}{4\pi i} (\partial - \bar{\partial})$ so that
\[
dd^c = -\tfrac1{2\pi i}
\partial \bar{\partial}= -\tfrac1{4\pi} \Delta_z d\mu(z).
\]
For the relationship between $\xi$ and $\varphi_1$ as currents, see Remark~\ref{Kudla-current}.
%Furthermore, one easily computes
%\begin{equation}\label{xi2}
%R_{-3/2} \xi(X,\tau,z) = -\varphi_0(X,\tau,z),
%\end{equation}
%see \cite{BFDuke}, \S 7. {\bf BUT IT LOOKS WE WON'T NEED THIS - MIGHT THIS BE OF ANY GOOD FOR US?!?}

We now define for $X \ne 0$ our Green function $\eta$ by
\begin{align}
\eta(X,\tau,z) &= \pi \left( \int_v^{\infty} E_1(2\pi
R(X,z)t ) e^{2\pi (X,X)t}
\frac{dt}{\sqrt{t}} \right)e(Q(X){\tau}).
% \notag \\
%& = -2\pi \sqrt{v}  \left( \int_{1}^{\infty} \Ei(2\pi vR(X,z) w^2)
%e^{2\pi v(X,X)w^2} dw \right) e((X,X){\tau}/2). \notag
\end{align}
We often drop the dependence on $\tau$ in the notation. We easily calculate
\begin{equation}\label{psi}
\partial_z \eta(X,z) = -\pi \sgn (X,X(z)) \frac{(X, X'(z))}{R(X,z)} \erfc \left( \sqrt{\pi v} |(X,X(z)|\right) e(Q(X){\tau}) dz,
\end{equation}
where $X'(z) = \tfrac{\partial}{\partial z} X(z)$ and $\erfc(t) = 1- \erf(t) = \tfrac2{\sqrt{\pi}} \int_t^{\infty} e^{-r^2} dr$ is the complimentary error function. Note that
\[
X'(z) = \frac{i}{2y} X(z)  + \frac1{\sqrt{N} y} \zxz{-\frac12}{\bar{z}}{ 0}{ \frac12}.
\]
For $Q(X) \ne 0$, we now analyze the singularities of $\eta$ in more detail.

\begin{lemma}\label{eta-sing}

\begin{itemize}

\item[(i)]
Let $X \in V(\R)$ such that $Q(X)=m<0$. Then $\eta$ has a logarithmic singularity
at $z=D_X$. More precisely, we have
\[
\tilde{\eta}(X,z) := \eta(X,z) +   \pi  \frac{\erfc(2\sqrt{\pi|m|v})}
{\sqrt{|m|}} e(m{\tau}) \log |z-D_X|^2
\]
is a smooth function in a neighborhood of $D_X$. Furthermore, $\eta(X)$ and its derivatives $\partial \eta(X)$, $\bar{\partial}\eta(X)$ are square exponential decreasing (in the coordinates $x,y$) at the boundary of $D$.

\item[(ii)]
Let $X \in V(\R)$ such that $Q(X)=m>0$. Then $\eta(X,z)$ is
differentiable, but {\it not} $C^1$. We have
\[
\partial \eta(X,z) = \frac{\pi i}2 {\sgn (X,X(z))} \erfc\left(\sqrt{\pi v} |(X,X(z))|\right) e(m{\tau}) dz_X,
\]
which is discontinuous at the cycles $c_X = \{ z \in D; \, (X,X(z))=0 \}$. Furthermore, assume that $\bar{\G}_X$ is infinitely cyclic. Then $\eta(X)$ and its derivatives $\partial \eta(X)$, $\bar{\partial}\eta(X)$ are square exponential decreasing (in the coordinates $x,y$) at the boundary of the `tube' $\G_X \back D$.

\end{itemize}

\end{lemma}

\begin{proof}

For (i), we have $m<0$. Via \eqref{E1formula} we therefore immediately see that
\[
\eta(X,z) +  \pi \log R(X,z) \left( \int_v^{\infty} e^{-4\pi|m| t}
\frac{dt}{\sqrt{t}} \right)e(m{\tau})
\]
is smooth. By \eqref{R-elliptic} the singularity of $\eta(X,z)$ at $z=D_X$ is hence given by
\[
-\pi  \left( \int_v^{\infty} e^{-4\pi|m| t}
\frac{dt}{\sqrt{t}} \right)e(m{\tau}) \log |z-D_X|^2= - \pi  \frac{\erfc(2\sqrt{\pi|m|v})}
{\sqrt{|m|}} e(m{\tau}) \log |z-D_X|^2.
\]
Since $E_1(w) \leq e^{-w}/w$, we have $
|\eta(X)| \leq \pi   \left(\int_{v}^{\infty} \frac{e^{-\pi (X,X(z))^2}}{R(X,z)} t^{-3/2} dt \right)e(m{\tau})$.
Now the growth behavior follows from $(X,X(z))^2 = \tfrac{N}{y^2}\left(x_3|z|^2 -2x_1Re(z)-x_2\right)^2$ for $X = \kzxz{x_1}{x_2}{x_3}{-x_2}$. Note that since $Q(X)<0$ we must have $x_3 \ne 0$.

By the $G$-equivariance properties of $\eta$, $X(z)$, and $dz_X$, it suffices to show (ii) for $X = \pm \kzxz{\sqrt{m/N}}{}{}{-\sqrt{m/N}}$. Then we have $(X,X(z)) = \mp 2x\sqrt{m}/y$, $R(X,z) = 2m|z|^2/y^2$, $(X,\partial X(z)) = i\sqrt{m} \bar{z}/y^2 dz $, and $dz_X = \pm dz/\sqrt{m}z$. Hence by \eqref{psi} we obtain
\begin{equation}\label{del-eta-formula}
\partial \eta(X) =  -\sgn(x) \frac{\pi i}{2\sqrt{m}} \erfc\left(2\sqrt{\pi vm} \tfrac{|x|}{y} \right) e(m{\tau})\frac{dz}{z},
\end{equation}
which is the asserted equality for this $X$. For this $X$, we have $\bar{\G}_X=  \left\langle \left( \begin{smallmatrix}
r&0\\0&r^{-1} \end{smallmatrix} \right) \right\rangle $ with some $r >1$. Hence $c_X$ is the imaginary axis and a fundamental domain for $\G_X \back D$ is given by the annulus $\{ z \in D; \,  1 \leq |z| < r \, \}$. Then \eqref{del-eta-formula} implies the very rapid decay in $\G_X \back D$.
\end{proof}

The analog to \eqref{xi1} is

\begin{proposition}\label{new-xi-prop}
Outside the singularities, we have
\[
dd^c \eta(X,\tau,z) = \varphi_0(X,\tau,z)d\mu(z).
\]
\end{proposition}

\begin{proof}%[Proof of Proposition~\ref{new-xi-prop}]

Using \eqref{xi1} we compute
\begin{align*}
dd^c \eta(X,\tau,z) & = \pi \left( \int_v^{\infty} dd^c
E_1(2\pi R(X,z)t ) e^{2\pi (X,X)t} \frac{dt}{\sqrt{t}}
\right)e^{\pi i (X,X){\tau}} \\
& = \pi \left( \int_v^{\infty}  t^{-3/2}\varphi_{1}(X,u+it,z)
e^{-\pi i(X,X)(u+it)}  \frac{dt}{\sqrt{t}}
\right)e^{\pi i(X,X){\tau}} d\mu(z)\\
& = - \left( \int_v^{\infty} t^{-2}\left(L_{\frac12}
\varphi_0(X,u+it,z) \right)e^{-\pi i(X,X)(u+it)}dt
\right)e^{\pi i(X,X){\tau}} d\mu(z)\\
& = - \left( \int_v^{\infty} \frac{\partial}{\partial t}
\left[ \sqrt{t} e^{-\pi (X,X(z))^2 t} \right] dt \right) e^{\pi i(X,X){\tau}}d\mu(z) \\
& = \varphi_0(X,\tau,z) d\mu(z).
\end{align*}
Here we used $\varphi_0(X,u+it,z) =  \sqrt{t} e^{-\pi (X,X(z))^2 t} e^{\pi i(X,X)(u+it)}$.
%\begin{align*}
%\frac{1}{\sqrt{v}} & \left( R_{-\frac12} \varphi(X,\tau,z) \right)
%e^{-\pi i(X,X)\bar{\tau}} \\  & \quad = \frac1{\sqrt{v}}\left(2i
%\frac{\partial}{\partial \tau} -\frac12 v^{-1}\right)\left( v
%e^{-2\pi (X,X(z)^2)v} e^{\pi i (X,X)
%\bar{\tau}} \right) e^{-\pi i (X,X)\bar{\tau}}\\
%& \quad = \frac1{2\sqrt{v}} e^{-2\pi (X,X(z)^2)v} +
%\sqrt{v}\frac{\partial}{\partial v} e^{-2\pi (X,X(z)^2)v} \\
%& \quad = \frac{\partial}{\partial v}\left( {\sqrt{v}} e^{-2\pi
%(X,X(z)^2)v} \right).
%\end{align*}
\end{proof}

The relationship to Kudla's Green function is given by

\begin{lemma}\label{raising-eta}
Outside the singularity $D_X$, we have
\[
L_{\frac12} \eta(X,\tau,z) = - \pi \xi(X,\tau,z).
\]
\end{lemma}

\begin{proof}
We compute
\begin{align*}
 L_{\frac12} \eta(X,\tau,z)
& = -2\pi i v^2\frac{\partial}{\partial \bar{\tau}} \left( \int_v^{\infty} E_1(2\pi
R(X,z)t ) e^{2\pi (X,X)t}
\frac{dt}{\sqrt{t}} \right)e(Q(X){\tau}) \\
& \quad = \pi v^2\left( \frac{\partial}{\partial v}
\int_v^{\infty} E_1(2\pi R(X,z)t ) e^{2\pi (X,X)t}
\frac{dt}{\sqrt{t}}
\right)e(Q(X){\tau}) \\
& \quad = -\pi \xi(X,\tau,z),
\end{align*}
as claimed.
\end{proof}

To summarize we have obtained the following diagram
\begin{align}
\xymatrix{
 \eta(X,\tau,z) \;\ar @{|->}[r]^{-\tfrac1{\pi}L_{1/2}} \ar @{|->}[d]^{dd^c} & \xi(X,\tau,z)  \ar @{|->}[d]^{dd^c}\\
\varphi_{0}(X,\tau,z) d\mu(z) \;\ar @{|->}[r]^{-\tfrac1{\pi}L_{1/2}} &  \varphi_1(X,\tau,z)d\mu(z).}
\end{align}

\subsection{Current equations}\label{current-equations}

We now consider $\eta$ as a current. The current equations we obtain for $\eta$ can be viewed as a refinement of Kudla's current equation for $\xi$, see \cite{KAnn}, Proposition~11.1. Namely, for $X\ne 0$ we have
\begin{equation}\label{Kudla-current}
dd^c[\xi(X,\tau)] + v^{3/2}e(m\bar{\tau})\delta_{D_X}= [\varphi_{1}(X,\tau)d\mu(\tau)],
\end{equation}
as currents acting on functions with compact support on $D$. Here $D_X = \emptyset$ if $Q(X) \geq 0$. We recover \eqref{Kudla-current} by applying the lowering operator $L_{1/2}$ to the current equations for $\eta$ below.
%\begin{proof} of Kudla's current equation
%This is easy from \eqref{KM-def}, Lemma~\ref{raising-eta} and
%follows for $(X,X)>0$ from
%\[
%L_{\tfrac12} \left(\pi \frac{\erfc(2\sqrt{\pi |m|v})}
%{2\sqrt{|m|}} e(m{\tau})\right) = - \pi v^{3/2}e(m\bar{\tau}),
%\]
%which one easily obtains from
%\begin{equation}\label{erfAbl}
%\frac{\partial}{\partial v} \erfc(a\sqrt{v}) = -
%\frac{\partial}{\partial v} \erf(a\sqrt{v}) = - a \frac1{\sqrt{\pi
%v}} e^{-a^2v}.
%\end{equation}
%For $(X,X)>0$, it follows from $L_{\frac12} \left(e(m{\tau}) \right) =0$.
%\end{proof}

\medskip

We first note that by Proposition~\ref{new-xi-prop} for a $C^2$-function $f$ on $D$ we have
\begin{equation} \label{P-L-eq}
2\pi i f(z)  \varphi_0(X,z) d\mu(z) =  d \left( f(z) \partial \eta(X,z) \right) -  d \left( \bar{\partial} f(z) \eta(X,z)  \right)+ \partial \bar{\partial} f(z) \eta(X,z),
\end{equation}
away from the singularities of $\eta$.

\subsubsection{The elliptic case}

Throughout this subsection we assume that $X \in V$ is a vector of length $Q(X) =m<0$. Then the stabilizer $\bar{\G}_X$ of $X$ is finite.

\begin{proposition}\label{big-current-elliptic}
The function $\eta(X,\tau,z)$ satisfies the following current
equation:
\[
dd^c[\eta(X,\tau)]  + \pi  \frac{\erfc(2\sqrt{\pi|m|v})}
{2\sqrt{|m|}} e(m{\tau}) \delta_{D_X} = [\varphi_0(X,\tau)]
\]
as currents on $C^2$-functions  on $D$ with at most "linear
exponential" growth. That is, for such $f$ we have
\begin{align*}
\int_D  f(z) \varphi_0(X,\tau, z) d\mu(z) &
= \frac{\pi  e^{2 \pi m{\tau}}}{2\sqrt{|m|}} \erfc(2\sqrt{\pi |m|v})f(D_X) \\
&\quad -\frac1{4\pi}  \int_D \left(\Delta f(z)\right)
\eta(X,\tau,z) d\mu(z).
\end{align*}
%Here $\erfc(x) = 1- \erf(x) = \tfrac2{\sqrt{\pi}} \int_x^{\infty} e^{-r^2} dr$ the
%complimentary error function %and $\erf(x) = \tfrac2{\sqrt{\pi}}
%\int_0^x e^{-r^2} dr$ is the error function. Note $\erfc(2\sqrt{\pi
%|m|v}) = O(e^{-4\pi |m|v})$ as $v \to \infty$.
\end{proposition}

\begin{proof}
For functions with compact support this can be easily seen using \eqref{P-L-eq}, Stokes' theorem,
 and the logarithmic singularity of $\eta$. In fact, it is very special case of the Poincar\'e-Lelong Lemma, see e.g.\cite{SABK} p.41/42. For functions with at most linear exponential growth the same argument goes through since $\eta(X)$ and its derivatives are square exponentially decreasing.
\end{proof}

This holds in particular for $f$ a weak Maass form of weight $0$.

\begin{corollary}\label{big-current-elliptic1}
Let $f \in H_0^+(\G)$. Then
\[
\int_M f(z) \sum_{ \g \in \G_X \back \G} \varphi_0(X,\tau,\g z) d\mu(z)
\]
converges, and we have
\begin{align*}
\int_M f(z) \sum_{ \g \in \G_X \back \G} \varphi_0(X,\tau,\g z) d\mu(z) &= \frac1{|\bar{\G}_X|}
\int_D  f(z) \varphi_0(X,\tau, z) d\mu(z) \\& =  \frac{\pi  e^{2 \pi m{\tau}}}{2\sqrt{|m|}} \erfc(2\sqrt{\pi |m|v})\frac1{|\bar{\G}_X|} f(D_X).
\end{align*}
\end{corollary}

\begin{proof}
This is immediate from Proposition~\ref{big-current-elliptic} and the linear exponential growth of weak Maass forms.
\end{proof}

\subsubsection{The non-split hyperbolic case}

Throughout this subsection we assume that $X \in V$ is a vector of positive length $Q(X) =m>0$. In addition, we assume that the stabilizer $\bar{\G}_X$ is infinitely cyclic.

\begin{proposition}\label{big-current-hyperbolic}

For $X$ as above, the function $\eta(X,\tau,z)$ satisfies the following current
equation:
\[
dd^c[\eta(X,\tau)]  +  \frac12 e(m{\tau})
\delta_{c(X),dz_X} = [\varphi_0(X,\tau)].
\]
as currents on $C^2$-functions on $\G_X \back D$ with at most linear
exponential growth. That is, for such $f$ we have
\begin{align*}
\int_{\G_X \back D}  f(z) \varphi_0(X,\tau, z) d\mu(z)  &=
\frac12 e(m{\tau}) \int_{c(X)} f(z) dz_X
\\ & \quad
 -\frac1{4\pi}   \int_{\G_X \back D} \left(\Delta f(z)\right)
\eta(X,\tau,z) d\mu(z).
\end{align*}
\end{proposition}

\begin{proof}
We can assume that $X = \sqrt{m/N} \left(
\begin{smallmatrix} 1&\\&-1 \end{smallmatrix} \right)$ so that $c_X = \{ z \in D; \, (X,X(z))= - \tfrac{x}{y}=0 \}$ is the imaginary axis,
$\bar{\G}_X=  \left\langle \left( \begin{smallmatrix} r&0\\0&r^{-1} \end{smallmatrix} \right)
\right\rangle $, and $\G_X \back c_X$ inside the annulus $\G_X\back D = \{ z \in D; \, 1 \leq |z| \leq r \}$ is given by $\{z =iy; \, 1 \leq y \leq r\}$. %Since
%$R(X,z) = 2m|z|^2/y^2$, we see that $\int_{\G_X \back D}  f(z) \varphi_0(X,\tau, z) d\mu(z)$ converges for $f$ functions with linear exponential growth in the tube $\G_X \back c_X$.
We define an $\eps$-neighborhood for $c_X$ in $\G_X \back D$ by $U_{\eps}(c_X) = \{ z \in \G_X \back D; |(X,X(z)| = \tfrac{|x|}{y}< \eps \}$.  We have
\begin{align*}
\int_{\G_X \back D}  f(z) \varphi_0(X,\tau, z) d\mu(z) &= \lim_{\eps \to 0} \int_{\G_X \back D-U_{\eps}(c_X)}  f(z) \varphi_0(X,\tau, z) d\mu(z).
\end{align*}
Using \eqref{P-L-eq} we obtain for fixed $\eps$
\begin{align*}
 -\frac{1}{2\pi i} \int_{\partial U_{\eps}(c_X)} \hskip-.3cm f(z) \partial \eta(X,z) +  \frac{1}{2\pi i} \int_{\partial U_{\eps}(c_X)} \hskip-.3cm  \bar{\partial} f(z) \eta(X,z) +
 \frac{1}{2\pi i} \int_{\G_X \back D-U_{\eps}(c_X)} \hskip-.3cm   \partial \bar{\partial} f(z) \eta(X,z).
\end{align*}
(Note $\partial U_{\eps}(c_X) = - \partial (\G_X \back D-U_{\eps}(c_X))$). Here we also used the very rapid decay of $\varphi_0$, $\eta$, and $\partial \eta$ at the boundary of the tube $\G_X \back D$. As $\eps \to 0$ the last term becomes $\int_{\G_X \back D} \left(dd^c f(z)\right) \eta(X,\tau,z) d\mu(z)$. The second term vanishes since $\eta$ is continuous. For the first term, we define $c_{X, \pm \eps} = \{ z \in \G_X \back D; \, - (X,X(z)= \frac{x}{y} = \pm \eps \}$.
Then by \eqref{del-eta-formula} we obtain
\[
 \int_{\partial U_{\eps}(c_X)} \hskip-.3cm f(z) \partial \eta(X,z) = \left[\int_{ c_{X,\eps} } f(z) dz_X + \int_{c_{X,-\eps} } f(z) dz_X \right]\erfc\left(\sqrt{\pi m v} \eps \right) e(m{\tau}).
\]
Taking the limit completes the proof.
\end{proof}

This holds in particular for $f$ a weak Maass form of weight $0$.

\begin{corollary}\label{big-current-hyperbolic1}
Let $f \in H_0^+(\G)$. Then
\[
\int_M f(z) \sum_{ \g \in \G_X \back \G} \varphi_0(X,\tau,\g z) d\mu(z)
\]
converges, and we have
\begin{align*}
\int_M f(z) \sum_{ \g \in \G_X \back \G} \varphi_0(X,\tau,\g z) d\mu(z) &=
\int_{\G_X \back D}  f(z) \varphi_0(X,\tau, z) d\mu(z) \\
&=  \frac12 e(m{\tau}) \int_{c(X)} f(z) dz_X.
\end{align*}
\end{corollary}

\begin{proof}
This is immediate from Proposition~\ref{big-current-hyperbolic} and the linear exponential growth of weak Maass forms.
\end{proof}

\begin{remark}[The split hyperbolic case]
Assume that $X \in V$ is a vector of positive length $Q(X) =m>0$ such that the stabilizer $\bar{\G}_X$ is trivial. Hence $\G_X \back D = D$. Then Proposition~\ref{big-current-hyperbolic} carries to the present situation over if one assumes that $f$ is a function of compact support on $D$. However, for a function $f$ not of sufficient decay,
% we have that
\[
\int_M f(z) \sum_{ \g \in \bar\G} \varphi_0(X,\tau,\g z) d\mu(z)
\]
does {\it not} converge. In fact, exactly these terms require the theta lift to be regularized.
\end{remark}

\section{The Fourier expansion of the regularized theta lift}\label{sec:Fourier}

In this section, we give the proofs for the results stated in Section~\ref{sec:Main-results}. We set
\begin{equation}\label{Fourier1}
\theta_{m,h}(\tau,z) = \sum_{X \in {L}_{m,h}} \varphi_0(X,\tau,z),
\end{equation}
which defines a $\G$-invariant function on $D$. We then have
\begin{align}\label{Fourierexp}
I_{h}(\tau,f) = \sum_{m\in\Q} \int_{M}^{reg}  f(z) \theta_{m,h}(\tau,z)
d\mu(z),
\end{align}
which is the Fourier expansion of $I_{h}(\tau,f)$. (Since picking out
the $m$-th Fourier coefficient is achieved by integrating over a
circle, we can interchange the regularized integral with the `Fourier
integral'). More precisely, let $f \in H_0^+(\G)$ be a harmonic weak
Maass form for $\G$ with constant terms $a^+_{\ell}(0)$ at the cusp
$\ell$. Then by Proposition~\ref{prop:reg2} the $m$-th Fourier
coefficient of the regularized lift is given by
\begin{multline}\label{reg-int-m}
\int^{reg}_{M} f(z) \theta_{m,h}(\tau,z) d\mu(z) \\= \lim_{T \to \infty} \left[ \int_{M_T} f(z)  \theta_{m,h}(\tau,z)
d\mu(z) - \frac{\log(T)}{\sqrt{N}}\sum_{\ell \in \G \back \Iso(V)} a^+_{\ell}(0) \eps_{\ell} b_{\ell}(m,h) \right].
\end{multline}
Here $M_T$ is the truncated surface $M_T$ defined in \eqref{truncated} and $b_{\ell}(m,h)$ is the $(m,h)$-Fourier coefficient of the unary theta series $\tilde \Theta_{K_{\ell}}(\tau)$ as before. For $m \ne 0$, the set $\G \back {L}_{m,h}$ is finite. Therefore, for these $m$, we see
\begin{multline}\label{reg-int-m2}
\int^{reg}_{M} f(z) \theta_{m,h}(\tau,z) d\mu(z)  \\= \lim_{T \to \infty} \left[ \sum_{X\in \G \backslash {L}_{m,h} } \hskip-.15cm\int_{M_T}  \hskip-.15cm f(z) \hskip-.15cm \sum_{\g \in \G_X \backslash \G}  \varphi_0(X,\tau,\g z) d\mu(z) -  \frac{\log(T)}{\sqrt{N}} \hskip-.15cm\sum_{\ell \in \G \back \Iso(V)} a^+_{\ell}(0) \eps_{\ell} b_{\ell}(m,h)\right].
\end{multline}

For non-zero $m$ in the elliptic and the split hyperbolic situation we have seen in Section~\ref{current-equations} that
\[
\int_{M} \sum_{\g \in \G_X \backslash \G} f(z) \varphi_0(X,\tau,\g z) d\mu(z)
\]
actually converges, corresponding to the fact that  $b_{\ell}(m,h)=0$. Then the current equations in Section~\ref{current-equations}, Corollaries~
\ref{big-current-elliptic1} and \ref{big-current-hyperbolic1}, give the Fourier coefficients for Theorem~\ref{th:Main-Maass} for those $m$. We will consider the split hyperbolic periods below, as well as the $0$-th coefficient.

\subsection{The split hyperbolic Fourier coefficients}

We now consider the case $m/N$ is a square, when the associated cycles are infinite geodesics. Throughout $X \in V$ denotes a vector of length $Q(X) = m$ with $m/N$ is a square and $f \in H_0^+(\G)$ is a  harmonic weak Maass form.

In view of the characterization of the regularized integral in \eqref{reg-int-m2}, we need to consider the behavior of $\int_{M_T} \sum_{\g \in \bar\G} f(z) \varphi_0(X,\tau,\g z) d\mu(z)$ as $T \to \infty$. For this we have

\begin{proposition}\label{key-split-hyperbolic}
The asymptotic behavior of
\[
\int_{M_T} \sum_{\g \in \bar\G} f(z) \varphi_0(X,\tau,\g z) d\mu(z)
\]
as $T \to \infty$ is given by
\begin{align*}
& \frac{1}{2} \left( \int^{reg}_{c_X} f(z) dz_X \right)e(m\tau)\\
 &;-  \frac{1}{2\sqrt{m}} \Biggl[ \left( \log 2 \sqrt{\pi v m} + \tfrac12 \log2 + \tfrac14 \g- \sqrt{\pi} \int_0^{2\sqrt{\pi vm}} e^{w^2} \erfc(w)dw \right) (a^+_{\ell_X}(0)+ a^+_{\ell_{-X}}(0))  \\
& +2\pi \left( \int_0^{\sqrt{v}} e^{4\pi m w^2} dw \right) \left( \sum_{n<0} a^+_{\ell_X}(n)  e^{2\pi i Re(c(X))n } + a^+_{\ell_{-X}}(n)  e^{2\pi i \re(c(-X))n } \right) \\
& + (a^+_{\ell_X}(0)+ a^+_{\ell_{-X}}(0))\log(T) \Biggr] e(m\tau).
\end{align*}
\end{proposition}

Before we prove the proposition, we first show how this implies the formula for this Fourier coefficient given in Theorem~\ref{th:Main-Maass}. In view of \eqref{reg-int-m2} we only need to show

\begin{lemma}
\begin{equation}\label{logT-matching}
  \frac{1}{2 \sqrt{m}} \sum_{X\in \G \backslash {L}_{m,h} }  (a^+_{\ell_X}(0)+ a^+_{\ell_{-X}}(0)) =  \frac{1}{\sqrt{N}}\sum_{\ell \in \G \back \Iso(V)} a^+_{\ell}(0) \eps_{\ell} b_{\ell}(m,h).
\end{equation}

\end{lemma}

\begin{proof}
We can sort the infinite geodesics by the cusps $\ell$ to which they go. We define $\delta_{\ell}(m,h)$ to be $1$ if there exists a $X \in L_{m,h}$ such that $c_X$ ends at the cusp $\ell$, that is, if $X$ is perpendicular to $\ell$. By \cite{FCompo}, Lemma~3.7, there are either no or $2 \sqrt{m/N} \eps_{\ell}$ many $X$ in $\G \backslash {L}_{m,h}$ such that the corresponding $c_X$ all end in $\ell$. Hence the left hand side of \eqref{logT-matching} is equal to
\[
\frac1{\sqrt{N}} \sum_{\ell \in \G \back \Iso(V)} \eps_{\ell}(\delta_{\ell}(m,h) + \delta_{\ell}(m,-h)) a^+_{\ell}(0)= \frac1{\sqrt{N}} \sum_{\ell \in \G \back \Iso(V)} \eps_{\ell} b_\ell(m,h)a^+_{\ell}(0).
\]
%But $\delta_{\ell,h}(m)$ is exactly $1$ if there exists a vector in $K_{\ell,h}$ of length $m$, so $b_{\ell}(m,h)=\delta_{\ell}(m,h) + \delta_{\ell}(m,-h)$. The lemma follows.
This proves the lemma.
\end{proof}

The remainder of the section will be devoted to the proof of Proposition~\ref{key-split-hyperbolic}. We begin with

\begin{lemma}\label{big-current-split-hyperbolic1}
Let $f \in H_0^+(\G)$. Then
\begin{align*}
\int_{M_T} f(z) \sum_{ \g \in \bar\G} \varphi_0(X,\tau,\g z) d\mu(z)& =  \frac12 e(m{\tau}) \int_{
c^T_X} f(z) dz_X \\
& \quad + \frac{1}{2\pi i} \int_{\partial M_T} f(z) \sum_{ \g \in  \bar\G} \partial \eta(X,\tau, \g z) \\
& \quad + \frac{1}{2\pi i} \int_{\partial M_T} \bar{\partial} f(z) \sum_{ \g \in  \bar\G} \eta(X,\tau, \g z).
\end{align*}
Here $c^T_X = c_X \cap M_T$.
\end{lemma}

\begin{proof}
We proceed as in the proof of Proposition~\ref{big-current-hyperbolic}. Since in a truncated fundamental domain for $\G$, the only singularities of  $\sum_{ \g \in \bar\G} \partial \eta (X,\tau,\g z)$ are along $c^T_X$, everything goes through as before except that one obtains in addition the boundary terms above.
\end{proof}

%To study $\int_{M_T} \sum_{\g \in \bar\G} f(z) \varphi_0(X,\tau,\g z) d\mu(z)$ we use the above Lemma.

As $f \in H_0^+(\G)$, we have that $\bar{\partial}f(z)$ is rapidly decreasing and hence

\begin{lemma}\label{partialF-behavior}
\[
\frac1{2\pi i} \lim_{T \to \infty}  \int_{\partial M_T} \bar{\partial} f(z) \sum_{ \g \in  \bar\G} \eta(X,\tau, \g z) =0.
\]
\end{lemma}

The main work is to consider
\begin{equation}\label{hard-term}
\frac1{2\pi i}\int_{\partial M_T} f(z) \sum_{ \g \in  \bar\G} \partial \eta(X,\tau, \g z).
\end{equation}
By arguments exactly analogous to \cite{BFCrelle}, Lemma~5.2 (where the integral of $f$ against $\partial \xi(X)$ is considered), we see that the asymptotic behavior of \eqref{hard-term} as $T \to \infty$ is the same as the one of
\begin{align}\label{gisela}%sic in [BFCrelle]
 \frac{1}{2\pi i} \int_{\partial M_{T,\ell_X}  }
 f ( z) \sum_{ \g \in \bar\G_{\ell_X}} \partial \eta(X, \tau, \g  z) +
\frac{1}{2\pi i} \int_{\partial M_{T,{\ell}_{-X}}  }
 f(z)  \sum_{ \g \in \bar\G_{{\ell}_{-X}}} \partial \eta ( X,  \tau,  \g z).
\end{align}
Here $\partial M_{T,\ell_X},{\partial M_{T,{\ell}_{-X}} }$ are the boundary components of $M_T$ at the cusps $\ell_X,\tilde\ell_X= \ell_{-X}$ respectively. All other terms are rapidly decaying.
%We assume from now on that $f=f^+$ is weakly holomorphic, leaving the necessary adjustments for $f \in H^+_0(\G)$ to the reader. Since $f$ is rapidly decreasing at the cusp $\ell_X$, we get vanishing for this piece in the limit as well. For that reason we assume from now on that $f=f^+$ is weakly holomorphic. The reader will make the necessary adjustments to obtain the Lemma~\ref{partialeta-behavior}

The key is the asymptotic behavior of $ \sum_{ \g \in \bar\G_{\ell_X}} \partial \eta(X, \tau, \g  z)$.
\begin{lemma}\label{l:key-formula}
Let $r \in \Q $ be the real part of the geodesic $c(X)$ and write $\alpha=\alpha_{\ell_X}$. Define for a nonzero integer $n$ the function $g(n,y)$ by
\[
g(n,y) = \int_{1}^{\infty} e^{4\pi vm/w^2}\left(
 e^{-\left(\tfrac{2\sqrt{\pi v m}}{w} + \tfrac{\pi n y w}{2 \sqrt{\pi v m} \alpha}\right)^2} -
\tfrac{2\pi \sqrt{vm}}{w} \erfc\left(\tfrac{2\sqrt{\pi v m}}{w} + \tfrac{\pi n y w}{2 \sqrt{\pi v m} \alpha} \right)
\right) \frac{dw}{w}.
\]
Then for $r < \Re(\sigma_{\ell_X}z) \leq r +\alpha$ we have
\begin{multline*}
\frac{1}{2\pi i} \sum_{ \g \in \bar\G_{{\ell}_X}} \partial \eta ( X,  \tau,  \g \sigma_{\ell_X} z) e( -m \tau) \\ =
 \frac{1}{2 \sqrt{m}\alpha}  \Biggl[ \sum_{n \ne 0}  g(n,y) e(-n(z-r)/\alpha)
-   \sqrt{\pi} \int_0^{2\sqrt{\pi vm}} e^{w^2} \erfc(w)dw + \log(2 \sqrt{\pi v m}/y)   \\
+ \tfrac12 \psi((z-r)/\alpha) + \tfrac12 \psi(1- (z-r)/\alpha)+ \log \alpha + \tfrac12 \log(2) + \tfrac14 \g\Biggr] dz.
\end{multline*}
\end{lemma}

\begin{proof}

By applying $\sigma_{\ell_X}$ we can assume that $X =\sqrt{m/N}\kzxz{1}{-2r}{0}{-1}$ so that $c_X = \{ z \in D; \, \Re(z) = r \}$ is a vertical geodesic. Hence $\ell_X$ represents the cusp $\infty$ and $\bar\G_{\ell_X} = \left\{ \kzxz{1}{\alpha k}{0}{1}; \, k \in \Z \right\}$ with $\alpha = \alpha_{\ell_X}$. Then (see also \eqref{del-eta-formula})
\[
\frac{1}{2 \pi i} \partial \eta(X,z) = -\sqrt{v} \frac{x-r}{y} \frac{1}{z-r} \left(\int_1^{\infty} e^{-4 \pi vm \frac{(x+r)^2}{y^2} w^2}  dw\right)  e(m\tau) dz.
\]
Replacing $z$ by $z+r$, we can assume $r=0$. We set for $s \in \C$
\[
\omega(z,s) = -\sqrt{v} \frac{x}{y} \frac{1}{z} \int_1^{\infty} e^{-4 \pi vm \frac{x^2}{y^2} w^2} w^s dw,
\]
so that
\[
\frac{1}{2 \pi i} \partial \eta(X,z) = \omega(z,0) e(m\tau) dz.
\]
We also define
\[
\Omega(z,s) := \sum_{n \in \Z} \omega(z+\alpha n,s),
\]
so that
\[
 \frac{1}{2 \pi i} \sum_{ \g \in \bar\G_{{\ell}_X}} \partial \eta ( X,  \tau,  \g z) = \Omega(z,0) e(m\tau) dz.
\]
Note that $\Omega(z,s)$ is a holomorphic function in $s$. For $\Re(s) > - 1$, we write
\begin{align}
\Omega(z,s)& \notag \\ & = -\sqrt{v} \frac{1}{y} \sum_{n \in \Z} \frac{x+\alpha n}{z+ \alpha n} \int_0^{\infty} e^{-4 \pi vm \frac{(x+ \alpha n)^2}{y^2} w^2} w^s dw  \label{O1} \\
& \quad + \sqrt{v} \frac{1}{y} \sum_{n \in \Z} \frac{x+\alpha n}{z+ \alpha n} \int_0^1 e^{-4 \pi vm \frac{(x+ \alpha n)^2}{y^2} w^2} w^s dw. \label{O2}
\end{align}
For the term \eqref{O1}, we compute
\begin{equation}\label{O11}
-2^{-s-2} v^{-s/2} (\pi m)^{-(s+1)/2} \alpha^{-s-1}y^s \G\left( \tfrac{s+1}{2}\right) \sum_{n \in \Z} \frac{\sgn(x+\alpha n)}
{(\tfrac{z}{\alpha} + n) |\tfrac{x}{\alpha} + n|^s}
\end{equation}
Since $0<x \leq \alpha$, we see
\begin{multline}
\sum_{n \in \Z} \frac{\sgn(x+\alpha n)}{(\tfrac{z}{\alpha} + n) |\tfrac{x}{\alpha} + n|^s}
 = \sum_{n =0}^{\infty } \frac{1}{(\tfrac{z}{\alpha} + n) |\tfrac{x}{\alpha} + n|^s} + \sum_{n =0}^{\infty} \frac{1}{(n +(1-\tfrac{z}{\alpha}) |n+(1-\tfrac{x}{\alpha})|^s}.
\end{multline}
Now (using \eqref{digamma})
\[
\lim_{s \to 0^+} \sum_{n =0}^{\infty } \frac{1}{(w + n) |w'+n|^s} - \frac{1}{(w'+n)^{s+1}} = \sum_{n=0}^{\infty} \frac{1}{w+n} - \frac{1}{w'+n} = -\psi(w) + \psi(w').
\]
Since the constant term of the Laurent expansion of the Hurwitz zeta-function $H(w',s)$ at $s=1$ is $-\psi(w')$, we conclude
\[
\sum_{n =0}^{\infty } \frac{1}{(w + n) |w'+n|^s} = \frac{1}{s} - \psi(w) + O(s).
\]
Via \eqref{O11} we therefore easily see
\begin{multline}\label{O111}
 -\sqrt{v} \frac{1}{y} \sum_{n \in \Z} \frac{x+\alpha n}{z+ \alpha n} \int_0^{\infty} e^{-4 \pi vm \frac{(x+ \alpha n)^2}{y^2} w^2} w^s dw \\
 = \frac{1}{2 \sqrt{m}\alpha} \left( -\frac1{s} + \tfrac12 \psi\left(\tfrac{z}{\alpha}\right) +  \tfrac12 \psi\left(1-\tfrac{z}{\alpha}\right) -\frac{\G'\left(\tfrac12\right)}{4\sqrt{\pi}}+ \log\left(\tfrac{2 \sqrt{\pi v m} \alpha}{y}\right) \right) + O(s).
\end{multline}
Note $\G'(1/2)= -\sqrt{\pi}(2 \log(2) + \g)$.

For \eqref{O2}, we first substitute $w \to 1/w$ in the integral and obtain
\begin{equation}\label{O22}
 \sqrt{v} \frac{1}{y} \sum_{n \in \Z} \frac{x+\alpha n}{z+ \alpha n} \int_1^{\infty} e^{-4 \pi vm \frac{(x+ \alpha n)^2}{y^2 w^2}} w^{-2-s} dw.
\end{equation}
Now we apply Poisson summation. Using \cite{BFCrelle}, Lemma~5.1, we see that \eqref{O22} is equals
\[
 \frac{1}{2 \sqrt{m}\alpha} \sum_{n \in \Z}  g(n,y,s) e(-nz/\alpha)
\]
with
\begin{multline*}
g(n,y,s) \\ = \int_{1}^{\infty} e^{4\pi vm/w^2}\left(
 e^{-\left(\tfrac{2\sqrt{\pi v m}}{w} + \tfrac{\pi n y w}{2 \sqrt{\pi v m} \alpha}\right)^2} -
\tfrac{2\pi \sqrt{vm}}{w} \erfc\left(\tfrac{2\sqrt{\pi v m}}{w} + \tfrac{\pi n y w}{2 \sqrt{\pi v m} \alpha} \right)
\right) w^{-s} \frac{dw}{w}.
\end{multline*}
For $n \ne 0$, the function $g(n,z,s)$ is holomorphic at $s=0$, while for $n=0$ we have
\begin{align*}
g(0,z,s) &= \int_1^{\infty}w^{-s-1}dw -  2\pi \sqrt{vm} \int_{1}^{\infty} e^{4\pi vm/w^2}\erfc\left(\tfrac{2\sqrt{\pi v m}}{w} \right)w^{-s-2}dw \\
&= \frac{1}{s} - \sqrt{\pi} \int_0^{2\sqrt{\pi vm}} e^{w^2} \erfc(w)dw + O(s).
\end{align*}
Combining this with \eqref{O111} completes the proof of Lemma~\ref{l:key-formula}.
\end{proof}

Lemma~\ref{l:key-formula} now immediately gives

\begin{lemma}\label{l:111}
Let $r \in \Q $ be the real part of the geodesic $c(X)$ and write $\alpha=\alpha_{\ell_X}$. Let $f(\sigma_{\ell_X}z) = \sum_{n \in \Z} a^+_{\ell_X}(n) e^{2\pi i nz/\alpha} + \sum_{n <0} a^-_{\ell_X}(n) e^{2\pi i n\bar{z}/\alpha}$ be the Fourier expansion of $f$ at the cusp $\ell_X$. Then
\begin{align*}
&\frac{1}{2\pi i} \int_{\partial M_{T,\ell_X}  }
 f ( z) \sum_{ \g \in \bar\G_{\ell_X}} \partial \eta(X, \tau, \g  z) \notag  \\ &=
 - \frac{1}{2 \sqrt{m}} \sum_{n \ne 0}  g(n,T)  e(n r/\alpha) \left( a^+_{\ell_X}(n) + a^-_{\ell_X}(n) e^{4\pi n y} \right)
 \\
 & \quad -\frac{1}{4 \sqrt{m}\alpha} \int_{z=iT}^{iT+\alpha}
 \bigl[ \psi(z/\alpha) + \psi(1-z/\alpha) + 2\log \alpha \bigr] f(\sigma_{\ell_X}(z+r)) dz \\
  &\quad-\frac{1}{2 \sqrt{m}} \left( \tfrac12 \log(2) + \tfrac14 \g+ \log(2 \sqrt{\pi v m} /T)- \sqrt{\pi} \int_0^{2\sqrt{\pi vm}} e^{w^2} \erfc(w)dw \right) a^+_{\ell_X}(0).
\end{align*}
\end{lemma} \qed

For the first term in Lemma~\ref{l:111}, using
$\erfc(t)= O(e^{-t^2})$ as
$t \rightarrow \infty$ and
$\erfc(t)= 2$ as $t\rightarrow-\infty$,
 we easily see

\begin{lemma}\label{l:112}
\begin{multline*}
 - \frac{1}{2 \sqrt{m}} \lim_{T \to \infty }\sum_{n \ne 0}  g(n,T) e(n r/\alpha) \left( a^+_{\ell_X}(n) + a^-_{\ell_X}(n) e^{4\pi n y} \right) \\ = 2\pi \left( \int_0^{\sqrt{v}} e^{4\pi m w^2} dw \right) \left( \sum_{n <0} a^+_{\ell_X}(n) e(n r/\alpha)\right).
\end{multline*}
\end{lemma} \qed

Summarizing, in view of \eqref{gisela}, using Lemmas~\ref{l:111}, \ref{l:112}, and Theorem~\ref{reg-def2} we finally obtain the asymptotic behavior of \eqref{hard-term}. Namely,

\begin{lemma}\label{partialeta-behavior}
%Let $f(z) = \sum_{N \in \Z} a_{\ell_X}(N) e^{2\pi i Nz/\alpha}$ be the Fourier expansion of $f$ at the cusp $\ell_X$.
The asymptotic behavior of
\[
\frac{1}{2\sqrt{m}} \int_{c_X^T} f(z) dz_X +  \frac{1}{2\pi i} \int_{\partial M_{T}  } f ( z) \sum_{ \g \in \bar\G} \partial \eta(X, \tau, \g  z) e(-m\tau)
\]
as $T \to \infty$ is given by
\begin{multline*}
\frac{1}{2} \int_{c_X}^{reg} f(z) dz_X \\
-\frac{1}{2 \sqrt{m}} \left( \log \tfrac{2 \sqrt{\pi v m}}{T} + \tfrac12 \log 2 + \tfrac14 \g- \sqrt{\pi} \int_0^{2\sqrt{\pi vm}} e^{w^2} \erfc(w)dw \right) (a^+_{\ell_X}(0)+ a^+_{\ell_{-X}}(0))  \\
+2\pi \left( \int_0^{\sqrt{v}} e^{4\pi m w^2} dw \right) \left( \sum_{n <0} a^+_{\ell_X}(n)e(nRe(c(X))/\alpha) + a^+_{\ell_{-X}}(n) e(nRe(c(-X))/\alpha)\right).
\end{multline*}
\end{lemma}\qed

Combining Lemma~\ref{partialeta-behavior} with Lemma~\ref{partialF-behavior} and using Lemma~\ref{big-current-split-hyperbolic1} completes the proof of Proposition~\ref{key-split-hyperbolic}!

\subsection{The parabolic Fourier coefficient}

Let $f \in H^+_0(\G)$. For $m=0$ we have
\begin{align*}
\int^{reg}_{M} f(z) \theta_{0,h}(\tau,z) d\mu(z)  = \int^{reg}_M f(z) d\mu(z)  +
 \int_M^{reg} \sum_{\substack{X \in L_{0,h} \\ X \ne 0}} f(z) \varphi_0(X,\tau,\g z) d\mu(z),
\end{align*}
where
\[
 \int^{reg}_M f(z) d\mu(z) = \lim_{T \to  \infty} \int_{M_T} f(z) d\mu(z)
\]
as in \cite{BFCrelle} and
\begin{multline}\label{reg-int-0}
\int_M^{reg} f(z)\sum_{\substack{X \in L_{0,h} \\ X \ne 0}} \varphi_0(X,\tau,\g z) d\mu(z)  \\ = \lim_{T \to \infty} \sum_{\ell \in \G \back \Iso(V)} \Big[ \int_{M_T} f(z)\sum_{\g \in \G/\G_{\ell}} \sum_{\substack{X \in \ell \cap L_{0,h} \\ X \ne 0}} \varphi_0(X,\tau,\g z) d\mu(z)  - \frac{a^+_{\ell}(0) \eps_{\ell}b_{\ell}(0,h)}{\sqrt{N}} \log T \Big].
\end{multline}
by \eqref{reg-int-m}. Note that $b_{\ell}(0,h) =1$ if and only if $\bar h =0$. Otherwise $b_{\ell}(0,h)=0$.

For $Q(X)=0$, the function $\eta(X,z)$ and its derivatives have no singularities in $D$ so that the equation $dd^c \eta(X,z) = \varphi_0(X,z)$ holds unequivocally. The following lemma is therefore immediate.

\begin{lemma}\label{current-parabolic}
Let $f \in H_0^+(\G)$. Fix a cusp $\ell$. Then for $T$ sufficiently large we have
\begin{multline*}
\int_{M_T}  f(z)  \sum_{\g \in \G/\G_{\ell}} \sum_{\substack{X \in \ell \cap L_{0,h} \\ x \ne 0}} \varphi_0(X,\tau,\g z) d\mu(z) \\ = \frac{1}{2\pi i} \Big[ \int_{\partial M_{T}} f(z) \sum_{\g \in \G/\G_{\ell}} \sum_{\substack{X \in \ell \cap L_{0,h} \\ X \ne 0}} \partial \eta (X,\tau,\g z)  + \int_{\partial M_{T}} \bar{\partial} f(z) \sum_{\g \in \G/\G_{\ell}} \sum_{\substack{X \in \ell \cap L_{0,h} \\ X \ne 0} }  \eta (X,\tau,\g z) \Big].
 %\\
%& \quad - \frac{1}{2\pi}
%\int_{M_T} \Delta f(z)  \sum_{\g \in \G/\G_{\ell}} \sum_{X \in \ell \cap L_{0,h}} \eta(X,\tau,\g z) d\mu(z).
\end{multline*}

\end{lemma}

As before, the second term on the right hand side in Lemma~\ref{current-parabolic} vanishes in the limit. In view of \eqref{reg-int-0} the following proposition gives the constant coefficient in Theorem~\ref{th:Main-Maass}.

\begin{proposition}
We can write $ \ell \cap (L+h) = \Z \beta_{\ell} u_{\ell} + k_{\ell}u_{\ell}$ for some $0 \leq k_{\ell} < \beta_{\ell}$. Then the asymptotic behavior as $T\to\infty$ of
\[
\frac{1}{2\pi i}  \int_{\partial M_{T}} f(z) \sum_{\g \in \G/\G_{\ell}} \sum_{\substack{X \in \ell \cap L_{0,h} \\ X \ne 0}} \partial \eta (X,\tau,\g z)
\]
is given by
\[
-a^+_{\ell}(0) \frac{\eps_{\ell}}{2\sqrt{N}} \left[  \log(4 \beta_{\ell}^2 \pi v) +\g + \psi(k_{\ell}/\beta_{\ell})+\psi(1-k_{\ell}/\beta_{\ell}) - 2\log T \right].
\]
Here we (formally) set $\psi(0)=-\g$, which is justified since $-\g$ is the constant term of the Laurent expansion of $\psi$ at $0$.
\end{proposition}

\begin{proof}

We have
\[
\sum_{\substack{X \in \ell \cap L_{0,h} \\ X \ne 0}} \partial \eta (X,\tau,z) =  \sideset{}{'}{\sum}_{n=-\infty}^{\infty}
\partial \eta (nu_{\ell}+h_{\ell}),  z).
\]
Here $\sum'$ indicates that we omit $n=0$ in the sum in the case of the trivial coset. We can assume that $\ell$ corresponds to the cusp $\infty$ so that $u_{\ell} = \kzxz{0}{\beta}{0}{0}$ with $\beta= \beta_{\ell}$ and $h_{\ell} = \kzxz{0}{k}{0}{0}$ for some $0\leq k=k_{\ell} < \beta$. We easily see
\[
\frac{1}{2\pi i} \partial \eta (n X_{\ell} + h_{\ell},z) = - \frac{\sqrt{v}}{2y} \left(\int_1^{\infty} e^{-\pi(n\beta+k)^2 vNt/y^2} \frac{dt}{\sqrt{t}} \right) dz.
\]
Hence $\sum_{\g \in \G/\G_{\ell}} \sum_{\substack{X \in \ell \cap L_{0,h} \\ X \ne 0}} \partial \eta (X,\tau,\g z)$ is rapidly decaying at all cusps except $\infty$, and for that cusp in the limit all terms in the sum over $\G/\G_{\ell}$ vanish except $\g=1$. We set
\begin{align}
\Omega(s) &=  - \frac{\sqrt{v}}{2y} \sideset{}{'}{\sum}_{n=-\infty}^{\infty}   \int_1^{\infty} e^{-\pi(n\beta+k)^2 vNt/y^2} t^s\frac{dt}{\sqrt{t}} \\
 & = - \frac{\sqrt{v}}{2y} \sideset{}{'}{\sum}_{n=-\infty}^{\infty}   \int_0^{\infty} e^{-\pi(n\beta+k)^2 vNt/y^2} t^s\frac{dt}{\sqrt{t}} + \frac{\sqrt{v}}{2y} \sideset{}{'}{\sum}_{n=-\infty}^{\infty}   \int_0^{1} e^{-\pi(n\beta+k)^2 vNt/y^2} t^s\frac{dt}{\sqrt{t}}. \label{Omega-0}
\end{align}
For the first term in \eqref{Omega-0}, we have
\begin{multline}\label{Omega-01}
 - \frac{1}{2\beta\sqrt{\pi N}} (\pi \beta^2vN/y^2)^{-s} \G(s+1/2)\left(H(2s+1,k/\beta)+ H(2s+1,1-k/\beta) \right) \\
 =\frac{1}{2\beta\sqrt{N}}\left[ \frac{-1}{s} + \log(\pi \beta^2vN/y^2) +2\log 2+\g + \psi(k/\beta)+\psi(1-k/\beta) \right] + O(s).
  \end{multline}
 Here $H(s,w) = \sum_{n=0}^{\infty} (n+w)^{-s}$ is the Hurwitz zeta function, where for $w=0$ we set $H(s,w) = \zeta(s)$. Then $H(s,w)$ has a simple pole at $s=1$ with constant term $-\psi(w)$ in the Laurent expansion. With our convention for $\psi(0)$ above, \eqref{Omega-01} also holds for $k=0$.

 For the second term in \eqref{Omega-0}, we substitute $t \to 1/t$ and apply the theta transformation formula to obtain
\begin{align}
\label{Omega-02}
&\frac{1}{2\beta\sqrt{N}} \int_1^{\infty} \sum_{n \in \Z} e^{2\pi i nk/\beta} e^{-\pi y^2n^2t/\beta^2vN} t^{-s} \frac{dt}{t} - \delta_{k,0} \frac{\sqrt{v}}{y} \int_1^{\infty} t^{-s-3/2} dt \\
\nonumber
&= \frac{1}{2\beta\sqrt{N}}\frac{1}{s} + g(y) + O(s),
\end{align}
for a function $g$ with $\lim_{y \to \infty} g(y) =0$. Combining \eqref{Omega-01} and \eqref{Omega-02} gives an expression for $\Omega(0) dz = \frac{1}{2\pi i} \partial \eta (n X_{\ell} + h_{\ell},z)$ which we can easily integrate over $\partial M_{T,\ell}$ to obtain the lemma.
\end{proof}

\end{document}